\documentclass[12pt]{article}

\usepackage{amscd}     
\usepackage{amsmath}   
\usepackage{amssymb}   
\usepackage{amsthm}    
\usepackage{xr-hyper}
\usepackage{hyperref}  
\usepackage{bookmark}  
\usepackage{titlesec}
\usepackage[title]{appendix} 
\usepackage{floatrow}
\usepackage{mathtools} 
\usepackage{verbatim}  
\usepackage{graphicx}  
\usepackage{color}     
\usepackage{soul}
\usepackage{cancel}

\usepackage{scalerel,stackengine}
\usepackage[rightcaption]{sidecap}
\usepackage[skip=2pt,font={small, it}]{caption} 

\usepackage{algorithm} 
\usepackage{algorithmicx}
\usepackage{multirow}

\usepackage{algpseudocode}
\usepackage[normalem]{ulem}

\usepackage[UKenglish]{babel}
\usepackage{authblk}
\usepackage{adjustbox}
\usepackage{geometry}
\usepackage{tikz}
	\usetikzlibrary{shapes,arrows}
	\usetikzlibrary{positioning}
	\usetikzlibrary{calc,shapes, positioning}
	\usetikzlibrary{patterns}
	
\DeclareMathAlphabet{\mathpzc}{OT1}{pzc}{m}{it}

\newtheorem{definition}{Definition}

    \newtheorem{Theorem}{Theorem}[section]
\newtheorem{Lemma}[Theorem]{Lemma}

\newtheorem{corollary}[Theorem]{Corollary}

\theoremstyle{remark}
\newtheorem{remark}[Theorem]{Remark}

\newcommand{\RR}{\mathbb{R}}

\newcommand{\EE}{\mathbb{E}}
\newcommand{\MM}{\mathcal{M}}
\newcommand{\DD}{\mathcal{D}}

\newcommand{\OO}{\mathcal{O}}
\newcommand{\UU}{\mathcal{U}}
\newcommand{\WW}{\mathcal{W}}
\newcommand{\KK}{\mathcal{K}}
\newcommand{\LL}{\mathcal{L}}

\newcommand{\C}{\mathcal{C}}
\newcommand{\HH}{\mathcal{H}}

\newcommand{\PP}{\mathcal{P}}

\newcommand{\eps}{\varepsilon}
\newcommand{\defeq}{\stackrel{\bigtriangleup}{=}}

\newcommand{\ceil}[1]{\left\lceil #1 \right\rceil}
\newcommand{\abs}[1]{\left\vert #1 \right\vert}
\newcommand{\norm}[1]{\left\Vert #1 \right\Vert}

\newcommand{\dist}{\operatorname{dist}}
\newcommand{\wtilde}[1]{\widetilde{#1}}

\newcommand{\lrangle}[1]{\langle #1 \rangle}
\newcommand{\lrbrackets}[1]{\left( #1 \right)}
\newcommand{\emphbrackets}[1]{\emph{(}#1\emph{)}}
\newcommand{\maxangle}{\angle_{\max}}

\newcommand{\stau}{\tau}

\newcommand{\sM}{{M}}
\DeclareMathOperator*{\argmin}{\arg\!\min}

\DeclareMathOperator{\vol}{Vol}
\DeclareMathOperator{\Ima}{Im}
\DeclareMathOperator{\cov}{cov}

\newrobustcmd*{\mysquare}[1]{\tikz{\filldraw[draw=#1,fill=#1] (0,0)
rectangle (0.2cm,0.2cm);}}

\newrobustcmd*{\mycircle}[1]{\tikz{\filldraw[draw=#1,fill=#1] (0,0) circle [radius=0.1cm];}}

\usepackage{enumitem}
\setlistdepth{9}

\newlist{myEnum}{enumerate}{9}
\setlist[myEnum,1]{label=\arabic*.}
\setlist[myEnum,2]{label=(\alph*)}
\setlist[myEnum,3]{label=\roman*.}
\setlist[myEnum,4]{label=\Alph*.}
\setlist[myEnum,5]{label=(\arabic*)}
\setlist[myEnum,6]{label=(\Roman*)}
\setlist[myEnum,7]{label=(\Alph*)}
\setlist[myEnum,8]{label=(\roman*)}
\setlist[myEnum,9]{label=(\arabic*)}

\definecolor{myPurple}{rgb}{0.4940, 0.1840, 0.5560}
\definecolor{myGreen}{rgb}{0, 0.5, 0}

\allowdisplaybreaks

\tikzstyle{decision} = [diamond, aspect=2, draw, fill=yellow!0,
text width=7em, text badly centered, node distance=3cm, inner sep=0pt]
\tikzstyle{block} = [rectangle, draw, fill=blue!20,
text width=5em, text centered, rounded corners, minimum height=2em]
\tikzstyle{wblock} = [rectangle, draw, fill=blue!20,
text width=7em, text centered, rounded corners, minimum height=2em]
\tikzstyle{line} = [draw, -latex']
\tikzstyle{bbblock} = [rectangle, draw, fill=blue!20,
text width=10.5em, text centered, rounded corners, minimum height=2em]
\tikzstyle{xbblock} = [rectangle, draw, fill=blue!20,
text centered, rounded corners, minimum height=2em]
\tikzstyle{wwblock} = [rectangle, draw, fill=blue!20,
text width=9.5em, text centered, rounded corners, minimum height=2em]
\tikzstyle{Iwwblock} = [rectangle, draw, fill=blue!50,
text width=9.5em, text centered, rounded corners, minimum height=2em]
\tikzstyle{line} = [draw, -latex']

\tikzstyle{Iblock} = [rectangle, draw, fill=blue!40,
text width=5em, text centered, rounded corners, minimum height=2em]

\tikzstyle{Iwblock} = [rectangle, draw, fill=blue!40,
text width=7em, text centered, rounded corners, minimum height=2em]

\tikzstyle{Rblock} = [rectangle, draw, fill=red!20,
text width=5em, text centered, rounded corners, minimum height=2em]

\tikzstyle{Rbbblock} = [rectangle, draw, fill=red!20,
text width=10em, text centered, rounded corners, minimum height=2em]

\newcommand{\dd}{2cm}

\tikzstyle{innerWhite} = [-, white,line width=1mm, shorten >= 4.5pt]
\tikzstyle{innerWhiteEq} = [-, white,line width=1mm, shorten >= 0pt]
\tikzstyle{imp} = [-{implies}, double, double distance=1.1mm, line width=0.3mm]
\tikzstyle{Rimp} = [-{implies}, red,double, double distance=1.1mm, line width=0.3mm]
\tikzstyle{Gimp} = [-{implies}, green,double, double distance=1.1mm, line width=0.3mm]
\tikzstyle{Req} = [=, red,double, double distance=1.1mm, line width=0.3mm]
\tikzstyle{eq} = [=,  double distance=1.1mm, line width=0.4mm]
\tikzstyle{Geq} = [=, green, double distance=1.1mm, line width=0.4mm]
\usetikzlibrary{arrows, shapes.gates.logic.US, calc}

\graphicspath{{./Figures/}}

\newcommand{\blind}{0}

\begin{document}
\title{Estimation of Local Geometric Structure on Manifolds from Noisy Data}
\author[1]{Yariv Aizenbud\footnote{authors contributed equally}}
\author[2]{Barak Sober$^*$}
\affil[1]{Department of Mathematics, Tel Aviv University}
\affil[2]{Department of Statistics and Data Science, Center of Digital Humanities, The Hebrew University of Jerusalem}

\maketitle

\begin{abstract}
A common observation in data-driven applications is that high-dimensional data have a low intrinsic dimension, at least locally. 
In this work, we consider the problem of point estimation for manifold-valued data.
Namely, given a finite set of noisy samples of $\mathcal{M}$, a $d$ dimensional submanifold of $\mathbb{R}^D$, and a point $r$ near the manifold we aim to project $r$ onto the manifold.
Assuming that the data was sampled uniformly from a tubular neighborhood of a $k$-times smooth boundaryless and compact manifold, we present an algorithm that takes $r$ from this neighborhood and outputs $\hat p_n\in \mathbb{R}^D$, and $\widehat{T_{\hat p_n}\mathcal{M}}$ an element in the Grassmannian $Gr(d, D)$.
We prove that as the number of samples $n\to\infty$, the point $\hat p_n$ converges to $\mathbf{p}\in \mathcal{M}$, the projection of $r$ onto $\MM$, and $\widehat{T_{\hat p_n}\mathcal{M}}$ converges to $T_{\mathbf{p}}\mathcal{M}$ (the tangent space at that point) with high probability.  
Furthermore, we show that $\hat p_n$ approaches the manifold with an asymptotic rate of $n^{-\frac{k}{2k + d}}$, and that $\hat p_n, \widehat{T_{\hat p_n}\MM}$ approach $\mathbf{p}$ and $T_{\mathbf{p}}\MM$ correspondingly with asymptotic rates of $n^{-\frac{k-1}{2k + d}}$.
\end{abstract}
\smallskip

\section{Introduction}
Differentiable manifolds are an indispensable language in modern physics and mathematics. 
As such, there is a plethora of analytic tools designed to investigate manifold-based models (e.g., connections, differential forms, curvature tensors, parallel transport, bundles). 
In order to facilitate these tools, one normally assumes access to a manifold's atlas of charts (i.e., local parametrizations). 
Over the past few decades, manifold-based modeling has permeated into data analysis as well (e.g., see  \cite{hastie1984principal,roweis2000LLE, scholkopf1998KPCA}), usually to avoid working in high dimensions, due to the known curse of dimensionality \cite{stone1980optimal}. 
However, in these data-driven models, charts are not accessible and the only information at hand are the samples themselves. 
As a result, a common practice in Manifold Learning is to embed the data in a lower dimensional Euclidean domain, while maintaining some notion of distance (e.g., geodesic or diffusion).
Subsequently, the embedded data is being processed using linear methods on the low dimensional domain; to mention just a few of this body of literature, see  \cite{belkin2003laplacian,belkin2006manifold,coifman2006diffusion,saul2003LLE}. Some of the approaches have robustness guarantees~\cite{ding2020phase,dunson2021spectral,el2016graph,shen2020scalability}.
The main drawback of such dimensionality reduction approaches is that they inevitably lose some of the information in the process of data projection.

In recent years there have been a growing interest in the problem of manifold estimation.
The aim of these approaches is to reconstruct an underlying manifold $\widehat{\MM}$, approximating the sampled one $\MM$, based upon a given discrete sample set.
The first attempt (neglecting the literature dealing with approximations of curves and surfaces \cite{dey2006curve,wendland2004scattered}) at this problem was probably made by Cheng et. al. at 2005 \cite{cheng2005manifold} who present an algorithm that outputs a simplicial complex, homeomorphic to the original manifold $\MM$ and is proven to be close to it in the Hausdorff sense.
However, this algorithm is deemed intractable by the authors as it is based on the creation of Delaunay complexes through Voronoi diagrams in the ambient space. 
Harvesting the idea of tangential Delaunay complexes, Boissonat and Ghosh \cite{boissonnat2014manifold} have provided a method reconstructing a simplicial complex which is computationally tractable (i.e., its complexity has linear dependency in the ambient dimension).
In this approach there is an underlying assumption that the local tangent at each sampled point is given, and they recommend using a local Principal component analysis (PCA) to find these tangents.
The choice of local PCA as an approximating tangent is shown to be a valid one in the analysis given in \cite{aamari2018stability,kaslovsky2011optimal,kaslovsky2014non,singer2012vector}. Furthermore, \cite{aamari2018stability} shows that the estimate given by Boissonat and Ghosh~\cite{boissonnat2014manifold} achieves optimal minmax rates of convergence in case of noiseless samples, with respect to a certain class of manifolds.
Some other works aims at learning multiscale dictionaries to describe the manifold data \cite{allard2012multi}. 

In parallel, meshless methods for the reconstruction of manifolds from point sets have been developed.
Niyogi et. al. \cite{niyogi2008finding} present such an approach through a union of $\varepsilon$-balls around the samples.
They show that this approximant can be homologous to $\MM$ under certain conditions. 
This approach is somewhat similar to the one proposed by Fefferman et. al. in \cite{fefferman2018fitting}, where an analysis of convergence under a Gaussian noise model is given as well. 
Furthermore, Fefferman et. al. \cite{fefferman2019fitting} 
propose another meshless way of approximating manifolds from point sets up to arbitrarily small Hausdorff distance in the noiseless case, and such that the approximant itself is a smooth manifold of the same intrinsic dimension as $\MM$ (\cite{mohammed2017manifold} uses this framework to provide two more algorithms of such properties).
Faigenbaum-Golovin and Levin uses a generalization of the $L_1$ median to approximate manifolds from meshless data \cite{faigenbaum2020manifold,faigenbaum2025mind}.
Sober and Levin \cite{sober2016MMLS} give an approximation scheme based upon a generalization of the the Moving Least-Squares (MLS) \cite{levin2004mesh,mclain1976two} that provides a smooth manifold with optimal convergence rates in the noiseless case as well (this approach is referred below as the Manifold-MLS).
Their approximation is built through a two-stage procedure, first estimating a local coordinate system and then building a local polynomial regression over it. 
This framework is extended to deal with approximations of functions over manifolds \cite{sober2017approximation} as well as geodesic distances \cite{sober2020Geodesics}.
Aamari and Levrard \cite{aamari2019nonasymptotic} provide a different algorithm, which is shown to be optimal in the noiseless case as well. 
Differently from Sober and Levin's approach, this algorithm estimates a tangent along with a polynomial estimation above the tangent domain at once.
However, in their practical implementation Aamari and Levrard propose a two step solution (first perform PCA to achieve a tangent and then a polynomial regression above it). Note, that although there are results regarding the convergence of local PCA to the tangent space of some manifold these works assume that the localization is around a point on the manifold itself, which is not given in the current problem setting.

Upper bounds on the minimax rates of convergence were first introduced for smooth manifolds by Genovese et. al. \cite{genovese2012minimax,genovese2012manifold}. Later, in \cite{kim2015tight}, the same rates were shown to be optimal.
These results were later refined to a class of H\"older-like smooth manifolds by Aamari and Levrard \cite{aamari2019nonasymptotic}.
They come to the conclusion that the optimal rate of convergence for such $k$-times smooth manifold estimation is $\mathcal O(n^{-k/d})$ for the noiseless case and is bounded from below by $\wtilde{\mathcal{O}}(n^{-k/(k+d)})$ in an additive orthogonal noise model, where $d$ is the intrinsic dimension of $\MM$, $n$ is the number of samples, $\sigma$ is a bound on the noise level, and $\wtilde{\mathcal{O}}$ neglects $\log$ factors in the order.
They do not show that their bound in the noisy case is achievable. 
However, in the case $k=2$, Genovese et. al. showed that this rate is indeed optimal.
It is worth noting here that in the work of Stone from 1982 \cite{stone1982optimal} (quintessential in the analysis we perform below), he shows that the optimal rates of convergence for function estimation are $\wtilde{\mathcal O}(n^{-k/(k+d)})$  which is slower than the result of Genovese et. al. for manifold estimation in the case of $k=2$.
Despite the different noise models of both analyses, we admit that we find this difference surprising, as the a priori case of manifold estimation seems like a harder problem.
Another recent work in this area \cite{lim2021tangent} takes a different approach and gives the sample complexity required to estimate tangent spaces and intrinsic dimensions of manifolds.

Many manifold reconstruction algorithms in the literature show convergence to the underlying manifold under noise assumptions when the noise level decays to zero as the sample size $n$ tends to infinity \cite{aamari2019nonasymptotic,puchkin2022structure}. 
In their recent work, Fefferman, Ivanov, Lassas, and Narayanan presented an algorithm that fits an entire manifold to noisy samples.
They show that the algorithm converges to the underlying manifold in the Hausdorff sense in the presence of large noise \cite{fefferman2023fitting}.
The case they analyze restricts the set of manifolds that can be reconstructed to $R$-exposable manifolds, and they achieve $\mathcal{O}(\log(n)^{-1/2})$ convergence rates.

In the current paper, we assume a sample of size $n$ drawn from the uniform distribution on $\MM_\sigma$, a $\sigma$-tubular neighborhood of the manifold $\MM$.
Then, for a given $ r  \in \MM_\sigma $, we present an algorithm that outputs a point $ \hat p_n\in \RR^D$ and $ \widehat{T_{\hat p_n}\MM}\in Gr(d, D) $, a $d$-dimensional linear subspace of $\RR^D$, which estimate $ \mathbf{p}= Proj_{\MM}(r)\in \MM $, the projection of $r$ onto $\MM$, and $ T_{ \mathbf{p}}\MM $, the subspace tangent to $\MM$ at $\mathbf{p}$.
We prove, in Theorem \ref{thm:MainResult}, that with high probability $dist(\hat p_n, \MM) \leq \wtilde{ \mathcal{O}}(n^{-k/(2k+d)})$,
$ \norm{\hat p_n -  \mathbf{p}} = \wtilde{\mathcal{O}}(n^{-(k-1)/(2k+d)}) $ 
and that $ \maxangle(\widehat{T_{\hat p_n}\MM}, T_{ \mathbf{p}}\MM) = \wtilde{\mathcal{O}}(n^{-(k-1)/(2k+d)}) $, 
where $ \wtilde{\mathcal{O}} $ neglects dynamics weaker than polynomial order (e.g., $ \ln(n) $ and $ \ln(\ln(n)) $).
These achieved convergence rates are similar to the optimal rates of non-parametric estimation of functions \cite{stone1980optimal}.
To avoid notation inconvenience, we assume in all of our proofs that the noise level $\sigma > 0$.
Our approach is based on the Manifold-MLS \cite{sober2016MMLS}, but differs from it, as explained below.

This paper presents an improved algorithm that builds upon our previous work in \cite{aizenbud2021non}. 
The key advancement is that we now demonstrate that the limit point $\mathbf{p}$ as $n\to \infty$ is indeed the projection of $r$ onto $\MM$ -- a critical property we were unable to establish in our earlier analysis. 
To achieve this, we have both simplified the algorithm and refined the theoretical analysis.

The algorithm presented in this paper is divided into two steps. In Step 1, we find an initial local coordinate system. 
It is proved in Theorem \ref{thm:Step1} that this coordinate system is a ``reasonable" approximation to the tangent of the manifold at some point. 
Next, in step 2, we improve the estimation of step 1 in an iterative manner (sharing some similarity in spirit with the approach presented in \cite{puchkin2022structure}) to get an accurate estimation of the clean point on the manifold (i.e., the projection or $r$ onto the manifold) and its tangent.  
We prove, in Theorem \ref{thm:Step2}, that these iterations indeed converge to an accurate estimate, and show the convergence rates mentioned above. The formal problem setting, along with the algorithm's description are presented in Section \ref{sec:Problem_Setting}. In Section \ref{sec:mainResults} the formal results are presented, where Theorem \ref{thm:MainResult} is the main result of this paper. The theorems of Section \ref{sec:mainResults} are proved in Section \ref{sec:Proofs}. Finally, in Section \ref{sec:applications} we present one possible application of the presented method. Although there are many possible applications (e.g.  denoising, trajectory reconstruction, etc.) , we chose one which can easily be demonstrated visually. The code for the algorithm in this paper, along with examples, can be found in  \ifthenelse{\blind=1}{\textcolor{blue}{\href{https://github.com/aizeny/manapprox}{https://github.com/aizeny/manapprox}}}{the \textcolor{blue}{attached code}}.

\section{Problem Setting and Algorithm Description}
\label{sec:Problem_Setting}
Throughout the paper we limit the discussion to estimation of $\MM\in\C^k$, $k$-times smooth, compact submanifolds of $\RR^D$ without boundary.
This limitation is important for the analysis, but the algorithm we present is local and thus from a practical perspective, a local version of these assumptions should suffice. The smoothness requirement is also crucial for the algorithm. However, it is shown in \cite{kouvrimska2024free} that a non-smooth manifold can be approximated arbitrarily well by a smooth manifold without a significant change to the reach. This implies that the method presented in this paper is useful for non-smooth manifolds as well.
A key concept in the analysis of manifold estimates is the \textit{reach} of a manifold (e.g., see the analyses at \cite{aamari2018stability,aamari2019nonasymptotic,fefferman2018fitting,niyogi2008finding}), which was introduced by Federer in \cite{federer1959curvature} and is defined as the maximal distance for which there exists a unique projection onto $\MM$.
\begin{definition}[Reach \cite{federer1959curvature}]\label{def:reach}
	The reach of a subset $A$ of $\RR^D$, is the largest $\tau$ \emph{(}possibly $\infty$\emph{)} such that for any $x\in\RR^D$ that maintains $dist(A,x)\leq \tau$, there exists a unique point $Proj_A(x)\in A$, nearest to $x$. 
\end{definition}
Using the reach, we can bound both the local behavior ($1/\tau$ bounds all sectional curvatures of the manifold) and the global behavior of a manifold (i.e., it measures how close a manifold can get to itself) \cite{boissonnat2017reach}. 
Thus, the reach provides a good way of expressing our limitations in the problem of manifold estimation (in \cite{niyogi2008finding} the same concept is defined as the \textit{condition number} of a manifold).
For example, if the reach is too small and the sampling density is not fine enough, we would expect that small features could not be recovered.
Moreover, through the reach of a manifold we can define the acceptable levels of noise that do not obscure the geometrical shape.
In accordance with that, we limit our discussion to manifolds with a reach bounded away from zero (notice that in the case of flat manifolds, the reach is infinite) and with a noise model that limits the noise level from above by the reach.

Our noise model in the analysis is as follows: 
We assume that we are given a finite set of samples $\{r_i\}_{i=1}^n$ drawn independently from 
\begin{equation}\label{eq:Msigma}
	\MM_\sigma \defeq \{q\in \RR^D ~|~ dist(q, \MM) < \sigma \}
	,\end{equation}
a tubular neighborhood of $\MM$.
Explicitly, we assume that $r_i \sim \operatorname{Unif}(\MM_\sigma)$, which is the uniform distribution on $\MM_\sigma$; i.e., the normalized Lebesgue measure with respect to $\RR^D$.

Finally, let $p\in\MM$, we wish to describe $W_p \subset \MM$ a neighborhood of $p$ as a graph of a function 
\begin{equation}
	\varphi_p:W_{T_p\MM}\to \MM \label{eq:LocalChart_clean}    
\end{equation}
where $W_{T_p\MM} = Proj_{T_p\MM}(W_p)$, the projection of $W_p$ onto the tangent, and $\varphi_p$ is defined by
\begin{equation}
	\varphi_p(x) = p+(x,\phi_p(x))_{T_p\MM}\label{eq:phi_def_clean}
	,\end{equation}
where $x\in \RR^d \simeq T_p\MM$, $\phi_p(x)\in \RR^{D-d} \simeq T_p\MM^\perp $, and $(x,y)_{T_p\MM} \in \RR^D$ denotes that $x\in\RR^d$ and $y\in\RR^{D-d}$ represent a point in some basis of $T_p\MM$ and $T_p\MM^\perp$ correspondingly.
Then, we define the graph of $\phi_p$ to be
\begin{equation}\label{eq:FunctionGraph_init}
	\Gamma_{\phi_p, W_p} \defeq \{p + (x, \phi_p(x))_{T_p\MM} | x\in W_{T_p\MM} \}  
	.\end{equation}
For simplicity, throughout the paper we identify the graph of $\phi_p$ with $\Gamma_{\phi_p,W_p}$. That is, we refer to $\MM$ as locally a graph of $\phi_p$ (see Figure \ref{fig:SkewedH}).

We would like to stress that throughout the paper there is a slight misuse of notation.
Explicitly, when we refer to $T_p\MM$ and $T_p\MM^\perp$, we sometimes look at it as elements in the Grassmannian $Gr(d,D)$ and $Gr(D-d, D)$; i.e., subspaces of $\RR^D$ with dimensions $d$ and $D-d$ correspondingly.
On the other hand, in some other occasions (as in \eqref{eq:FunctionGraph_init}) we neglect the fact that these are subsets of $\RR^D$ which is equivalent to choosing some basis and working in it.

\begin{SCfigure}
	\centering
	\includegraphics[width=0.6\linewidth]{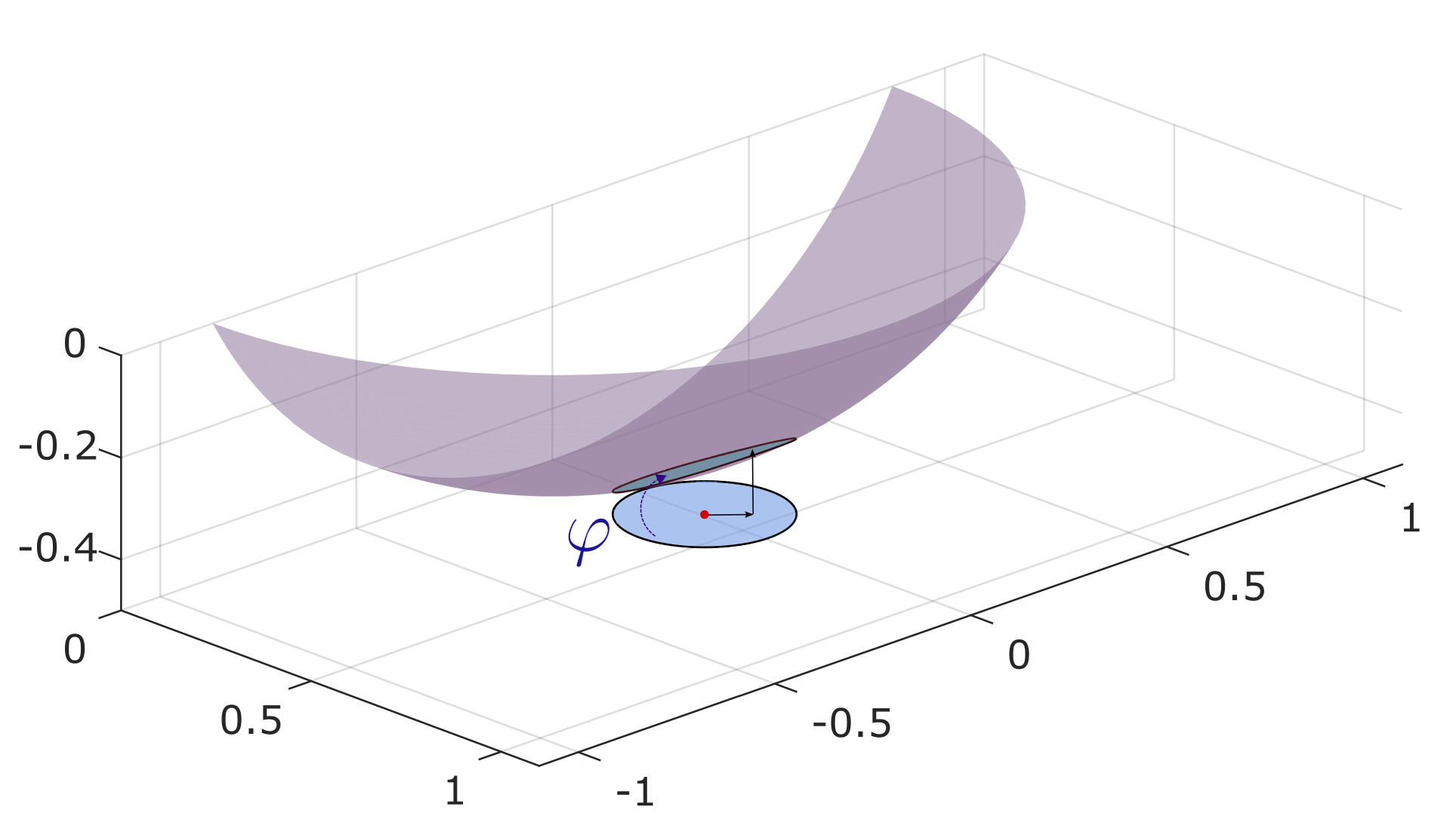}
	\caption{Illustration of $ \varphi_{p} $: The $ xy $-plane is $ H(p) $; The local origin $ q(p) $, which is mapped by $ \phi_{p} $ to $ p $ is marked by the red dot; the vector $ x $ represents a ``tangential" movement; and $ \phi_{p}(x) $ is a normal movement. }
	\label{fig:SkewedH}
\end{SCfigure}

\subsection{Summary of manifold and sampling assumptions}\label{sec:SamplingAssumptions}
Throughout this paper we assume that the (unknown) manifold $\MM$ and the samples $\{r_i\}_{i=1}^n$ satisfy the following requirements: 
\begin{enumerate}
	\item $\MM\in\C^k$ is a compact $d$-dimensional sub-manifold of $\RR^D$ without boundary
	\item $M = \frac{\tau}{\sigma}$ is large enough, where $\tau$ is the reach of $\MM$ and $\sigma>0$ is the noise level. 
	\label{assume:noise}
	\item $\{r_i\}_{i=1}^n$ are samples drawn independently and uniformly from $\MM_\sigma$ (i.e., $r_i\sim \operatorname{Unif}(\MM_\sigma)$).\label{assume:sampling}
\end{enumerate}

\subsection{Algorithm Description}\label{sec:Algorithm}
As explained above, given a point $r\in\MM_\sigma$ we aim at providing a procedure $\PP(r)$ that will estimate $ p = Proj_\MM(r) \in \MM$, the projection of $r$ on $\MM$. 
This is performed through an altered version of the Manifold-MLS that was introduced in \cite{sober2016MMLS}. 
The Manifold-MLS is constructed through a two-step procedure. 
First, an estimate of a local coordinate system is computed.
Second, above this local coordinate system, a local polynomial regression is performed, by which we derive the estimate for the projection onto $\MM$ as well as the tangent domain $T_p\MM$.
Below, we show that the first step of the Manifold-MLS yields a reasonable estimate of the tangent even in the presence of noise.
However, performing the second step, which is just a local polynomial regression, above the slightly tilted domain results in a biased estimate.
That is, if we try to estimate the manifold locally as a function over this approximated domain, the noise in the sample is biased with respect to this coordinate system (see Figure \ref{fig:TiltedBias}).
To account for the bias, in our altered version of the algorithm, we perform the second step iteratively, taking the tangent estimate at each iteration as an improved coordinate system.
We show that in the limit, as the number of samples $n$ approaches $\infty$, our estimate projects $r$ onto $\MM$, and the estimated tangent coincides with the tangent at that projected point.

Before describing the steps in detail, we mention that from a practical point of view, there are some minor adaptations that we made to the implemented version. 
These changes are described in detail in Section \ref{sec:practical}.

\begin{figure}
	\centering
	\includegraphics[width=0.6\textwidth]{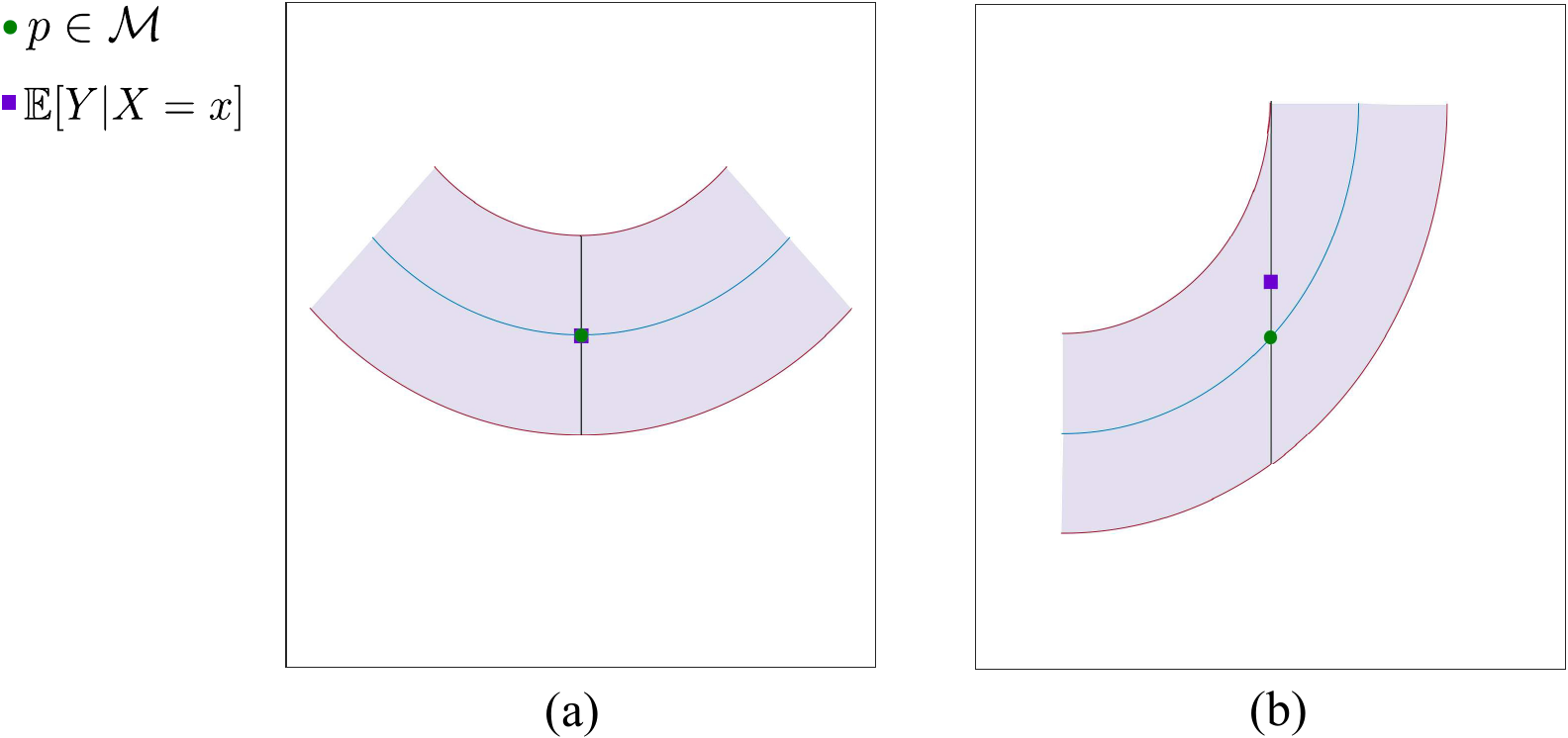}
	\caption{Illustration of a manifold $\MM$ (marked by the blue line) along with its tubular neighborhood $\MM_\sigma$. Assuming uniform sampling in $\MM_\sigma$ we mark the point $p\in\MM$ by \mycircle{myGreen} and the expected value with respect to the given coordinate system by \mysquare{myPurple}. \emph{(}a\emph{)} The coordinate system is aligned with the tangent. \emph{(}b\emph{)} The coordinate system is tilted with respect to the tangent. As can be seen, in (a) the two points coincide, whereas in (b) the expected value differ from the point we wish to estimate.}
	\label{fig:TiltedBias}
\end{figure}

\subsubsection{Step 1 - The Initial Coordinate System}
Given a point $r\in\MM_\sigma$ we limit the the region of interest (ROI) to:
\begin{equation}\label{eq:ROI_clean}
	U_{\textrm{ROI}} = {\{r_i|\dist(r_i,r)<\sqrt{\sigma\tau}\}}
	,\end{equation}
and denote the number of samples in the ROI by $N$.

Intuitively, the ROI has to be significantly smaller than $\tau$ so that we will not include samples of faraway regions in the Riemannian sense.
We also require that the ROI is not a subset of $\MM_\sigma$ as this causes the tangential directionality of the sample to become ambiguous. 

Then, we define the relevant coordinate system as the pair $(q, H)\in \RR^D\times Gr(d, D)$, which minimizes the functional:
\begin{equation}\label{eq:J1_clean}
	J_1(r; q, H) = \frac{1}{N}\sum_{r_i\in U_{\textrm{ROI}}} \dist^2 (r_i - q,H)
\end{equation}
under the constraints
\begin{enumerate}
	\item Orthogonality: $ r - q\perp H $.
	\item Region of interest: $ r_i \in U_{\textrm{ROI}} $.
	\item Search region: $\norm{r-q} < 2\sigma$.
\end{enumerate}
Explicitly, we denote
\begin{equation}\label{eq:argmin1_clean}
	q^*(r), H^*(r) = \argmin_{\substack{ q, H \\ r-q \perp H \\ \|r- q\|<2\sigma}} J_1(r; q, H)
\end{equation}

Note, that Constraint 2 limits our region of interest in accordance with sampling assumptions \ref{assume:sampling} and \ref{assume:noise}.
Furthermore, since $r\in\MM_\sigma$ we know that the true projection onto the manifold is in the search region defined in Constraint 3 (if $q=Proj_\MM(r)$, the projection of $r$ onto $\MM$ and $H= T_{q}\MM$).
Finally, Constraint 1 extends the notion of orthogonal projection onto manifolds.
As discussed in \cite{sober2016MMLS}, this constraint is responsible for having a unique minimizer for \eqref{eq:argmin1_clean} given enough samples. 
The aforementioned minimization problem is summarized in Algorithm \ref{alg:alg1_clean}.

\begin{algorithm}[H]
	\caption{Step 1: Find an initial coordinate system}
	\label{alg:alg1_clean}
	\begin{algorithmic}[1]
		\Statex {\bfseries Input:}\begin{tabular}[t]{ll}
			$\{r_i\}_{i=1}^n\subset \MM_\sigma$ & noisy samples of a $d$ dimensional manifold $ \MM $.\\
			$r\in \MM_\sigma$ & a point in a tubular neighborhood of $\MM$.\\
            $d$ & The dimension of the manifold $\MM$.\\
            $\tau$ & lower bound on the reach of $\MM$. \\
            $\sigma$ & upper bound on the noise level.\\
            
		\end{tabular}
		\Statex {\bfseries Output:}\begin{tabular}[t]{ll}
			$q^*$ & crude estimation of $p = Proj_\MM(r)$.\\
			$H^*$ & estimation of $T_p\MM$.\\
		\end{tabular}
		\State Disregard points outside of $U_{\textrm{ROI}}$ (plugging $\tau_\text{bound}, \sigma_\text{bound}$ into Eq. \eqref{eq:ROI_clean}).
		\State Find $q^*, H^*$ minimizing \eqref{eq:argmin1_clean} subject to $\norm{q^* - r}<2\sigma$ and $r - q^* \perp H^*$. 
	\end{algorithmic}
\end{algorithm}

\subsubsection{Step 2 - The iterated projection}
Given $(q, H)\in \RR^D\times Gr(d,D)$ we define the following minimization scheme, known as local polynomial regression (e.g., \cite{cleveland1979robust,mclain1976two}):
Find $\pi\in \Pi_{k-1}^{d\mapsto D}$ a polynomial of total degree $deg(\pi)\leq k-1$ from $\RR^d$ to $\RR^{D-d}$ which minimizes
\begin{equation}\label{eq:Step2_clean}
	J_2(\pi ~|~ q, H) = \frac{1}{N_{q, H}}\sum_{r_i \in U_{\textrm{ROI}}^n} \norm{r_i - (x_i, \pi(x_i))_H}^2
	,\end{equation}
where $x_i\in \RR^d$ are the projections of $r_i - q$ onto $H$, $(x, y)_H\in \RR^d\times\RR^{D-d}$ are coefficients in a basis of $H\times H^\perp$, $U_{\textrm{ROI}}^n(q,H)$ is defined through a bandwidth $\epsilon_n$ as
\begin{equation}\label{eq:ROI_ell_clean}
	U_{\textrm{ROI}}^n(q,H) = {\{r_i\in U_\textrm{ROI}~|~\norm{x_i}<\epsilon_n\}}
	,\end{equation}
and $N_{q, H}$ denotes the number of samples in $U_{\textrm{ROI}}^n(q, H)$.
Explicitly, the local polynomial regression is defined through
\begin{equation}\label{eq:argmin2_clean}
	\pi^*_{q, H} = \argmin_{\pi\in \Pi_{k-1}^{d\mapsto D}} J_2(\pi ~|~ q, H)
	.\end{equation}
As required to ensure convergence in probability for local polynomial regression \cite{aizenbud2021VectorEstimation,stone1980optimal}, we demand that the bandwidth $\epsilon_n\to 0$ as $n\to\infty$ is such that 
\begin{equation}\label{eq:bandwidthDecay_clean}
	0<\lim_{n\rightarrow\infty}n^{1/(2k+d)}\cdot \epsilon_n < \infty
	.\end{equation}

Unfortunately, the minimization problem stated in Equation \eqref{eq:argmin2_clean} could not be used as-is, since we need an explicit relation between the failure probability, the error bound, and the number of samples. 
The analysis of convergence of local polynomial regression of vector-valued functions in Theorem 3.2 of \cite{aizenbud2021VectorEstimation}, uses a variant of local polynomial regression, namely some sort of ``median trick"~\cite{alon1999space} on \eqref{eq:argmin2_clean} to unveil such an explicit connection. 
For simplicity of notations, we abuse the definition of $ \pi^*_{q, H}$ in \eqref{eq:argmin2_clean}, and define
\begin{equation}\label{eq:pi_star_median_trick}
	\pi^*_{q, H} = \mbox{Algorithm 2 of \cite{aizenbud2021VectorEstimation}} (x_i,y_i),
\end{equation}
where $y_i$ are the projection of $r_i-q$ onto $H^\perp$. 
Any derivative of $\pi^*_{q, H}$ can also be estimated by means of Algorithm 2 of~\cite{aizenbud2021VectorEstimation}. For simplicity of presentation, throughout the paper when we write $\partial \pi^*_{q, H}$ or $\DD_{\pi^*_{q_\ell,H_{\ell}}}$ what we actually mean is the estimate of the derivative rather than the derivative of $\pi^*_{q, H}$.

We begin with $q^*$ (a coarse approximation of a point on the manifold) and $H_0 = H^*$ (the initial tangent estimate) resulting from Algorithm \ref{alg:alg1_clean}.
Then, we update iteratively the directions of $H$ with respect to a tangent of $\MM$.
Explicitly, in iteration $\ell$, we define $H_{\ell+1}$ to be the subspace coinciding with $ \Ima(\DD_{\pi^*_{r, H_{\ell}}}[0]) $ the image of the differential of $(\textrm{I}_d,\pi^*_{r,H_{\ell}}):\RR^d\to\RR^D$ at $ 0 $ (i.e., the tangent to the graph of $ \pi^*_{r, H_{\ell }} $).
In other words, we look at the manifold as a local graph of a function 
\begin{equation}
	f_{\ell}:\RR^d\simeq H_\ell \to \RR^{D-d} \simeq H_\ell ^\perp 
	.\end{equation}
That is, we define a local patch of the manifold through the graph 
\begin{equation}\label{eq:Graph_fl}
	\Gamma_{f_\ell, r, H_\ell} = \{r + (x, f_\ell(x))_{H_\ell} | x\in B_{H_\ell}(0,\rho) \}\subset\MM,   
\end{equation}
where $\rho$ is some radius where this function is defined (see Lemma \ref{lem:M_is_locally_a_fuinction_clean} for more details regarding the existence of such $\rho > 0$). 
Then, we estimate the first order differential of $f_\ell$ through taking the differential of the local polynomial regression estimate $\pi^*_{r,H_{\ell }}$.
The image of the differential determines $d$-directions in $\RR^D$ (i.e., an element in the Grassmannian $Gr(d,D)$), by which we define $H_{\ell+1}$. 
Following this, we define 
\begin{equation}\label{eq:fl_def}
	f_{\ell+1}:(r, H_{\ell+1}) \to H_{\ell+1} ^\perp \simeq \RR^{D-d}    
\end{equation}

We note that the estimate for the first derivative in case of scalar valued functions was analyzed by \cite{stone1980optimal} (as well as others) and was shown to converge to the true derivative with optimal rates in case of unbiased noise. The results of \cite{stone1980optimal} are generalized to vector valued functions in \cite{aizenbud2021VectorEstimation}. 
A core assumption in these results is that $\EE(Y | X=x)$, the expected value of the samples, aligns with the estimated function.
However, in our case this assumption does not hold, since the noise model is tubular with respect to the manifold and unless the coordinate system is aligned with the tangent, the expected value of $Y$ given $X=x$ does not equal to $f_\ell(x)$ (see Figure \ref{fig:TiltedBias}); i.e., the samples have bias.
We show below that the iterations described above improves the maximal angle with respect to a true tangent. 
Thus, eliminating the problem of bias iteratively. 

Finally, after performing $\kappa$ iterations we get the estimate for $ p\in \MM$ and $ T_{ p}\MM $ by 
\begin{equation}
	\hat p_n \defeq (0, \pi^*_{r, H_\kappa}(0))_{r, H_\kappa} ,\quad \widehat{T_{\hat p_n}\MM} \defeq H_\kappa
	.\end{equation}
Below, we show for a specific value of $\kappa$ that with probability tending to 1 (as the number of samples tend to $\infty$), $dist(\hat p_n,\MM) = \wtilde{\mathcal{O}}(n^{-k/(2k+d)})$, $ \norm{\hat p_n -  p} =  \wtilde{\mathcal{O}}(n^{-(k-1)/(2k+d)}) $ and that $ \maxangle(\widehat{T_{\hat p_n}\MM}, T_{ p}\MM) = \wtilde{\mathcal{O}}(n^{-(k-1)/(2k+d)}) $, where $ \wtilde{\mathcal{O}} $ neglects dynamics weaker than polynomial order (e.g., $ \ln(n) $ and $ \ln(\ln(n)) $). The aforementioned algorithm is summarized in Algorithm \ref{alg:step2_clean}.

\begin{algorithm}[H]
	\caption{Step 2: estimating the manifold from a good initial guess}
	\label{alg:step2_clean}
	\begin{algorithmic}[1]
		\State {\bfseries Input:}\begin{tabular}[t]{ll}
			$\{r_i\}_{i=1}^n\subset \MM_\sigma$ & Noisy samples of a $d$ dimensional manifold $ \MM $.\\
            $d$ & The dimension of the manifold $\MM$.\\
			$r\in \MM_\sigma$ & a point in the tubular neighborhood of $\MM$\\
			$H_0$ & Initial approx. of the tangent
		\end{tabular}
		\State {\bfseries Output:}\begin{tabular}[t]{ll}
			$\hat p_n$ & Estimation of $Proj_{\MM}(r)\in \MM$ .\\
			$\widehat{T_{\hat p_n}\MM}$ & Estimation of   $T_{Proj_{\MM}(r)}\MM$.\\
		\end{tabular}
		
		\For{$\ell=0$ to $\kappa-1$}
		\State Compute  $\pi^*_{r,H_{\ell}}$ through a version of linear least-squares minimization  \eqref{eq:pi_star_median_trick}.
		\State Compute $ \DD_{\pi^*_{r,H_{\ell}}}[0] $ the differential of  $\pi^*_{r,H_{\ell}}$ in some basis at zero.
		\State $H_{\ell+1} = \Ima(\DD_{\pi^*_{r,H_{\ell}}}[0]) $ \Comment{this is the column space of $ \DD_{\pi^*_{r,H_{\ell}}}[0] $}.
		\EndFor 
		\State $\hat p_n = r + (0, \pi^*_{r, H_{\kappa}}(0))^T$ 
		\State $\widehat{T_{\hat p_n}\MM} = H_\kappa $
		\State\Return $\hat p_n$ and $\widehat{T_{\hat p_n}\MM}$.
	\end{algorithmic}
\end{algorithm}

\subsection{Practical Considerations of Implementation Details }\label{sec:practical}
The minimization problem portrayed in \eqref{eq:argmin1_clean} is non-linear since we optimize for both $q$ and $H$ at the same time; note that if we fix $q$ this amounts to the Principal Component Analysis (which is also related to the iterated linear least-squares problem motivating our algorithms -- see \cite{aizenbud2019approximating}).
This problem has already been studied in \cite{sober2017approximation,sober2016MMLS} and we recommend using the iterative scheme presented in Algorithm \ref{alg:step1Inpractice_clean} to solve it (which is a slight adaptation of the algorithm proposed in \cite{sober2016MMLS}). We note that as the initial estimation of the tangent (step \ref{algstep:U_init_in_practice} in Algorithm \ref{alg:step1Inpractice_clean}) we use the local PCA, which was utilized in many other works and shown to be of merit \cite{aamari2019nonasymptotic,aizenbud2015OutOfSample,kaslovsky2014non}. However, if one wishes to improve the computational complexity, the initialization step can be done in a randomized manner as well \cite{aizenbud2016SVD,halko2011algorithm}.
Algorithm \ref{alg:step1Inpractice_clean} can be shown to converge in theory to a local minimizer of \eqref{eq:argmin1_clean}. 
As explained at length in \cite{sober2016MMLS} under some conditions this minimization has a unique minimum. 
Furthermore, in practical implementations we experienced very fast convergence to a minimum.

As for the practical implementation of Step 2, we note that the derivatives of $f_{\ell + 1/2}$ identify with those of $f_{\ell +1}$.
Finally, the number of iterations $\kappa$ in Algorithm \ref{alg:step2_clean} can be computed explicitly to obtain the rates of convergence as explained in the proofs of Theorem \ref{thm:MainResult}.
However, for the practical implementation, given a specific sample we suggest to iterate until convergence.
See Algorithm \ref{alg:step2InPractice_clean} for the adapted implementation.

\begin{algorithm}[H]
\caption{Step 1: in practice}
\label{alg:step1Inpractice_clean}
\begin{algorithmic}[1]
\State {\bfseries Input:} $\lbrace r_i \rbrace_{i=1}^N, r, d, \tau, \sigma,\epsilon$
\State{\bfseries Output:}\begin{tabular}[t]{ll}
                         $q$ - an $n$ dimensional vector \\
                         $U$ - an $n\times d$ matrix whose columns are $\lbrace u_j \rbrace_{j=1}^d$ 
                         \end{tabular}
                         \Comment{$H = q + Span\lbrace u_j \rbrace_{j=1}^d$}
\State define $R$ to be an $n\times N$ matrix whose columns are $r_i$
\State initialize $U$ with the first $d$ principal components of the spatially weighted PCA \label{algstep:U_init_in_practice}
\State $q\leftarrow r$
\Repeat
    \State $q_{prev} = q$
    \State $\tilde{R} = (R - repmat(q,1,N))\cdot \Theta$ \Comment{where $\Theta$ is an indicator for points in $U_{\textrm{ROI}}$}
    \State $X_{N\times d} = \tilde{R}^T U$ \Comment{find the representation of $r_i$ in $Col(U)$}
    \State define $\tilde{X}_{N\times (d+1)} = \left[(1,...,1)^T, X\right]$
    \State solve $\tilde{X}^T\tilde{X}\alpha = \tilde{X}^T \tilde{R}^T$ for $\alpha \in M_{(d+1)\times n}$ \Comment{solving the LS minimization of $\tilde{X}\alpha \approx \tilde{R}^T$}
    \State $\tilde{q} = q + \alpha(1,:)^T$
    \State $Q, \hat{R} = qr(\alpha(2:end, :)^T)$ \Comment{where $qr$ denotes the QR decomposition}
    \State $U \leftarrow Q$
    \State $q = \tilde{q} + U U^T (r-\tilde{q})$
\Until {$\|q-q_{\text{prev}}\|<\epsilon$}
\end{algorithmic}
\end{algorithm}

\begin{algorithm}[H]
	\caption{Step 2: in practice}
	\label{alg:step2InPractice_clean}
	\begin{algorithmic}[1]
		\State {\bfseries Input:}\begin{tabular}[t]{ll}
            $\{r_i\}_{i=1}^n$& a set of points in $\MM_\sigma$ \\
            $d$ & The dimension of the manifold $\MM$.\\
            $q_{-1}$ & a crude approximation of $p_r = Proj_\MM(r)$ \\ &(initialized with $q^*$ of \eqref{eq:argmin1_clean})\\
			$H_0$ & Initial approximation of $T_{p}\MM$ s.t. $\angle_{\max}(H_0,T_{p}\MM) < \alpha_0$ \\
            &(initialized with $H^*$ of \eqref{eq:argmin1_clean})\\
		\end{tabular}
		\State {\bfseries Output:}\begin{tabular}[t]{ll}
				$\hat p_n$ & Estimation of some $p\in\MM$.\\
				$\widehat{T_{\hat p_n}\MM}$ & Estimate of  $T_p\MM$.\\
		\end{tabular}
		
		\State Compute $\pi^*_{q_{-1}, H_0}$
		\State $q_0 = q_{-1} + (0,\pi^*_{q_{-1}, H_0}(0))_{H_0}$ 
        \Repeat 
        \State Compute $\pi^*_{q_\ell,H_{\ell}}$ through the least-squares minimization of \eqref{eq:argmin2_clean}.
        \State Compute $\DD_{\pi^*_{q_\ell,H_{\ell}}}[0]$.
        \State Set $H_{\ell + 1} = \Ima(\DD_{\pi^*_{q_\ell,H_{\ell}}}[0])$
        \State Set $q_{\ell+1} = q_\ell + (0, \pi^*_{q_\ell,H_{\ell}}(0))_{H_\ell}$
		\Until {$|q_{\ell} - q_{\ell +1}| \leq \epsilon$}
		\State \Return $\hat p_n = q_{\ell+1}$,  $\widehat{T_{\hat p_n}\MM} = H_{\ell+1} $
	\end{algorithmic}
\end{algorithm}

\section{Main Results} \label{sec:mainResults}
The main result reported in this paper is Theorem \ref{thm:MainResult}.
For convenience, we wish to reiterate the sampling assumptions presented above in Section \ref{sec:SamplingAssumptions}, as they are relevant for all the following theorems.
Namely, we assume that 
\begin{enumerate}
	\item[i] $\MM\in\C^k$ is a compact $d$-dimensional sub-manifold of $\RR^D$ without boundary.
	\item[ii] $M = \frac{\tau}{\sigma}$ is large enough, where $\tau$ is the reach of $\MM$ and $\sigma>0$ is the noise level. 
	\item[iii] $\{r_i\}_{i=1}^n$ are samples drawn independently and uniformly from $\MM_\sigma$ (that is, $r_i\sim \operatorname{Unif}(\MM_\sigma)$).
\end{enumerate}

\begin{Theorem}\label{thm:MainResult}
	Assume $M > C_\stau \sqrt{D\log D}$ for some constant $C_\stau$ independent of $\tau$, and let $r\in \MM_\sigma$.
	Then, for any $\delta>0$ arbitrarily small, there exists $ N $ such that for any number of samples $ n > N $,
	applying Algorithm~\ref{alg:step2_clean} with inputs $q_{-1}, H_0$ being the outputs of Algorithm~\ref{alg:alg1_clean}, and with the number of iterations $\kappa$ specified in Lemma \ref{lem:comute_kappa} and is dependent on $n,d, \delta, k$, we get $ \hat p_n, \widehat{T_{\hat p_n}\MM} $, for which 
	\[
	dist(\hat p_n, \MM) \leq \frac{C_1\ln\left(\frac{1}{\delta}\right)}{n^{r_0}}
	\]
	\[
	\norm{\hat p_n - \mathbf{p}} \leq C_2\ln\left(\frac{1}{\delta}\right) \left(\frac{ n}{\ln\left(\ln(n) \right)^2}\right)^{-r_1} = \wtilde{\mathcal{O}}(n^{-r_1}) = \mathcal{O}(n^{-r_1})
	,\]
	and 
	\[
	\angle_{\max}(\widehat{T_{\hat p_n}\MM},T_{\mathbf{p}}\MM) \leq  C_3\ln\left(\frac{1}{\delta}\right) \left(\frac{ n}{\ln\left(\ln(n) \right)^2}\right)^{-r_1} = \wtilde{\mathcal{O}}(n^{-r_1}) 
	\]
	with probability of at least $ 1 - \delta $, where $ r_0 = \frac{k}{2k +d} $, $ r_1 = \frac{k-1}{2k +d} $ and $\mathbf{p} = Proj_\MM (r)$ is the projection of $r$ onto $\MM$.
\end{Theorem}

While the explicit formula for $\kappa$, the number of iterations needed in Algorithm \ref{alg:step2_clean}, is given in Lemma \ref{lem:comute_kappa}, it is worth mentioning that it depends only logarithmically on the number of samples $n$.

We derive the result of Theorem \ref{thm:MainResult} by showing that Algorithm \ref{alg:alg1_clean} yields a ``reasonable" estimation for the tangent directions (Theorem \ref{thm:Step1}), and the fact that Algorithm \ref{alg:step2_clean} yields estimates that converge to a point and its tangent on the original manifold as $n \to \infty$ (Theorem \ref{thm:Step2}). Consequently, Theorem \ref{thm:MainResult} can be proven directly from Theorems \ref{thm:Step1} and \ref{thm:Step2}.

\begin{Theorem}\label{thm:Step1}
	Let $ (q^*(r), H^*(r)) $, the output of Algorithm \ref{alg:alg1_clean}, and let $p = Proj_\MM(q^*)$. 
	Denote $\alpha = \sqrt{C_\sM/M}$ for some constant $C_\sM$ (independent of $\alpha$ and $M$).
	Then, for any $\delta>0$ arbitrarily small, there exists $ N_{\delta} $ such that for all $ n > N_{\delta} $ 
	\[\angle_{\max}(H^*, T_p\MM)\leq \alpha\]
	with probability of at least $ 1 - \delta $.
	Furthermore, we have
	\[
	\norm{p - q^*} \leq 3 \sigma
	\]
\end{Theorem}

The proof of this theorem can be found in Section \ref{sec:Step1Analysis}.
\begin{Theorem}\label{thm:Step2}
	Assume that $M>C_\stau\sqrt{D\log D}$.  Denote $\mathbf{p} = Proj_\MM (r)$ where r is the input point for Algorithm \ref{alg:step2_clean}.
	Let $(q, H)$ be an initial coordinate system for which $\norm{q - p}\leq 3 \sigma$ and $ \maxangle (H, T_p\MM) \leq \sqrt{c_\sM/M}$, where $p= Proj_\MM(q) \in \MM$.
	For any $\delta>0$ arbitrarily small, denote by $ \hat p_n, \widehat{T_{\hat p_n}\MM} $ the estimates derived from Algorithm \ref{alg:step2_clean} initialized with $(q,H)$ with the number of iterations $\kappa$ specified in Lemma \ref{lem:comute_kappa}. Then, there is $ N_{\delta} $ such that for all $ n > N_{\delta}  $ we have 
	\begin{equation}\label{eq:thm3.3_p_dist_bound}
		dist(\hat p_n, \MM) \leq \frac{C_1\ln\left(\frac{1}{\delta}\right)}{n^{r_0}}= \mathcal{O}(n^{-r_0}),
	\end{equation}
	\begin{equation}\label{eq:thm3.3_p_bound}
		\norm{\hat p_n - \mathbf{p}} \leq C_2\ln\left(\frac{1}{\delta}\right) \left(\frac{ n}{\ln\left(\ln(n) \right)^2}\right)^{-r_1} = \wtilde{\mathcal{O}}(n^{-r_1})  ,
	\end{equation}
	and 
	\begin{equation}\label{eq:thm3.3_tangent_ang}
		\angle_{\max}(\widehat{T_{\hat p_n}\MM},T_{\mathbf p}\MM) \leq C_3\ln\left(\frac{1}{\delta}\right) \left(\frac{ n}{\ln\left(\ln(n) \right)^2}\right)^{-r_1} = \wtilde{\mathcal{O}}(n^{-r_1})    
	\end{equation}
	with probability of at least $ 1 - \delta $, where $ r_0 = \frac{k}{2k +d}, r_1 = \frac{k-1}{2k +d} $, and $ C_1,C_2,C_3$ are constants independent of $d, \delta$. 
\end{Theorem}
The proof of this theorem can be found in Section \ref{sec:Step2Analysis}.

Note that since the errors in the manifold and tangent estimation are bounded we can also conclude from Theorem \ref{thm:MainResult}, by choosing $\delta = n^{-r_0}$, that
\[
E(dist(\hat p_n, \MM)) \leq C_4 n^{-r_0}.
\]
Similarly, by choosing $\delta = n^{-r_1}$, we have
\[
E(\angle_{\max}(\widehat{T_{\hat p_n}\MM},T_{\hat p}\MM)) \leq  C_5\left(\frac{ n}{\ln\left(\ln(n) \right)^2}\right)^{-r_1}.
\]

We wish to note here that although we show convergence to the underlying manifold when $n\to\infty$ the convergence rates we achieve are sub-optimal in the case of $k=2$ as observed by Genovese et. al. \cite{genovese2012minimax} where the optimal rates for manifold estimation in the Hausdorff sense are shown to be $\mathcal{O}(n^{-2/(2+d)})$. 
This analysis combined with the fact shown by Amaari and Levrard \cite{aamari2019nonasymptotic} lends the possibility that the optimal rates for manifold estimation should be $\mathcal{O}(n^{-k/(k+d)})$, which are not achieved here.
However, the convergence rates we achieve here rely on the rates for local polynomial estimation, which are known to be optimal for non-parametric point estimation of functions.
Thus, we believe that our analysis of the algorithm represents its true rates of convergence rather than just bounding them. While the presented approach may be extended to estimating the entire manifold, and the current bounds may be extended to bound the Hausdorff distance between the estimated manifold and the underling manifold, we believe that the convergence rates will stay as they are.

\section{Proofs}
\label{sec:Proofs}
Theoretically, if we had known the tangent bundle of the sampled manifold at every point, we could have utilized it as a moving frame for the ``$x$-domain" to simply perform  a Moving Least-Squares approximation.
In this case, the convergence analysis would have been similar to standard local polynomial regression \cite{stone1980optimal} (with a varying coordinate system), as the sample bias issue described above would not have occurred. 
Thus, the first part of our investigation is focused on proving that $(q^*(r), H^*(r))$, the solution to the minimization problem of Equation \eqref{eq:argmin1_clean}, yields crude approximations to a tangent of the manifold.

Then, we refine the coordinate system in order to prevent bias introduced by the fact that $H^*(r)$ is tilted with respect to $T_{p}\MM$.
Yet, as we show below, one of the keys to unlocking the convergence rates are the known rates for local polynomial regression.

\subsection{Proof of Theorem \ref{thm:Step1}}
\label{sec:Step1Analysis}
\begin{proof}[proof of Theorem \ref{thm:Step1}]
	The proof can be described by the following three arguments which are proven in Lemmas \ref{lem:J1pTp_clean}, \ref{lem:J_1qH_clean}.
	\begin{enumerate}
		\item[Arg. 1:]\label{outline:Step1_1_clean}
		Denote $p_r = Proj_\MM(r)$.
		Then, since $r - p_r\perp T_{p_r}\MM $ and $r\in\MM_\sigma$, we have that $q = p_r$ along with $H= T_{p_r}\MM$ are in the search space defined by the constraints of \eqref{eq:argmin1_clean}.
		
		\item[Arg. 2:]\label{outline:Step1_2_clean} 
		From Lemma \ref{lem:J1pTp_clean} it follows that for large enough $M = \tau/\sigma$ such that $ \sqrt{\frac{\sigma}{\tau}}  + \frac{\sigma}{\tau} < \frac{1}{2} $ we have $J_1(r; p_r; T_{p_r}\MM) \leq 50\cdot \sigma^2$. 
		According to Assumption \ref{assume:noise} in Section \ref{sec:SamplingAssumptions}, we have that $M$ is large enough, and Lemma \ref{lem:J1pTp_clean} hold.
		Due to the definition of $q^*, H^*$ in \eqref{eq:argmin1_clean}, we achieve that $J_1(r; q^*, H^*) \leq 50 \cdot \sigma^2$ as well.
		
		\item[Arg. 3:]\label{outline:Step1_3_clean} From Lemma \ref{lem:J_1qH_clean}, we have that  
		for $\alpha = \sqrt{C_\sM/M}$, where $M = \frac{\tau}{\sigma}$, and $C_\sM$ is a constant
		the following holds:
		For any $\delta >0$ arbitrarily small there is $ N_{\delta} $ sufficiently large (independent of $\alpha$) such that for all $ n>N_{\delta} $, \textbf{all} pairs $(q, H)$ in the search space of \eqref{eq:argmin1_clean} such that $\angle_{max}(H,T_{p_r}\MM) > \alpha$, the score $J_1(r; q, H) \geq 100 \sigma^2 $ with probability of at least $1 - \delta$.
	\end{enumerate}
	Combining Arguments 2 and 3 we have that 
	for $\alpha = \sqrt{C_\sM/M}$, where $M = \frac{\tau}{\sigma}$, and $C_\sM$ is a constant
	the following holds:
	For any $\delta>0$ arbitrarily small, there exists $ N_{\delta} $ (independent of $\alpha$) such that for all $ n > N_{\delta} $ 
	\[\angle_{\max}(H^*, T_p\MM)\leq \alpha\]
	with probability of at least $ 1 - \delta $. Additionally, since the search space of \eqref{eq:argmin1_clean} requires that $\|r-q^*\| <2\sigma$, we have that $\|p-q^*\|\leq 3\sigma$, and the proof is concluded.
\end{proof}

\begin{Lemma}\label{lem:J1pTp_clean}
	Let the sampling assumption of \ref{sec:SamplingAssumptions} hold, let ${p_r} = Proj_\MM(r)$ be the orthogonal projection of $r$ onto $\MM$, and let $ T_{p_r}\MM $ be the tangent to $ \MM $ at $ {p_r} $.
	Then, for $M$ (of Assumption \ref{assume:noise}) large enough
	\begin{equation}\label{eq:TangentScore_clean}
		J_1(r; {p_r}, T_{{p_r}}\MM) \leq 50 \cdot \sigma^2 
	\end{equation}
\end{Lemma}
\textbf{Idea of the proof: } Since all the sampled points are $\sigma$-close to the manifold which is linearly approximated by the tangent, the mean squared distance to the tangent is of the order of $\OO(\sigma^2)$.
The proof is given in Appendix \ref{subsec:proof_lem_J1pTp_clean}

We show, in Lemma \ref{thm:J1pH_clean} that given a coordinate system $(p_r, H)$ with $\angle_{max}(H,T_{{p_r}}\MM) > \alpha$ with yields a large score of of our cost $J_1$ with high probability.
This will be generalized to a coordinate system $(q, H)$ around any origin $q$ in the search space of \eqref{eq:argmin1_clean} in Lemma \ref{lem:J_1qH_clean}.

\begin{Lemma}\label{thm:J1pH_clean}
	Let the sampling assumptions of Section \ref{sec:SamplingAssumptions} hold.
	Let ${p_r} = Proj_\MM(r)$ be the projection of $r$ onto $\MM$, and $ T_{p_r}\MM $ be the tangent to $ \MM $ at $ {p_r} $. 
	For $\alpha = \sqrt{C_\sM/M}$, where $M = \frac{\tau}{\sigma}$, and $C_\sM$ is a constant,
	the following holds: For any $\delta >0$ there is $N_\delta$ (independent of $\alpha$) such that $\forall n > N_\delta$ all linear sub-spaces $H\in Gr(d,D)$ with $\angle_{max}(H,T_{{p_r}}\MM) > \alpha$, yield a score 
	\[J_1(r; {p_r}, H) \geq 109 \sigma^2 ,\]
	with probability of at least $1 - \delta$.
\end{Lemma}

The proof of Lemma \ref{thm:J1pH_clean} appears in Section \ref{sec:proof_thmJ1pH_clean}.

\begin{Lemma}\label{lem:J_1qH_clean}
	Let the sampling assumption of Section \ref{sec:SamplingAssumptions} hold. Let ${p_r} = Proj_\MM(r)$ be the projection of $r$ onto $\MM$, and $ T_{p_r}\MM $ be the tangent to $ \MM $ at $ {p_r} $.
	For $\alpha = \sqrt{C_\sM/M}$, where $M = \frac{\tau}{\sigma}$, and $C_\sM$ is a constant  the following holds:
	For any $\delta >0$ there is $N_\delta$ (independent of $\alpha$) such that $\forall n > N_\delta$ all linear sub-spaces $H \in Gr(d,D)$ with $\angle_{max}(H,T_{{p_r}}\MM) > \alpha$, and all $q$ in the search space of \eqref{eq:argmin1_clean} yield a score 
	\[
	J_1(r; {q}, H) \geq 100\sigma^2 
	\]
	with probability of at least $1 - \delta$. 
\end{Lemma}
\begin{proof}
	From Lemma \ref{thm:J1pH_clean}, we have that there is a constant $C_\sM$ such that for any $ {\alpha} < \frac{\pi}{2} $ and $M = C_\sM/\alpha^2$, and for any $\tau$ and $\sigma$ maintaining $\frac{\tau}{\sigma}>M$
	the following hold:
	For any $\delta >0$ arbitrarily small there is $ N_\delta $ sufficiently large such that for all $ n>N_\delta $, \textbf{all} linear spaces $H$ with $\angle_{max}(H,T_{p_r}\MM) > \alpha$, yield a score $J_1(r; p_r, H) \geq 109 \cdot \sigma^2 $, with probability of at least $1 - \delta$. 
	We now wish to show that $J_1(r; q, H) \geq 100 \sigma^2$ with high probability as well.
	For convenience we wish to reiterate \eqref{eq:J1_clean}
	\[
	J_1(r; q, H) = \frac{1}{n}\sum_{r_i\in U_{\textrm{ROI}}} \dist^2 (r_i - q,H)
	.\]
	By Constraint 3 of \eqref{eq:J1_clean} we achieve that $dist(q, \MM) \leq 3\sigma$ and so for $p=Proj_\MM(q)$ we have
	$
	\norm{q - p} \leq 3\sigma
	$.
	Thus,
	\[
	\dist^2 (r_i - q,H) = \norm{p - q}^2 + \dist^2 (r_i - p,H) \geq \dist^2 (r_i - p,H) - 9\sigma^2
	,\]
	and we achieve
	$
	J_1(r; q, H) \geq J_1(r; p, H) - 9\sigma^2
	$.
	By Lemma \ref{thm:J1pH_clean} we conclude the proof of the current lemma.
\end{proof}

\subsection{Proof of Theorem \ref{thm:Step2}}
\label{sec:Step2Analysis}

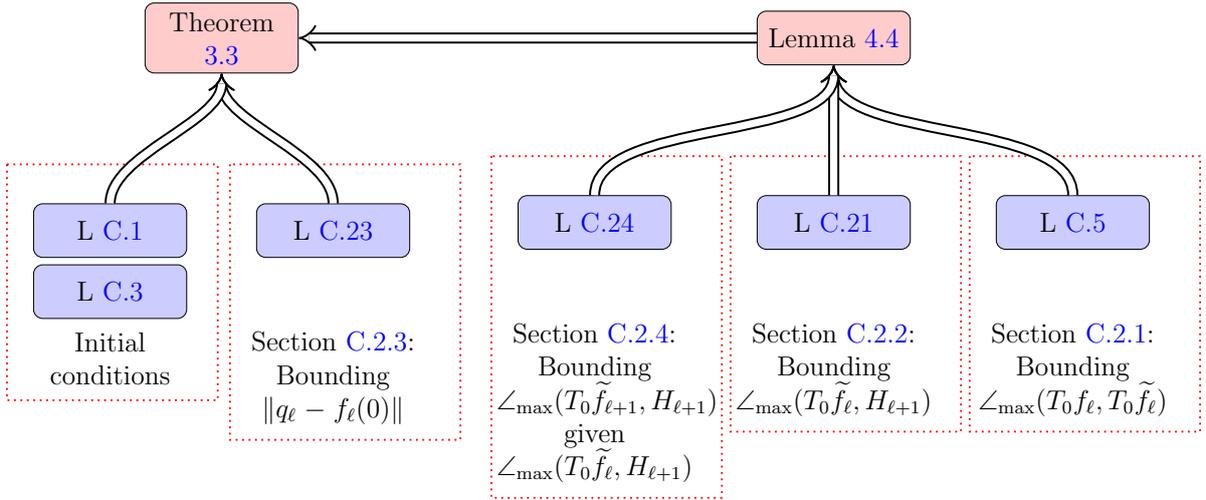
\begin{figure}[htb!]
	\centering
	\resizebox{16cm}{!}{
		\begin{tikzpicture}[auto]
			
			\node [Rblock] (main_lem) {Lemma \textcolor{blue}{\ref{lem:main_support_theorem_step2}}};
			\node [Rblock,left = 7cm of main_lem] (thm33) {Theorem \textcolor{blue}{\ref{thm:Step2}}};
			
			\draw [imp] (main_lem) to[out=180,in=0] (thm33) node [midway, below, sloped] (TextNode) {};
			
			\node [block, below right =\dd and 1.3cm of main_lem] (Tf_Tf_tild) {L \textcolor{blue}{\ref{lem:Tf_TtildeF_D}}};
			\node [below = 1cm of Tf_Tf_tild, text width = 3cm, align = center] 
			(secTf_Tf_tilde_title) {Section \textcolor{blue}{\ref{sec:Tf_Tf_tilde_err}}: Bounding $\maxangle(T_0 f_\ell, T_0 {\wtilde{f}_\ell})$};
			\draw[red,thick,dotted] ($( Tf_Tf_tild.north west)+(-0.4,0.6)$)  rectangle ($(secTf_Tf_tilde_title.south east)+(0.3,-0.1)$);
			\node [block, below =\dd of main_lem] (Tf_tilde_finite_est) {L \textcolor{blue}{\ref{lem:AngleImprovement}}};
			\node [below = 1cm of Tf_tilde_finite_est, text width = 3cm, align = center]
			(secfinit_sample_tangen_title) 
			{Section \textcolor{blue}{\ref{sec:FiniteSample_err}}: 
				Bounding
				$\maxangle(T_0\wtilde f_\ell, H_{\ell + 1})$};
			\draw[red,thick,dotted] ($( Tf_tilde_finite_est.north west)+(-0.4,0.6)$)  rectangle ($(secfinit_sample_tangen_title.south east)+(0.3,-.1)$);
			\node [block, below left = \dd and 1.3cm of main_lem]  (Tf_move_rot) {L \textcolor{blue}{\ref{lem:shift_ang_general}}};
			\node [below = 1cm of Tf_move_rot, text width = 3cm, align = center]
			(secshifted_center_t_title) {Section \textcolor{blue}{\ref{sec:ShiftedCenter_err}}: 
				Bounding 
				$\maxangle(T_0\wtilde f_{\ell+1}, H_{\ell + 1})$ 
				given 
				$\maxangle(T_0\wtilde f_\ell, H_{\ell + 1})$};
			\draw[red,thick,dotted] ($( Tf_move_rot.north west)+(-0.4,0.6)$)  rectangle ($(secshifted_center_t_title.south east)+(0.3,-0.1)$);
			\node[block, below right = \dd and -0.65 cm of thm33](f0_move_rot2) {L \textcolor{blue}{\ref{lem:dist_to_f_l0_srtong_small_alpha}}};
			\node [below = 1cm of f0_move_rot2, text width = 3cm, align = center] 
			(secf0_est_title) {Section \textcolor{blue}{\ref{sec:f0_est}}: Bounding $\|q_\ell - f_\ell(0)\| $};
			\draw[red,thick,dotted] ($( f0_move_rot2.north west)+(-0.4,0.6)$)  rectangle ($(secf0_est_title.south east)+(0.3,-0.1)$);

			\node [block, below left = \dd and -0.65 of thm33](init_H0) {L \textcolor{blue}{\ref{lem:H0_Tf_Init}}};
			\node [block, below = 0.1cm of init_H0](init_f0) {L \textcolor{blue}{\ref{lem:dist_to_f_00}}};

			\node [below = 1cm of init_H0, text width = 3cm, align = center] (secX_title) {Initial \\ conditions};
			\draw[red,thick,dotted] ($( init_H0.north west)+(-0.4,0.6)$)  rectangle ($(secX_title.south east)+(0,-0.1)$);
			
			
			\draw [imp] (Tf_Tf_tild) to[out=90,in=-90,looseness=0.73] (main_lem) node [midway, below, sloped] (TextNode) {};
			\draw [imp] (Tf_tilde_finite_est) to[out=90,in=-90,looseness=0.73] (main_lem) node [midway, below, sloped] (TextNode) {};
			\draw [imp] (Tf_move_rot) to[out=90,in=-90,looseness=0.73] (main_lem) node [midway, below, sloped] (TextNode) {};
			\draw [imp] (f0_move_rot2) to[out=90,in=-90,looseness=0.73] (thm33) node [midway, below, sloped] (TextNode) {};
			
			\draw [imp] (init_H0) to[out=90,in=-90,looseness=0.73] (thm33) node [midway, below, sloped] (TextNode) {};

		\end{tikzpicture}
	}
	\caption{Road-map for proof of Theorem \ref{thm:Step2}}
	\label{tikz:thm33_lemmas}
\end{figure}


Before delving into the details of the proof, we wish to reiterate the steps of Algorithm~\ref{alg:step2_clean} while introducing some useful notations.
According to the assumptions of Theorem \ref{thm:Step2} we have a local coordinate system $(q, H)\in \RR^D\times Gr(d, D)$, such that $ \norm{q-p} \leq 3 \sigma $ and $\angle_\textrm{max}(T_p\MM, H) \leq \alpha$, where $p\in \MM$ .
As Algorithm \ref{alg:step2_clean} involves an iterative process, we denote the initial directions of the tangent estimation as $H_0 := H$.
Then, at each iteration we update the local coordinates' directions $H_\ell$ (for $\ell=1, \ldots, \kappa$). 

Similar to \eqref{eq:FunctionGraph_init} and using the result of Lemma \ref{lem:M_is_locally_a_fuinction_clean}, we begin by looking at the manifold patch $\MM \cap \textrm{Cyl}_{H_0}(p,c_{\pi/4}\tau,\tau/2)$ as the graph of a function $f_{0}:(r,H_0) \simeq\RR^d\to\RR^{D-d}$; i.e., 
\begin{equation}
	\Gamma_{f_{0}, r,H_0} \defeq \{r + (x, f_{0}(x))_{H_0} ~|~ x\in B_{H_0}(0, c_{\pi/4}\tau)\}  
	,\end{equation}
where $\textrm{Cyl}_{H_0}(p,c_{\pi/4}\tau, \tau/2)$ is the $D$-dimensional cylinder with the base $B_H(p, c_{\pi/4}\tau)\subset H_0$ and height $\tau/2$ in any direction on $H_0^\perp$. 
For the remainder of this section, we assume that $ \alpha = \sqrt{C_\sM/M} $ is small enough (see the sampling assumptions in Section \ref{sec:SamplingAssumptions}).
Then, at each iteration of Algorithm \ref{alg:step2_clean} we update $H_\ell$ to $H_{\ell + 1}$ through taking the linear space coinciding with the directions of the estimated tangent to $f_{\ell}$ at zero. 
Explicitly, given $(r, H_{\ell})$, we look at the manifold as $\Gamma_{f_{\ell}, r, H_{\ell}}$ the local graph of a function $f_{\ell}$ and estimate $T_0 f_{\ell} \in Gr(d, D)$ the tangent to the graph of $f_{\ell}$ at $0$ through taking the image of $\DD_{\pi^*_{r, H_{\ell}}}[0]$, the first order differential of $\pi^*_{r, H_{\ell }}$ at zero.
In other words, we ``rotate" the coordinate system to the point where $H_{\ell+1}$ aligns with the former tangent estimation to get $f_{\ell+1}$. 

However, if we want to use the well-known convergence rates of local polynomial regression (i.e., the minimization of \eqref{eq:argmin2_clean}), a key assumption in the analysis is that the noise has a zero mean (see \cite{stone1980optimal}). 
However, in our case, for any $x\in H_\ell$, the samples above $x$ are uniformly distributed~in 
\begin{equation}\label{eq:Omega_def}
	\Omega_\ell(x) = (x + H_\ell^\perp) \cap  \MM_\sigma
	,\end{equation}
where $x + H_\ell^\perp = \{x + y ~|~ y\in H_\ell^\perp\}$.
That is, $\Omega_\ell(x):H_\ell \simeq \RR^d \to 2^{\RR^{D-d}}$.
Thus, denoting $\eta_\ell(x)\sim \textrm{Unif}(\Omega(x))$ and defining 
\begin{equation}\label{eq:ftilde_def}
	\wtilde f_\ell(x) = \mathbb{E}[\eta_\ell(y | x)]     \neq f_\ell
	,\end{equation}
the result of the regression will estimate $\wtilde f_\ell$ rather than $f_\ell$ itself  (see Figure \ref{fig:TiltedBias}). Thus, when we wish to estimate $f_\ell$ or its derivatives using local polynomial regression, we estimate a different function, for which the noise has zero mean (see Figure \ref{fig:bias_ex_clean}).

\begin{figure}[htb]
	\floatbox[{\capbeside\thisfloatsetup{capbesideposition={right,top},capbesidewidth=8cm}}]{figure}[\FBwidth]
	{\caption{Illustration of $ \MM $ as a graph of a function $ f_\ell $ \emphbrackets{marked by the red solid line} above the coordinate system $ f_\ell $. The boundary of  $ \MM_\sigma $ is delineated by the pink lines and $ \wtilde f_\ell $ is the conditioned expectancy $ \EE[\eta_\ell(y ~|~ x)] $ of this domain with respect to the presented $ y $-axis. }\label{fig:bias_ex_clean}}
	{\includegraphics[width=0.3\textwidth,height= 0.35\textwidth]{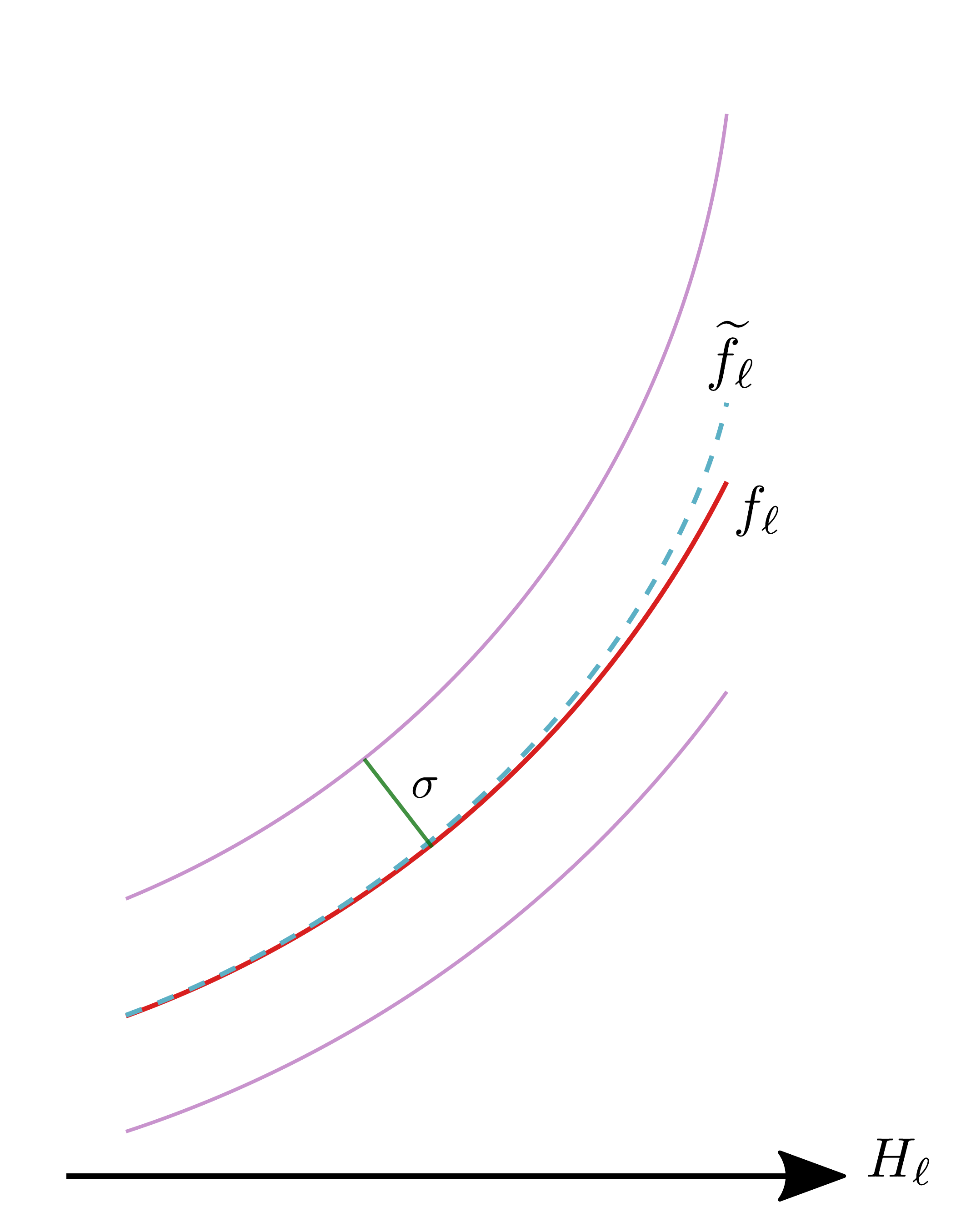}}
\end{figure}

Since for any arbitrarily fixed $ x $ the density of the random variable $ \eta_\ell(y | x) $ is constant, $\wtilde{f}_\ell(x)$ can be calculated as the mean of the set $\Omega_\ell(x)$ defined in \eqref{eq:Omega_def}.
Furthermore, the farthest point in each such direction is exactly $\sigma$ away from the graph of $f_\ell$ (see Figure \ref{fig:draw_step_2_bias_clean}).

\begin{figure}[b]
	\floatbox[{\capbeside\thisfloatsetup{capbesideposition={right,top},capbesidewidth=8cm}}]{figure}[\FBwidth]
	{\caption{Illustration of $g(0,\vec{y})$ and $\Omega(0)$ in the two dimensional case. Let $H$ be some local coordinate system and consider $\MM$ as a local graph of some function $f:H\to H^\perp$. The upper bound for the values of the sample distribution above $0$ in some direction $\theta$ of $H^\perp$ is $g(0,\theta)$. This value is $\sigma$-away from some point on $\MM$. We denote this point by $f(\wtilde x(0))\in \MM$.}\label{fig:draw_step_2_bias_clean}}
	{\includegraphics[width=0.45\textwidth]{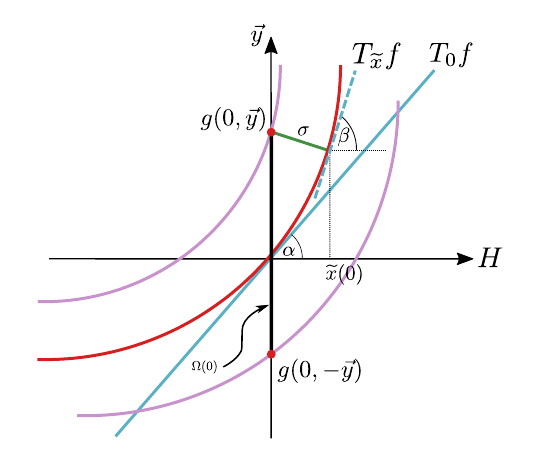}}
\end{figure}

A key lemma in the proof of Theorem \ref{thm:Step2} is the following, which is proved in Section \ref{subsubsec:proof_of_main_support_lemma}.

\begin{Lemma} \label{lem:main_support_theorem_step2}
	Let $H_\ell\in Gr(d, D)$, and let $f_\ell:H_\ell\simeq\RR^d \rightarrow \RR^{D-d}$, defined as in \eqref{eq:fl_def}. Define $H_{\ell+1} = \Ima(\DD_{\pi^*_{q_\ell,H_{\ell}}}[0])$, as in Algorithm \ref{alg:step2_clean}, and let $r_1 = \frac{k-1}{2k +d}$. 
	Assume that $M = \frac{\tau}{\sigma} \geq C_\stau \sqrt{D\log D}$ (where $C_\stau$ is a constant from Lemma \ref{lem:Tf_TtildeF_D}).
	For any $\delta>0$, there is $N$ such that for any number of samples $n>N$, and any $\alpha \leq\sqrt{1/D}$ smaller than some constant \emph{(}see Theorem \ref{thm:Step1} and Section \ref{sec:SamplingAssumptions} for the definition of $M$ and $C_\sM$\emph{)}.  If
	\begin{enumerate}
		\item $\maxangle (H_\ell, T_0 f_\ell) = \alpha_\ell\leq \alpha$
		\item $\|f_\ell(0)\| \leq \min (c_{\pi/4} \cdot \tau, \sqrt{\tau/32} -1) $
		\item $N$ satisfies $ 12\sqrt{d} \frac{C_0 \ln(1/\delta)}{N^{r_1}}\leq \alpha $ 
	\end{enumerate}
	hold, where $C_0$ is some constant. Then, we have  
	\[
	\alpha_{\ell+1} = \maxangle (H_{\ell+1}, T_0 f_{\ell+1})  = \maxangle (H_{\ell+1}, T_0 f_{\ell+1/2}) \leq \alpha/2
	\]
	and
	\[
	\|f_{\ell+1}(0)\| \leq \|f_\ell(0)\|\left(1+40\alpha^2\right) 
	\]
	with probability at least $1-\delta$.
\end{Lemma}
We now turn to the proof of Theorem~\ref{thm:Step2}.
\begin{proof}[proof of Theorem \ref{thm:Step2}]
	We divide Algorithm \ref{alg:step2_clean} into four steps:
	\begin{enumerate}
		\item[(i)] Show that the initialization $H_0 $ results with $\maxangle(H_0, T_{0} f_0)$ and $\|f_0(0)\|$ small enough (using Lemmas \ref{lem:H0_Tf_Init} and \ref{lem:dist_to_f_00}).
		\item[(ii)] Show that estimating $H_{\ell+1}$ from $H_{\ell}$ (with the origin set at $r$) improves the angle  $\maxangle(H_\ell, T_{0} f_\ell)$ - corresponds to rows 4 to 6 in Algorithm \ref{alg:step2_clean} (using Lemma \ref{lem:main_support_theorem_step2}).
		\item[(iii)] We show that for $\mathrm{\mathbf{p}} = Proj_\MM(r)$ - the projection of $r$ onto $\MM$, $\maxangle(H_\kappa, T_{\mathrm{\mathbf{p}}}\MM)$  is small (as required in \eqref{eq:thm3.3_tangent_ang}) (Using Lemmas \ref{lem:dist_given_angle}).
		\item[(iv)] Show that $\hat p_n$  is a good estimate of  $\mathrm{\mathbf{p}}$ - corresponds to row 8 in Algorithm \ref{alg:step2_clean} (using Lemma~\ref{lem:dist_to_f_l0_srtong_small_alpha}).
	\end{enumerate}
	Note that step (ii) is repeated $\kappa$ times.
	Below, we treat the four steps one by one. 
	Although we prove this point later, we start by assuming that $\kappa$ and $\delta_1$ are
	\begin{equation}\label{eq:n_bound1}
		12\sqrt{d}\frac{C_0 \ln(1/\delta_1)}{n^{r_1}} \leq \alpha_0 2^{-\kappa+1},
	\end{equation}
	where $C_0$ is the constant from Lemma \ref{lem:main_support_theorem_step2}. 
	
	\paragraph*{Step (i):}
	Let $\alpha_0=\sqrt{1/D}$ and denote by $A_0$ the event of
	\[
	\maxangle (T_0f_{0}, H_{0}) \leq \alpha_0 
	~~~~\mbox{ and } ~~~~
	\|f_0(0)\| \leq 10\sigma 
	.\]
	From Lemma \ref{lem:H0_Tf_Init} and Theorem \ref{thm:Step1} that for $M$ large enough, there is $N_{1,\delta_1}$ such that for all $n>N_{1,\delta_1}$ we have that:
	\[
	\Pr\left(\maxangle (T_0f_{0}, H_{0}) \leq \alpha_0 \right) \geq 1-\delta_1,
	.\]
	From Lemma \ref{lem:dist_to_f_00} we have that
	\[
	\Pr\left(\|f_0(0)\| \leq 10\sigma \right) \geq 1-\delta_1.
	\]
	This means that $P(A_0) \geq 1-\delta_1$ (note that the failure probability comes from Theorem \ref{thm:Step1}, and once the result of the theorem holds, the conditions of Lemmas \ref{lem:H0_Tf_Init}, \ref{lem:dist_to_f_00} are met).

	\paragraph*{Step(ii):}
	Denote by $A_\ell$ the event that $\maxangle (T_0f_{\ell}, H_{\ell}) \leq \alpha_0 2^{-\ell} = \alpha_{\ell}$ and also
	\[
	\|f_\ell(0)\| \leq \|f_{\ell-1}(0)\|\left(1+40\alpha_{\ell-1}^2\right)
	.\]
	We now wish to apply Lemma \ref{lem:main_support_theorem_step2} in an iterated fashion.
	To initiate the iterative process, we have to assume that $n$ is at least $N_{2, \delta_1} = \left( \frac{12\sqrt{d}C_0\ln(1/\delta)}{\alpha_0}\right)^{1/r_1}$ (i.e., requirement 3 of the Lemma).
	To comply with requirement 3 of Lemma \ref{lem:main_support_theorem_step2} in all the following iterations, the total number of iterations $ \kappa$ that we can perform is limited by the actual number of samples~$n$ (since  $\alpha$ shrinks with each iteration).
	Below, we show what should be the value of $\kappa$ that allows for the iterated application of Lemma \ref{lem:main_support_theorem_step2} given $n$.
	
	Since $M>C_\stau \sqrt{D\log D}$, and $\alpha_0 \leq \sqrt{ 1/D}$ if events $A_0, \ldots, A_\ell$ hold, then the requirements of Lemma \ref{lem:main_support_theorem_step2} are met for $\alpha = \alpha_\ell = \alpha_0 2^{-\ell}$. 
	The only nontrivial requirement is the bound on $\|f_\ell(0)\|$. Since all the events $A_{i}$ for $i\leq \ell$ hold we have
	\begin{align*}
		\|f_\ell(0)\| &\leq \|f_{\ell-1}(0)\|\left(1+40\alpha_{\ell-1}^2  \right) \nonumber
		\\
		&\leq \|f_{\ell-2}(0)\|\left(1+40\alpha_{\ell-2}^2 \right)\left(1+40\alpha_{\ell-1}^2  \right) \nonumber
		\\
		&\leq \|f_{0}(0)\|\prod_{i=0}^{\ell} \left(1+40\alpha_{i}^2  \right)
		\\
		&\leq \|f_{0}(0)\|\prod_{i=0}^{\ell} \left(1+40\alpha_0^2 2^{-2i} \right)
	\end{align*}
	Denote $a_0 = 40\alpha_0 ^2$, and then
	\begin{align*}
		\|f_\ell(0)\| &\leq \|f_{0}(0)\|\prod_{i=0}^{\ell} \left(1+a_0 2^{-2i} \right)
		\\
		& = \|f_{0}(0)\|\exp\left(\sum_{i=0}^{\ell} \ln\left(1+a_0 2^{-2i} \right)\right)
		\\
		& \leq \|f_{0}(0)\|\exp\left(\sum_{i=0}^{\ell} a_0 2^{-2i} \right)
		\\
		& \leq \|f_{0}(0)\|\exp\left(\sum_{i=0}^{\infty} a_0 2^{-2i} \right)
		\\
		& \leq \|f_{0}(0)\|\exp\left(2a_0  \right)
	\end{align*}
	Thus, we have that for $n> N_{2,\delta_1}$, with probability of at least $1-\delta_1$, the event $A_{\ell+1}$ holds given events $A_0, \ldots, A_{\ell}$ hold for any $\ell \leq \kappa$.

	To conclude our arguments so far, since $\kappa$ satisfies \eqref{eq:n_bound1}, using the union bound on the events $A_0,\ldots, A_\kappa$  we have that 
	\begin{equation}\label{eq:final_prob_bound1}
		\Pr\left(\maxangle (T_0 f_{\kappa}, H_{\kappa}) \leq \alpha_0 2^{-\kappa}\right) \geq 1 -  \kappa\delta_1.        
	\end{equation}
	Choosing $\delta_1 = \frac{\delta}{\kappa+1}$, we have from Lemma \ref{lem:comute_kappa}, that $\kappa$ from \eqref{eq:kappa_res_lem} satisfies  \eqref{eq:n_bound1}. We also have from Lemma \ref{lem:comute_kappa} that 
	\[
	\alpha_0 2^{-\kappa} \leq \tilde C_{d} \ln\left(\frac{1}{\delta}\right) n^{-r_1}\left(\ln\left(\ln(n) \right)\right)^{2r_1}.
	\]
	Thus, we have that 
	\begin{equation}\label{eq:prob_ang_bound_Tfk_Hk}
		\maxangle (Tf_{\kappa}(0), H_{\kappa}) \leq \tilde C_{d}\ln\left(\frac{1}{\delta}\right) n^{-r_1}\left(\ln\left(\ln(n) \right)\right)^{2r_1},    
	\end{equation}
	given events $A_0,\ldots A_\kappa$ hold.
	\paragraph*{Step (iii):} Denote by $p_{\kappa}=(r,f_\kappa(0))_{r, H_{\kappa}}$, and $\mathrm{\mathbf{p}} = Proj_\MM(r)$. Applying Lemma \ref{lem:dist_given_angle} we have that 
	\begin{equation}\label{eq:P_q_dist_bound}
		\|\mathrm{\mathbf{p}} - p_{\kappa}\| \leq \frac{\hat C_{d}\ln\left(\frac{1}{\delta}\right) n^{-r_1}\left(\ln\left(\ln(n) \right)\right)^{2r_1}}{\left(\frac{1}{2\sigma(2+\sigma/\tau)}-\frac{c}{\tau} \right)},
	\end{equation}
	with probability of at least $1-\delta$. From Corollary 3 in \cite{boissonnat2017reach}, we have that 
	\begin{equation} \label{eq:sin_TPM_Tfk}
		\sin \frac{\maxangle(T_{\mathbf{p}}\MM,Tf_\kappa(0))}{2} \leq \frac{\|\mathrm{\mathbf{p}}-p_\kappa\|}{2\tau}.
	\end{equation}
	or,
	\begin{equation} \label{eq:TPM_Tfk}
		\maxangle(T_{\mathbf{p}}\MM,Tf_\kappa(0)) \leq \frac{2\|\mathrm{\mathbf{p}}-p_\kappa\|}{\tau}.
	\end{equation}
	Combining \ref{eq:TPM_Tfk} and \eqref{eq:P_q_dist_bound} we have 
	\begin{equation} \label{eq:TPM_Tfk_full}
		\maxangle(T_{\mathbf{p}}\MM,Tf_\kappa(0))
		\leq \frac{2\hat C_{d}\ln\left(\frac{1}{\delta}\right) n^{-r_1}\left(\ln\left(\ln(n) \right)\right)^{2r_1}}{\tau\left(\frac{1}{2\sigma(2+\sigma/\tau)}-\frac{c}{\tau} \right)} = 
		\frac{2\hat C_{d}\ln\left(
			\frac{1}{\delta}\right) n^{-r_1}\left(\ln\left(\ln(n) \right)\right)^{2r_1}}{\frac{M}{4 + 2/M} -c} 
		,
	\end{equation}
	given events $A_0,\ldots A_\kappa$ hold.
	Note that the expression $\frac{M}{4 + 2/M} -c $ in \eqref{eq:TPM_Tfk_full} grows with $M$.
	Combining \eqref{eq:TPM_Tfk_full} with \eqref{eq:prob_ang_bound_Tfk_Hk}, and using the triangle inequality, we get 
	\begin{equation}\label{eq:res_thm33_tang}
		\maxangle(T_{\mathbf{p}}\MM,H_\kappa) \leq C_d \cdot \ln(1/\delta)n^{-r_1} (\ln \ln(n))^{2 r_1}
	\end{equation}
	given events $A_0,\ldots A_\kappa$ hold, where $C_d$ is some general constant. 
	Note that according to Algorithm~\ref{alg:step2_clean}  $\widehat{T_{\hat p_n}\MM} = H_\kappa$.
	
	\paragraph*{Step (iv):} In order to bound the error $\|\hat p_n - \mathrm{\mathbf{p}}\|$ we use the triangle inequality 
	\[
	\|\hat p_n - \mathrm{\mathbf{p}}\| \leq \|\hat p_n - p_\kappa\| + \| p_\kappa - \mathrm{\mathbf{p}}\|, 
	\]
	where $p_\kappa = (r,f_\kappa(0))_{H_{\kappa}}$,
	and bound the two parts in the right hand side separately.
	
	The first part, $\| p_\kappa - \mathrm{\mathbf{p}}\| $ is bounded in \eqref{eq:P_q_dist_bound}.
	
	Furthermore, assuming event $A_\kappa$ holds, since $\alpha = \alpha_0 2^{-\kappa}<1/D$, there is $N_3$ such that for $n>N_{3}$ we have from Lemma \ref{lem:dist_to_f_l0_srtong_small_alpha} that  
	\[
	\|\hat p_n - p_\kappa \| \leq 8 \sigma D \alpha_1^2 2^{-2\kappa} + \frac{c\ln\left(\frac{1}{\delta}\right)}{n^{r_0}}
	\]
	holds with probability at least $1-\delta/(\kappa+1)$, where $c$ is some general constant and $ r_0 = \frac{k}{2k +d}$. Explicitly. 
	Denote this event by~$B$.
	Substituting $\kappa$ from \eqref{eq:kappa_res_lem} we have from Lemma \ref{lem:comute_kappa} that 
	\[
	\| \hat p_n - p_\kappa \| \leq 8\sigma D C_{d}^2 \ln\left(\frac{1}{\delta}\right)^2 n^{-2r_1}\left(\ln\left(\ln(n) \right)\right)^{4r_1} + c\ln\left(\frac{1}{\delta}\right)n^{-r_0}
	\]
	or, for $n$ large enough, and some constant $ C_1$,
	\begin{equation}\label{eq:pnhat-pkappa_bound}
		\|\hat p_n - p_\kappa \| \leq C_1 \ln\left(\frac{1}{\delta}\right) n^{-r_0}.
	\end{equation}
	Note that $p_\kappa \in \MM$, so $dist(\hat p_n,\MM) \leq \|\hat p_n - p_\kappa \|$.
	Combining this with \eqref{eq:P_q_dist_bound}, assuming events $A_0,\ldots A_\kappa,B$ hold, we have that 
	\[
	\| \hat p_n - \mathrm{\mathbf{p}}\|  \leq  C_1 \ln\left(\frac{1}{\delta}\right) n^{-r_0} + \frac{\hat C_{d}\ln\left(\frac{1}{\delta}\right) n^{-r_1}\left(\ln\left(\ln(n) \right)\right)^{2r_1}}{\left(\frac{1}{2\sigma(2+\sigma/\tau)}-\frac{c}{\tau} \right)}.
	\] 
	Alternatively, for $n$ large enough, and some constant $C$, given the events $A_1,\ldots, A_\kappa, B$ hold, 
	\begin{equation} \label{eq:pn-p_bound}
		\|\hat p_n - \mathrm{\mathbf{p}}\| \leq C\ln\left(\frac{1}{\delta}\right) n^{-r_1}  \left(\ln\left(\ln(n) \right)\right)^{2r_1}.    
	\end{equation}
	Using the union bound on $A_1,\ldots, A_\kappa, B$, we have that Equations \eqref{eq:res_thm33_tang}, \eqref{eq:pnhat-pkappa_bound} and \eqref{eq:pn-p_bound} are satisfied with probability of at least $1 - \delta$.
	This concludes the proof of Theorem \ref{thm:Step2}. 
\end{proof}

\section{A Possible Application}\label{sec:applications}
While there are numerous applications for this method, we present here one example that demonstrate the potential of the presented approach.
In this example we show how the presented method can be used to follow the trajectory of a geodesic line on a manifold. 
We assume that a point $x_0$ is chosen on the manifold and some direction $\vec v_0$ on the tangent $T_{x_0} \MM$ (In practice, the process can initialized with a point near the manifold $\MM$ and then project it to the manifold). 

The process of tracking a geodesic line is iterative. 
At each step, we compute $\wtilde x_{i+1} =  x_i + \eps \vec{v}$, then ``project'' the new point back to the manifold $x_{i+1} \approx \mbox{Proj}_\MM \wtilde x_{i+1}$ through Algorithms \ref{alg:step1Inpractice_clean} and \ref{alg:step2InPractice_clean}, and parallel transport $\vec v_i$ to $T_{x_{i+1}}$ to get $\vec v_{i+1}$.

In the first toy case, the manifold $\MM$ is a circle of radius 10 in $\RR^2$. 
The dataset consists of 5000 points. 
We start with some sample (illustrated in red in Figure \ref{fig:geoWalk_circle}), project it onto the circle (in Figure \ref{fig:geoWalk_circle}, the circle is marked in blue, and the projected point in green), and than move the point in some direction, project it again (shown in another green point in Figure \ref{fig:geoWalk_circle}), etc.

\begin{figure}
	\floatbox[{\capbeside\thisfloatsetup{capbesideposition={left,top},capbesidewidth=8cm}}]{figure}[\FBwidth]
	{\caption{Geodesic ``walk" on a circle. The red point is the initial point. The yellow points are the data set. The red point is then projected onto the estimation of the blue circle. Then at each step, a new point is generated along the circle (the black arrows connect the points. The plot illustrates 30 steps.}\label{fig:geoWalk_circle}}
	{\includegraphics[width=0.35\textwidth]{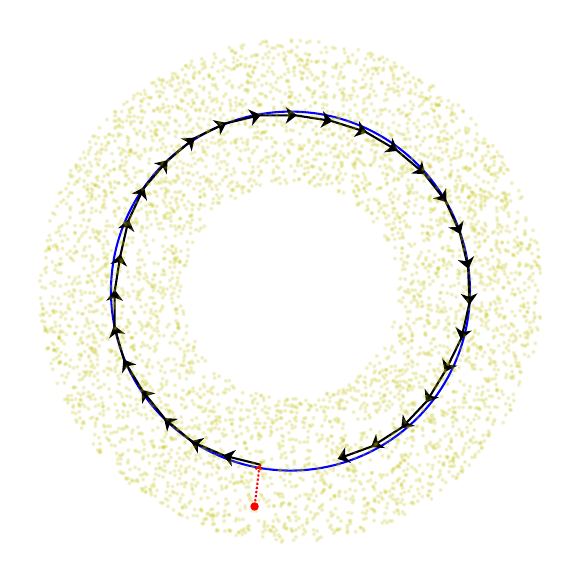}}
\end{figure}

In the second example, we took a 3d model of an airplane\footnote{\textcolor{blue}{\href{http://3dmag.org/en/market/download/item/4740/}{http://3dmag.org/en/market/download/item/4740/}}}, rotated it in the z-axis, and took 2000 snapshots. Each snapshot is an image of $290 \times 209$ gray-scale pixels. The input data set consist of the unsorted images, sampled from a one dimensional manifold embedded in $\RR^{60,610}$.
Several such images appear in Figure \ref{fig:geoWalk_airplane}. 
Starting from some image, we create a movie of the rotating airplane. 
The movie can be found in
\ifthenelse{\blind=1}{\textcolor{blue}{\href{https://youtu.be/aHYyUvu1Q-8}{https://youtu.be/aHYyUvu1Q-8}}}{the \textcolor{blue}{attached mp4 file}}, and the code for generating it can be found in \ifthenelse{\blind=1}{\textcolor{blue}{\href{https://github.com/aizeny/manapprox}{https://github.com/aizeny/manapprox}}}{the \textcolor{blue}{attached code zip file}}.

\begin{figure}[H]
	\centering
	\includegraphics[width=0.16\textwidth]{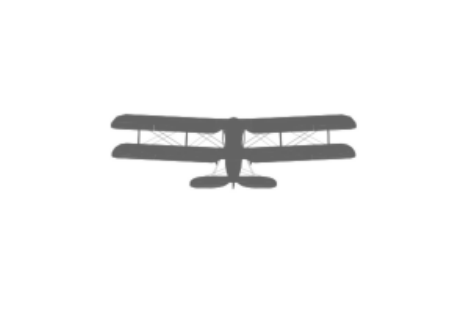}
	\includegraphics[width=0.16\textwidth]{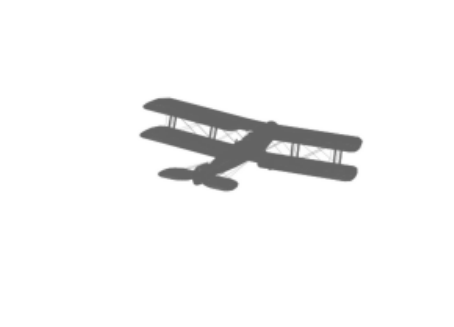}
	\includegraphics[width=0.16\textwidth]{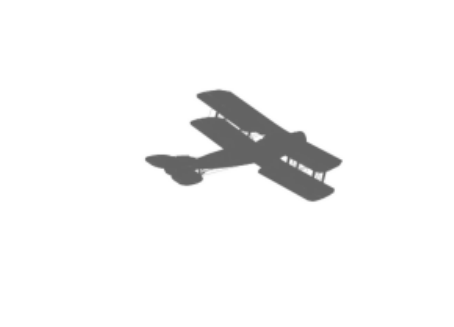}
	\includegraphics[width=0.16\textwidth]{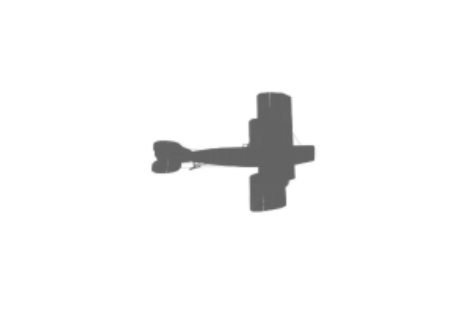}
	\includegraphics[width=0.16\textwidth]{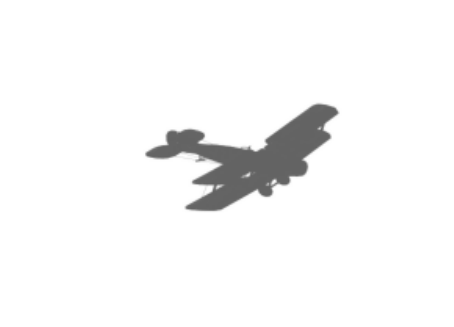}
	\includegraphics[width=0.16\textwidth]{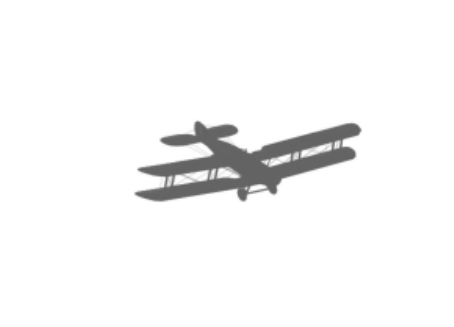}
	\\
	\includegraphics[width=0.16\textwidth]{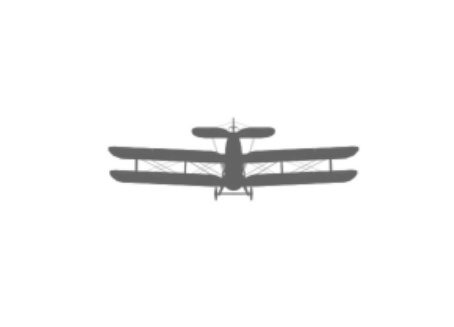}
	\includegraphics[width=0.16\textwidth]{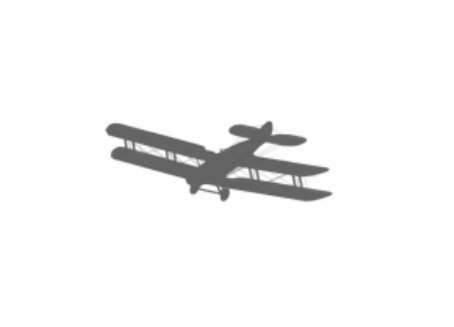}
	\includegraphics[width=0.16\textwidth]{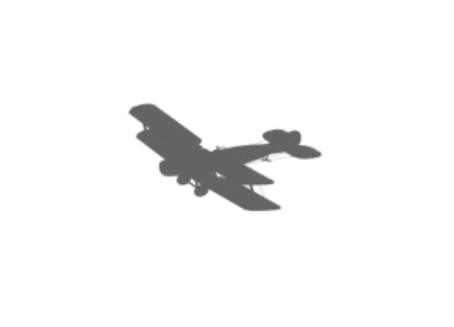}
	\includegraphics[width=0.16\textwidth]{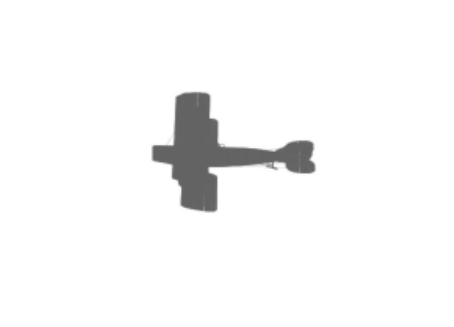}
	\includegraphics[width=0.16\textwidth]{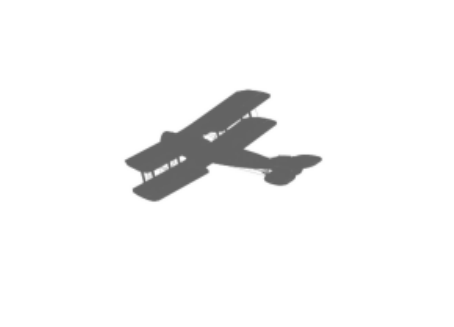}
	\includegraphics[width=0.16\textwidth]{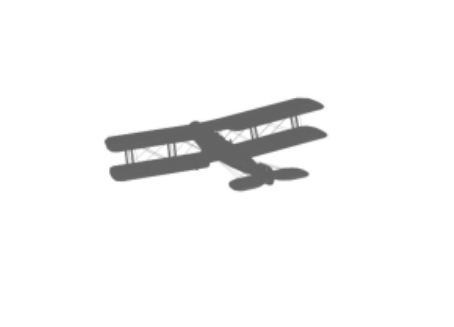}
	\caption{Sample images from a 3d model of an airplane.}
	\label{fig:geoWalk_airplane}
\end{figure}

\section{Conclusions}
This paper presents a novel algorithm for point estimation and denoising of manifold data. Our approach achieves theoretical convergence guarantees and demonstrates practical utility in manifold reconstruction tasks. The key contributions and implications of our work can be summarized as follows:
First, we developed and analyzed a two-step algorithm that effectively estimates points on the underlying manifold from noisy samples. The algorithm first establishes an initial coordinate system and then iteratively refines it to improve accuracy. Our theoretical analysis shows that this approach achieves convergence rates of $\OO(n^{-k/(2k+d)})$ for manifold estimation and $\OO(n^{-(k-1)/(2k+d)})$ for tangent space estimation.
A limitation of our analysis is the sampling model, which assumes uniform distribution within the $\sigma$-tubular neighborhood. The uniform sampling assumption enables rigorous theoretical analysis and convergence guarantees, but it is different than the additive normal noise models commonly used in manifold estimation literature. 
An important future direction would be extending our framework to handle more general noise distributions while maintaining theoretical guarantees.
Our method's practical value is demonstrated through applications like geodesic trajectory tracking, as illustrated in our examples. 
These applications suggest the algorithm's potential utility in various domains requiring manifold reconstruction and geometric processing of noisy data.

Several important research directions emerge from this work. 
A natural extension would be developing a complete manifold estimation framework that maintains our theoretical guarantees while reconstructing the entire manifold, not just individual points. 
Additionally, there are intriguing theoretical questions regarding the relationship between our achieved convergence rates and known optimal rates in manifold estimation. 
In particular, the apparent discrepancy between our rates and those established by Genovese et al. for manifold estimation in the Hausdorff sense warrants further investigation. 
Understanding these connections could provide deeper insights into the fundamental limits of manifold estimation under different noise models and sampling conditions.

In conclusion, while our current sampling assumptions present some limitations, the theoretical guarantees and practical effectiveness of our algorithm provide a solid foundation for future developments in manifold estimation and geometric data processing.

\section{Acknowledgments}
We wish to thank Prof. Felix Abramovich for driving us to do this work and for the fruitful discussions.
We also thank Prof. Ingrid Daubechies and Prof. David Levin for various discussions along the road.
B. Sober is supported by The Hebrew University of Jerusalem.

\begin{appendices}

\section{Preliminaries}
Before we delve into the proofs, we wish to introduce the concepts of Principal Angles between linear sub-spaces \cite{jordan1875essai,bjorck1973numerical} as well as develop some general results concerning the viewpoint of the manifold as being locally a graph of some function from a local coordinate system.
Both of these topics will play a key role in the proofs below.

In addition, two bounds resulting from the Taylor expansion will be used extensively in our proofs.
Thus, we note them here as the two following remarks:

\begin{remark}\label{rem:taylor_sqrt1-x2_clean}
For $x\in[0,\sqrt{3}/2]$
\begin{equation}\label{eq:Taylor2ndOrder}
    1-1/2x^2 \geq \sqrt{1 - x^2} \geq 1-x^2
\end{equation}
\end{remark}
\begin{remark}\label{rem:taylor2_sqrt1-x2_clean}
For $x\in[0,\sqrt{3}/2]$
\begin{equation}
1-1/2x^2-1/8x^4 \geq \sqrt{1 - x^2} \geq 1-\frac{1}{2}x^2 - x^4    
\end{equation}
\end{remark}

\subsection{Principal angles between linear sub-Spaces}
The concept of Principal Angles between flats were first introduced by Jordan in 1875 \cite{jordan1875essai}.
Below, we use the definition of Principal Angles between subspaces as described in \cite{bjorck1973numerical}.
\begin{definition}[Principal Angles]\label{def:principal_angles_clean}
	Let $V$ be an inner product space. Given two sub-spaces $\UU, \WW$ of dimensions $dim(\UU)=k, dim(\WW) = l$, where $k \leq l$ there exists a sequence of $k$ angles $0 \leq \beta_1 \leq \ldots \leq \beta_k \leq \pi/2$ called the principal angles and their corresponding principal pairs of vectors $(u_i, w_i)\in\UU\times\WW$ for $i=1,\ldots,k$ such that $\angle(u_i,w_i) = \beta_i$ are defined by: 
	\begin{equation*}
	\begin{array}{ll}
	    u_1, w_1 \defeq \argmin\limits_{\substack{u\in\UU, w\in\WW \\ \norm{u}=\norm{w} =1}}\arccos\left(\abs{\langle u, w \rangle}\right), &
	    \beta_1 \defeq \angle(u_1, w_1)
	\end{array}    
	,\end{equation*}
	and for $i > 1$
	\begin{equation*}
	\begin{array}{ll}
	    u_i, w_i \defeq \argmin\limits_{\substack{u\perp\UU_{i-1}, w\perp\WW_{i-1} \\ \norm{u}=\norm{w} =1}}\arccos\left(\abs{\langle u, w \rangle}\right), &
	    \beta_i \defeq \angle(u_i, w_i)
	\end{array}
	,\end{equation*}
    where $$\UU_i \defeq Span\{u_j\}_{j=1}^i ~,~ \WW_i \defeq Span\{w_j\}_{j=1}^i.$$
\end{definition}
In other words, given two linear subspaces of $\RR^D$ of the same dimension $d$ we can measure the distance between them based upon the principal angles. 
In our case, we measure the distance between two subspaces by taking the maximal principal angle (maximal angle) , and denote it as
\begin{equation}\label{eq:TheAngleLinear_clean}
    \angle_{\max}(\UU,\WW) \defeq \max_{1\leq i\leq d}\beta_i
.\end{equation}

Lemmas \ref{lem:angle_space_to_vec} and \ref{lem:angle_space_perp_space} are reformulation of results proven in \cite{knyazev2007majorization} and will be also used later on.
\begin{Lemma}\label{lem:angle_space_to_vec}
Let $F$ and $G$ be two linear spaces of dimension $d$ in $\RR^D$. Assume that $\maxangle(F,G)\leq \alpha$. Then for any vector $v \in F^\perp$
\[
\min\limits_{w\in G^\perp} \angle(v,w) \leq \alpha
\]
\end{Lemma}

\begin{Lemma}\label{lem:angle_space_perp_space}
Let $F$ and $G$ be two linear spaces of dimension $d$ in $\RR^D$. Assume that $\maxangle(F.G)\leq \alpha$. Then for any vector $v \in F^\perp$ and  $w\in G$,
\[
\angle(v,w) \geq \pi/2 - \alpha
\]
\end{Lemma}

\subsection{Viewing the manifold locally as a function graph}
It is well known that, locally, a sub-manifold of $\RR^D$ can be described as a graph of a function defined from the tangent space to its orthogonal complement.
In this section, we deal with expressing a manifold as a local function graph with respect to some tilted coordinate system and bounding the size of the neighborhood for which this definition still holds.
The results reported below are general and relate closely to the concept of the Reach (see Definition \ref{def:reach}), which was introduced by \cite{federer1959curvature} and further studied by Boissonat, Lieutier, and Wintraecken \cite{boissonnat2017reach}. 

\begin{Lemma}[Corollary 8 from \cite{boissonnat2017reach}] \label{lem:reach_ball_no_intersect_clean}
    Let $\MM$ be a sub-manifold of $\RR^D$ with reach $\tau$ and let $p\in\MM$.
    Then, any open $D$ dimensional ball of radius $\rho \leq \tau$ that is tangent to $\MM$ at $p$ does not intersect $\MM$.
\end{Lemma}

\begin{Lemma}[Bounding Ball]\label{lem:f_bound_circle_no_func_clean}
Let $\MM$ be a $d$-dimensional sub-manifold of $\RR^D$ with reach $\tau$. For any $p \in \MM$, let $T_{p}\MM$ be the tangent of $\MM$ at $p$.
    For any $x = p + x_T$, where $x_T\in T_p\MM$, and $y\in (T_p\MM)^\perp$  such that $\|x_T\|\leq \tau$, $\|y\| \leq \tau/2$ and $(x + y)\in \MM$, we have that 
    \[
    \norm{y} \leq \tau-\sqrt{\tau^2-\norm{x_T}^2}
    \]
\end{Lemma}
The proof follows directly from Lemma \ref{lem:reach_ball_no_intersect_clean}.

\begin{Lemma}[Bounding Ball with Noise]\label{lem:dist_ri_TpM_clean}
Under the sampling assumption of \ref{sec:SamplingAssumptions}. For $r_i\in U_{ROI}$ where $U_{ROI}$ is defined in \eqref{eq:ROI_clean}. Denote $p_r = Proj_\MM(r)$, and $x_i = Proj_{T_{p_r}\MM}(r_i - p_r)$. Then,
\[
\dist(r_i - p_r,T_{p_r}\MM) \leq \tau-\sqrt{\tau^2-\|x_i\|^2}  + \sigma
.\]    
\end{Lemma}
\begin{proof}
We first note that since $r$ is at distance at most $\sigma$ from $\MM$, we have that  $U_{\textrm{ROI}} \subset B_D(p_r, \sqrt{\sigma\tau}+\sigma )$. Denote $p_i = Proj_\MM(r_i)$.
By Lemma \ref{lem:f_bound_circle_no_func_clean} the distance between $r_i$ and $ T_{p_r}\MM $ is bounded by 
\[
\dist(r_i-p_r,T_{p_r}\MM) \leq \dist(p_i-p_r,T_{p_r}\MM) +\sigma \leq \tau-\sqrt{\tau^2-\|x_i\|^2}  + \sigma
.\]
\end{proof}

\begin{Lemma}[Bounding Ball From a Tilted Plane] \label{lem:f_bound_circle_H_no_func_clean}

    Let $\MM$ be a $d$-dimensional sub-manifold of $\RR^D$ with reach $\tau$. For any $p \in \MM$, let $T_{p}\MM$ be the tangent of $\MM$ at $p$ and let $ H\in Gr(d, D) $ such that $ \maxangle(T_p\MM, H) = \alpha \leq \pi/4 $.
    For any $x = p + x_H$, where $x_H\in H$, $\|x_H\|\leq c_{\pi/4}\tau$ for some constant $c_{\pi/4}$ , and $y\in H^\perp$  such that $\|x-p\|\leq \tau \cos \alpha$, $\|y\| \leq \tau/2$ and $(x + y)\in \MM$, we have that 
	\[
	-\tau\cos\alpha+\sqrt{\tau^2-(\norm{x} - \tau\sin\alpha)^2} \leq \|y\| \leq \tau\cos\alpha-\sqrt{\tau^2-(\norm{x} + \tau\sin\alpha)^2}
	\]
\end{Lemma}
The proof of Lemma \ref{lem:f_bound_circle_H_no_func_clean} follows directly from applying Lemma \ref{lem:f_bound_circle_no_func_clean} and observing the illustration in Figure~\ref{fig:BoundingCircle}.

\begin{figure}[htb]
	\centering
	\includegraphics[width=0.6\linewidth]{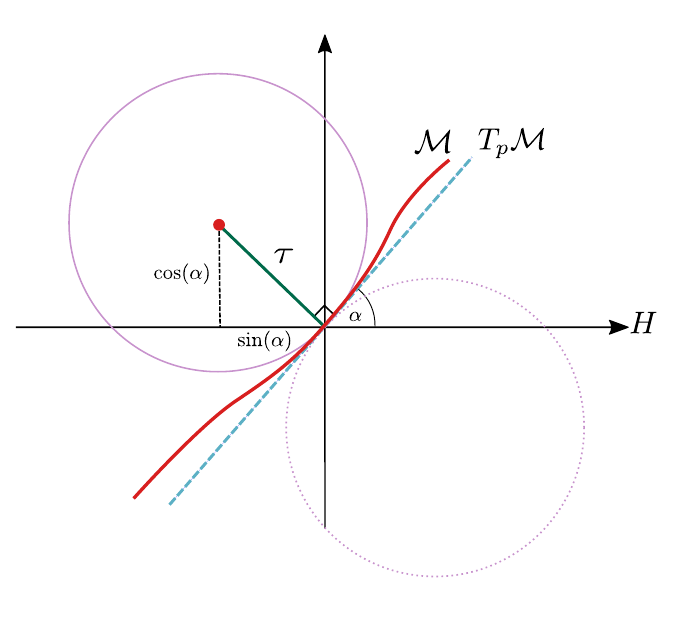}
	\caption{Illustration of bounding ball for $ d=1,~D=2 $. The manifold $ \MM $ is marked by the red solid line, $ T_\MM $ is marked by the blue dashed line and $ \alpha = \maxangle(H, T_p\MM) $ is the angle betweeen the $ x $-axis and $ T_p\MM $. The bounding balls defined by the reach $ \tau $ are marked in pink \emphbrackets{solid and dotted}  }
	\label{fig:BoundingCircle}
\end{figure}

\begin{Lemma}\label{lem:f_bound_circle_H_no_func_ver2_clean}
    Let $\MM$ be a $d$-dimensional sub-manifold of $\RR^D$ with reach $\tau$. For any $p \in \MM$, let $T_{p}\MM$ be the tangent of $\MM$ at $p$ and let $ H\in Gr(d, D) $ such that $ \maxangle(T_p\MM, H) = \alpha  \leq \pi/4$.
    For any $x = p + x_H$, where $x_H\in H$,$\|x_H\|\leq c_{\pi/4}\tau$ for some constant $c_{\pi/4}$, and $y\in H^\perp$  such that $\|x-p\|\leq \tau \cos \alpha$, $\|y\| \leq \tau/2$ and $(x + y)\in \MM$, we have that 
	\[
	\|y\| \leq \norm{x_H}(\tan \alpha  +c\|x_H\|/\tau)
	\]
 where $c$ is some constant depending on $c_{\pi/4}$.
\end{Lemma}
\begin{proof}
 Recalling Lemma \ref{lem:f_bound_circle_H_no_func_clean}, we have that 
	\begin{equation}
	\|y\| \leq \tau \cos \alpha - \sqrt{\tau^2-(\norm{x_H}+\tau \sin \alpha)^2}
	.\end{equation}
	Therefore,
	\begin{align*}
	\|y\| &\leq \tau (\cos \alpha - \sqrt{1-(\norm{x_H}/\tau + \sin \alpha)^2})\\
	      &=  \tau (\cos \alpha - \sqrt{1-\norm{x_H}^2/\tau^2  + 2\norm{x_H}/\tau\sin \alpha - \sin^2\alpha})\\
	      &=  \tau (\cos \alpha - \sqrt{\cos^2\alpha-\norm{x_H}^2/\tau^2 - 2\norm{x_H}/\tau\sin \alpha })\\
	      &=  \tau \cos \alpha(1 - \sqrt{1-\norm{x_H}^2/(\tau^2\cos^2\alpha)  - 2\norm{x_H}/\tau\tan \alpha / \cos\alpha })\\
	      &\leq  \tau \cos \alpha (1 - (1-\frac{\norm{x_H}^2}{2\tau^2\cos^2\alpha}  - \frac{\norm{x_H}\tan \alpha}{\tau \cos\alpha}  - \left(\frac{\norm{x_H}^2}{\tau^2\cos^2\alpha}  + 2\frac{\norm{x_H}\tan \alpha}{\tau\cos\alpha} \right)^2))\\
	      &= \tau   \cos\alpha(\frac{\norm{x_H}^2}{2\tau^2\cos^2\alpha}  + \frac{\norm{x_H}\tan \alpha}{\tau\cos\alpha}  + \left(\frac{\norm{x_H}^2}{\tau^2\cos^2\alpha}  + 2\frac{\norm{x_H}\tan \alpha}{\tau\cos\alpha} \right)^2)\\
	      &= \norm{x_H}\tan \alpha  +\frac{\norm{x_H}^2}{2\tau\cos\alpha}  +  \left(\frac{\norm{x_H}^2}{\tau^{3/2}\cos^{3/2}\alpha}  + 2\frac{\norm{x_H}\tan \alpha}{\tau^{1/2}\cos^{1/2}\alpha} \right)^2\\
	      &\leq \norm{x_H} (\tan \alpha + c\|x_H\|/\tau)
	\end{align*}
	for some constant $c$ depending only on $c_{\pi/4}$.
\end{proof}

\begin{Lemma}\label{lem:alpha-beta_x_no_func}
	Let $\MM$ be a $d$-dimensional sub-manifold of $\RR^D$ with reach $\tau$. For any $p \in \MM$, let $T_{p}\MM$ be the tangent of $\MM$ at $p$ and let $ H\in Gr(d, D) $ such that $ \maxangle(T_p\MM, H) = \alpha\leq \pi/4 $.
	For any $x = p + x_H$, where $x_H\in H$,$\|x_H\|\leq c_{\pi/4}\tau$ for some constant $c_{\pi/4}$, and $y\in H^\perp$  such that $\|x-p\|\leq \tau \cos \alpha$, $\|y\| \leq \tau/2$ and $(x + y)\in \MM$.
	Then, we get that
	\[
	\sin(\maxangle (T_{x+y}\MM, T_{p}\MM)) \leq \frac{\norm{x_H}}{\tau}(1 + \tan^2 \alpha)  +\OO(\|x_H\|^2/\tau^2)
	\]
\end{Lemma}
\begin{proof}
From  Corollary 3 in \cite{boissonnat2017reach} bounds the maximal angle between the tangent spaces at two points on $\MM$ through their Euclidean distance and $\tau$.
	Namely, let $p_1, p_2 \in \MM$
	\[
	sin(\maxangle(T_{p_1}\MM, T_{p_2}\MM)/2) \leq\frac{\norm{p_1 - p_2}}{2\tau}
	.\]
	Therefore, in our case we obtain,
	\begin{equation}\label{eq:sin_alpha_1}
	\sin(\maxangle (T_{x+y}\MM, T_{p}\MM)/2) \leq \frac{\sqrt{\norm{x_H}^2 + \|y\|^2}}{2\tau}
	.\end{equation}
		
	Recalling Lemma \ref{lem:f_bound_circle_H_no_func_ver2_clean}, we have that 
	\begin{equation*}
	\|y\| \leq \norm{x_H}\tan \alpha  +\OO(\|x_H\|^2/\tau)
	\end{equation*}
	
	\begin{align*}
	\sin(\maxangle (T_{x+y}\MM, T_{p}\MM)/2) &\leq \frac{\sqrt{\norm{x_H}^2 + \|y\|^2}}{2\tau}    \\
	&= \frac{\sqrt{\norm{x_H}^2 +  \norm{x_H}^2\tan^2 \alpha  +\OO(\|x_H\|^2/\tau)}}{2\tau}\\
	&=\frac{\norm{x_H}}{2\tau}\sqrt{1 + \tan^2 \alpha  +\OO(\|x_H\|/\tau)}\\
	&\leq\frac{\norm{x_H}}{2\tau}(1 + \tan^2 \alpha)  +\OO(\|x_H\|^2/\tau^2)
	\end{align*}
	Note that for sufficiently small $\gamma$ we can bound $sin(\gamma) < 2\sin(\gamma/2)$. 
    Thus, for $\tau$ large enough we obtain
	\[
	\sin(\maxangle (T_{x+y}\MM, T_{p}\MM)) \leq \frac{\norm{x_H}}{\tau}(1 + \tan^2 \alpha)  +\OO(\|x_H\|^2/\tau^2)
	\]
\end{proof}

\begin{Lemma}\label{lem:alpha_beta_x_no_func}
	Let $\MM$ be a $d$-dimensional sub-manifold of $\RR^D$ with reach $\tau$. For any $p \in \MM$, let $T_{p}\MM$ be the tangent of $\MM$ at $p$ and let $ H\in Gr(d, D) $ such that $ \maxangle(T_p\MM, H) = \alpha\leq \pi/4 $.
	For any $x = p + x_H$, where $x_H\in H$,$\|x_H\|\leq c_{\pi/4}\tau$ for some constant $c_{\pi/4}$, and $y\in H^\perp$  such that $\|x-p\|\leq \tau \cos \alpha$, $\|y\| \leq \tau/2$ and $(x + y)\in \MM$, we denote $ \beta = \maxangle(T_{x+y}\MM, H) $.
	Then, we get that
        \[
      \alpha- 2\frac{\norm{x_H}}{\tau}(1 + \alpha)  +c\|x_H\|^2/\tau^2 \leq \beta \leq \alpha+ 2\frac{\norm{x_H}}{\tau}(1 + \alpha)  +c\|x_H\|^2/\tau^2
	 ,\]
  for some constant $c\in \RR$.

\end{Lemma}
\begin{proof}
	Using the triangle inequality for maximal angles we have that 
	\[
	\maxangle(T_p \MM, H) \leq \maxangle(T_p \MM,T_{x+y} \MM)+\maxangle(T_{x+y} \MM,H),\]
 as well as
	\[ \maxangle(T_{x+y} \MM,H)
	 \leq \maxangle(T_p \MM,T_{x+y} \MM)+
  \maxangle(T_p \MM, H),\]

 which can be written as 
	\[
	|\alpha - \beta| \leq \maxangle(T_{x+y} \MM, T_p \MM)
	.\]
	From Lemma~\ref{lem:alpha-beta_x_no_func} we have that 
	\[
	\sin(\maxangle (T_{x+y}\MM, T_{p}\MM)) \leq \frac{\norm{x_H}}{\tau}(1 + \tan^2 \alpha)  +\OO(\|x_H\|^2/\tau^2)
	,\]
	and thus, 
	\[
	\sin(|\alpha - \beta|) \leq \frac{\norm{x_H}}{\tau}(1 + \tan^2 \alpha)  +\OO(\|x_H\|^2/\tau^2)
	\]
	Since for $x\leq \pi/2$, we have that $x/2<\sin(x)$, we have 
	
	\[
	|\alpha - \beta| \leq 2\frac{\norm{x_H}}{\tau}(1 + \tan^2 \alpha)  +\OO(\|x_H\|^2/\tau^2)
	.\]
	Since for $\alpha \leq \pi/4$, we have that $\tan \alpha \leq 1$, and $\tan^2\alpha \leq \alpha$, we have that 
    \[
      \alpha- 2\frac{\norm{x_H}}{\tau}(1 + \alpha)  +\OO(\|x_H\|^2/\tau^2) \leq \beta \leq \alpha+ 2\frac{\norm{x_H}}{\tau}(1 + \alpha)  +\OO(\|x_H\|^2/\tau^2)
	.\]

\end{proof}

\begin{Lemma} [$\MM$ is locally a function graph over a tilted plane]\label{lem:M_is_locally_a_fuinction_clean}
Let $\MM$ be a $d$-dimensional sub-manifold of $\RR^D$ with reach $\tau$. For any $p \in \MM$, let $T_{p}\MM$ be the tangent of $\MM$ at $p$. Let $H \in Gr(d,D)$, such that $\maxangle(H,T_p\MM) = \alpha \leq \pi/4$.  
Then $\MM \cap \emph{Cyl}_H(p,\rho, \tau/2)$ is locally a function over $H$, where $\emph{Cyl}_H(p,\rho, \tau/2)$ is the $D$-dimensional cylinder with the base $B_H(p, \rho)\subset H$ and height $\tau/2$ in any direction on $H^\perp$, for $\rho = c_{\pi/4}\tau$ where $c_{\pi/4}$ is some general constant. 

\noindent Explicitly, there exists a function 
$$
f:B_H(p, \rho) \to H^\perp
$$
such that the graph of $f$ defined as
$$
\Gamma_f = \{p + (x,f(x)) | x\in B_H(p, \rho) \} 
$$
identifies with $\MM \cap \emph{Cyl}_H(p,\rho, \tau/2)$.
\end{Lemma}
\begin{proof}
We split our arguments to two separate parts.
First, we show that for $c < 0.02$ there exists a function $f$ such that $\Gamma_f \subset \MM\cap \textrm{Cyl}_H(p,  c \tau, \tau/2)$.
Then, in the second part of the proof, we show that there is a constant $c_{\pi/4} < c$ such that $f$ is defined uniquely and $\Gamma_f = M\cap \textrm{Cyl}_H(p, \rho, \tau/2)$.

By definition, there is an open ball $U_T\subset T_p\MM$ of $p$ such that there is a neighborhood $W_\MM \subset \MM$ that can be written as a graph of a function from $U_T\simeq \RR^d$ to $T_p\MM^\perp\simeq \RR^{D-d}$.
Accordingly, for any $H$ such that $\maxangle(H, T_p\MM) < \pi/2$ there is an open ball $B_H(p,\eps) \subset H$ such that $W_H \subset \MM$ can be written as a graph of a function $f$  from $B_H(p,\eps)$ to $H^\perp$.
We wish to show that $f$ can be extended to a ball $B_H(p, 0.02 \tau)\subset H$ such that the graph of $f$ is a subset of $\MM$ (note that $f(0) = 0$).

By contradiction, let us assume that $\frak{r}$ the maximal radius of an open ball such that the $\Gamma_f \subset \MM$, is strictly smaller than $0.02 \tau$.
We claim that the graph $\Gamma_f$ is defined on the closed ball $\bar B_H(p, \frak{r})$ and is also subset of $\MM$.
This is true, from the following argument:
Take a sequence of points $\{x_n\}$ converging to $ x\in \partial \bar B_H(p, \frak{r})$, a point on the boundary of $\bar B_H(p, \rho)$, and consider $\{y_n = p + (x_n, f(x_n))\in \MM\}$. From the compactness of $\MM$ the sequence $y_n$ has a converging subsequence $y_{n_k}$ and we denote its limit as $y$.
Since $x_n \to x$, we define $f(x)\defeq\lim f(x_{n_k})$ and  $y = p + (x, f(x))$.

We now wish to show that there is $\eps>0$ such that $f$ can be extended to $B_H(p, \frak{r} + \eps)$.
Using a similar argument to the one used in the beginning of the proof, by showing that for any $x\in\partial B_H(p, \frak{r})$ the angle $\maxangle(T_x f, H) < \pi/2$, we get that there is $W_H\subset\MM$, a neighborhood of $y\in \MM$ that is the image of some function from $B_H(x,\eps_x)$ to $H^\perp$.
Therefore, $f$ can be extended into this neighborhood.
Taking $\eps$ to be the minimum over all $\eps_x$, which exists since $x$ is in $H(p, \frak{r})$, which is compact, we get that $f$ can be extended to $B_H(p, \frak{r}+ \eps)$.

The remaining piece of the existence puzzle is showing that for all $x\in\partial B_H(p, \frak{r})$ we have $\maxangle(T_x f, H) < \pi/2$.
From Lemma \ref{lem:alpha_beta_x_no_func} we have that for any $x$ such that 
\[
       \frac{\pi}{2} \leq \frac{\pi}{4}+ 2\frac{\norm{x}}{\tau}(1 + \frac{\pi}{4})  +c\|x\|^2/\tau^2
	 .\]

$\maxangle(T_xf,H) <\pi/2$ holds.
Rewriting the inequality we get
\[
\begin{array}{cc}
      \\
    \frac{\pi}{4} > & 2\frac{\norm{x}}{\tau}(1 + \frac{\pi}{4})  +c\|x\|^2/\tau^2\\
    0 > & c\frac{\norm{x}^2}{\tau^2} +  \frac{\norm{x}}{\tau} (1 + \frac{\pi}{4}) - \frac{\pi}{8},\\
\end{array}
\]
Thus, for $x< c_1 \tau$ we have that  $\maxangle(T_xf,H) <\pi/2$ holds.

We now turn to show that there is a constant $c_{\pi/4}$ for which $f$ is uniquely defined in $B_H(p, c_{\pi/4}\tau)$.
From Lemma \ref{lem:f_bound_circle_H_no_func_clean} we know that for any $x \in H$ with $\|x\|\leq \tau/2$ all the $y\in H^\perp$ such that $(x,y)\in\MM$ and $\|y\| \leq \tau/2$ must satisfy:
\begin{equation}
    \|y\| \leq \tau\cos\alpha-\sqrt{\tau^2-(\norm{x} + \tau\sin\alpha)^2} = \tau\left(\cos\alpha-\sqrt{1-(\frac{\norm{x}}{\tau} + \sin\alpha)^2}\right)
\end{equation}
Let $y_1, y_2$ be such that $(x,y_1),(x,y_2)\in\MM$ where $\norm{x}=\bar x\tau$.
Then,
\begin{equation}
    \|y_j\| \leq \tau\left(\cos\alpha-\sqrt{1-(\bar x+ \sin\alpha)^2}\right),\quad (j=1,2)
.\end{equation}
In other words, $y_1$ and $y_2$ cannot be too far from one another, and note that as $\bar x \to 0$
\begin{equation}\label{eq:y2_y1_diff}
\norm{y_2 - y_1} \to 0
\end{equation}

On the other hand, taking the point $(x, y_1)\in \MM$, denoting $\beta = \maxangle (T_{(x, y_1)}\MM, H)$, and applying Lemma \ref{lem:alpha_beta_x_no_func} we have that
\[
\beta \leq \alpha+ 2\frac{\bar x}{\tau}(1 + \alpha)  +\OO(\bar x^2/\tau^2)
\]
which tends to $\alpha$ when $\bar x\to 0$.
From Lemma \ref{lem:reach_ball_no_intersect_clean} we know that $(x,y_2)$ cannot be in any ball tangent to $\MM$ at $(x, y_1)$ of radius $\tau$.
We denote by $v$ the direction $(0, y_2 - y_1)\in H^\perp$. 
From Lemma \ref{lem:angle_space_to_vec}  we know that there is $w \in \lrbrackets{ T_{(x, y_1)} \MM} ^\perp$ such that $\angle (v, w) \leq \beta$.
Therefore, we can now limit our discussion to $L_{y_1}$ the affine space spanned by $v$ and $w$ from $(x, y_1)$, and note that it contains $(x, y_2)$ as well.
Taking the two balls $B_D((x, y_1) \pm \tau\cdot w, \tau)$ and intersecting them with $L_{y_1}$ we get two 2-dimensional disks of radius $\tau$ (see Figure \ref{fig:bouding_balls_for lem_locally_func}).
Thus, $(x, y_2)$ cannot be within either disks.
From basic trigonometry we achieve that either $y_2 = y_1$ or
\[
\norm{y_2 - y_1} \geq 2\tau\cos(\beta) \geq 2\tau\cos(\alpha + 2\sqrt{\bar x (2\alpha + \bar x)})
,\]
which tends to $2\tau \cos \alpha$ as $\bar x \to 0$.
Combining this with \eqref{eq:y2_y1_diff} we get there is $c_{\pi/4}$  such that for all $\bar x \leq c_{\pi/4}$ we have $y_1 = y_2$.
\begin{figure}[htb!]
    \centering
    \includegraphics
 [width=0.45\textwidth]{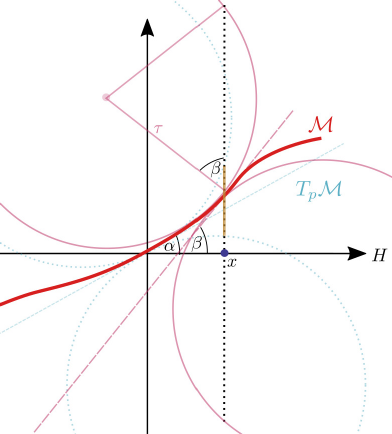}
    \caption{Illustration of bounding balls. The reach $\tau$ of the manifold $\MM$ (marked by the red line) bounds its sectional curvatures. Namely, the manifold cannot intersect a tangent open ball of radius $\tau$. In this illustration we have some coordinate system $H$ and the manifold can be described locally as a graph of some function $f:H\to H^\perp$. The coordinate system $H$ is not aligned with $T_0f$, the tangent at zero; i.e., $\maxangle(H, T_0f) = \alpha$. Then the value of $f$ at $x$ is bounded to the markered interval in $H^\perp$ above $x$. Furthermore, in order to have two different points in $\MM$ above $x$, the manifold cannot curve too fast as it cannot enter neither the dotted balls nor the solid ones.}
    \label{fig:bouding_balls_for lem_locally_func}
\end{figure}

\end{proof}

\begin{corollary} \label{cor:GraphOfFunctionTau2_clean}
Under the requirements of Lemma \ref{lem:M_is_locally_a_fuinction_clean} we get that $\MM$ is a function graph over $T_{p}\MM$ in a ${\tau}/2$ neighborhood of $p$. 
Reiterating \eqref{eq:phi_def_clean}, we have a function 
\[
\phi_p:B_{T_p\MM}(0,\tau/2) \rightarrow T_p\MM^\perp
\]
such that the graph of $\phi$ shifted to $p$ coincides with $\MM \cap \mathrm{Cyl}_{T_p\MM}(p,\tau/2,\tau/2)$.

\end{corollary}

\begin{corollary}\label{cor:GraphOfFunctionTau_ROI}
Let the requirements of Lemma \ref{lem:M_is_locally_a_fuinction_clean}, and the sampling assumptions of Section \ref{sec:SamplingAssumptions} hold. Denote the projection of $r$ onto $\MM$ by $p_r = Proj_\MM(r)$, let $U_{\textrm{ROI}}$ be as defined in \eqref{eq:ROI_clean}. 
Then, any $r_i \in U_{\textrm{ROI}}$ can be written as
\begin{equation}\label{eq:FunctionGraph_pi_ei_clean}
r_i = \underbrace{p_r + (x_i, \phi_{p_r}(x_i))_{T_{p_r}\MM}}_{p_i}+\eps_i
,\end{equation}
where $x_i= Proj_{T_{p_r}\MM}(r_i - p_r)$ and $\|\eps_i\| \leq \sigma$.
\end{corollary}
\begin{proof}
From the assumptions of Section \ref{sec:SamplingAssumptions} we know that  $\tau/\sigma > M $.
Therefore, there is $M$ such that $\sqrt{\sigma\tau}+\sigma < \tau/2$ and thus, from Corollary \ref{cor:GraphOfFunctionTau2_clean}, the intersection of $\MM$ with $\textrm{Cyl}_{T_{p_r}\MM}(p_r,\sqrt{\sigma \tau}+\sigma, \tau/2)$, a cylinder with base $B_{T_{p_r}\MM}(p_r, \sqrt{\sigma\tau}+\sigma)\subset T_{p_r}\MM$ and heights $\tau/2$ in $T_{p_r}\MM^\perp$ can be written as $\Gamma_{\phi_{p_r}, B_{T_{p_r}\MM}(p_r, \sqrt{\sigma\tau} + \sigma)}$, the graph of $\phi_{p_r}:T_{p_r}\MM\to T_{p_r}\MM^\perp$. Since $r_i$ are in a tubular neighborhood of $\MM$, the proof is concluded. 
\end{proof}

\begin{Lemma}[Function version of Lemma \ref{lem:f_bound_circle_no_func_clean}]\label{lem:f_bound_circle_clean}
    Let $\MM$ be a $d$-dimensional sub-manifold of $\RR^D$ with reach $\tau$. For any $p \in \MM$, let $T_{p}\MM$ be the tangent of $\MM$ at $p$.
    Let $\phi_p:B_{T_p\MM}(0, \tau/2)\to T_p\MM^\perp$ be defined as in Corollary \ref{cor:GraphOfFunctionTau2_clean}; that is, 
    \[\Gamma_{\phi_p,B_{T_p\MM}(0, \tau/2)}\subset\MM,\]
    where
    \[\Gamma_{\phi_p,B_{T_p\MM}(0, \tau/2)} = \{p + (x, \phi_p(x)) | x\in B_{T_p\MM}(0, \tau/2)\}.\]
    Then, for any $v \in T_p\MM^\perp$, such that $\norm{v} = 1$
    \[
    \lrangle{v, \phi_p(x)} \leq \tau-\sqrt{\tau^2-\norm{x}^2}
    \]
\end{Lemma}
\begin{proof}
    This follows immediately from Lemma \ref{lem:f_bound_circle_no_func_clean} and Lemma \ref{cor:GraphOfFunctionTau2_clean}.
\end{proof}

\begin{corollary}
It follows immediately from Lemma \ref{lem:f_bound_circle_clean} 
\begin{equation}\label{eq:eq:bound_phix_clean}
    \norm{\phi_p(x)}_{\RR^{D-d}} \leq \tau - \sqrt{\tau^2 - \norm{x}_{\RR^d}^2}
,\end{equation}
and using the triangle inequality we can say that
\begin{equation}\label{eq:bound_x_phix_clean}
    \norm{(x,\phi_p(x))}_{\RR^{D}} \leq \norm{x}_{\RR^d} + \tau - \sqrt{\tau^2 - \norm{x}_{\RR^d}^2}
.\end{equation}
\end{corollary}

\begin{Lemma}[Function version of Lemma \ref{lem:f_bound_circle_H_no_func_clean}]\label{lem:f_bound_circle_H_clean}
    Let $\MM$ be a $d$-dimensional sub-manifold of $\RR^D$ with reach $\tau$. For any $p \in \MM$, let $T_{p}\MM$ be the tangent of $\MM$ at $p$ and let $ H\in Gr(d, D) $ such that $ \maxangle(T_p\MM, H) = \alpha \leq \pi/4$.
	Let $f_p:\RR^d\to \RR^{D-d}$ be such that the neighborhood $W_p\subset\MM$ can be descried as the the graph of $f_p$
	\[\Gamma_{f_p,W_p} = \{p + (x, f_p(x)) | x\in Proj_{H}(W_p)\}\]
	Then, for any $v \in H^\perp$, such that $\norm{v} = 1$
	\[
	-\tau\cos\alpha+\sqrt{\tau^2-(\norm{x} - \tau\sin\alpha)^2} \leq \lrangle{v, f_p(x)} \leq \tau\cos\alpha-\sqrt{\tau^2-(\norm{x} + \tau\sin\alpha)^2}
	\]
\end{Lemma}
\begin{proof}
    This follows directly from Lemma \ref{lem:f_bound_circle_H_no_func_clean} and Lemma \ref{lem:M_is_locally_a_fuinction_clean}.
\end{proof}

\begin{Lemma}[Function Version of Lemma \ref{lem:alpha_beta_x_no_func}]\label{lem:alpha_beta_x}
    Let $\MM$ be a $d$-dimensional sub-manifold of $\RR^D$ with reach $\tau$. 
    For any $p \in \MM$, let $T_{p}\MM$ be the tangent of $\MM$ at $p$ and let $ H\in Gr(d, D) $ such that the origin of $H$ is set at $p$ and $ \maxangle(T_p\MM, H) = \alpha \leq \pi/4$.
	Let $f_p:\RR^d\to \RR^{D-d}$ be such that the neighborhood $W_p\subset\MM$ can be descried as the the graph of $f_p$
	\[\Gamma_{f_p,W_p} = \{p + (x, f_p(x)) | x\in Proj_{H}(W_p)\}\]
    Let $ x_0\in H $ such that $\|x_0\|\leq c_{\pi/4}\tau$, $ \beta(x_0) = \maxangle(T_{x_0}f_p, H) $, where $ T_\xi f $ is the tangent to the graph of $ f_p $ at $ f(\xi) $.
	Then, we get
    \[
      \alpha- 2\frac{\norm{x_0}}{\tau}(1 + \alpha)  +c\|x_0\|^2/\tau^2 \leq \beta(x_0) \leq \alpha+ 2\frac{\norm{x_0}}{\tau}(1 + \alpha)  +c\|x_0\|^2/\tau^2
	 ,\]
    for some constant $c\in \RR$.
\end{Lemma}

\begin{proof}
 This follows immediately from Lemma \ref{lem:alpha_beta_x_no_func} and Lemma \ref{lem:M_is_locally_a_fuinction_clean}.
\end{proof}

\section{Supporting lemmas for Step 1}

\subsection{Proof of Lemma \ref{lem:J1pTp_clean}} \label{subsec:proof_lem_J1pTp_clean}

\begin{proof}
We recall that $ U_\textrm{ROI} $ as defined in \eqref{eq:ROI_clean} is contained in $ B_D({p_r}, \sqrt{\sigma\tau} + \sigma) $.
Using Lemma \ref{lem:dist_ri_TpM_clean}, for any $r_i\in U_\textrm{ROI}$ we know that the distance between $r_i - p_r$ and $ T_{{p_r}}\MM $ is bounded by 
\begin{equation}\label{eq:d_ri_Tp_clean}
\dist(r_i - p_r,T_{{p_r}}\MM) \leq \tau-\sqrt{\tau^2-\|x_i\|^2}  + \sigma
,\end{equation}
where $ x_i = Proj_{T_{p_r}\MM}(r_i - p_r) $.
Since $ x_i \in B_D(0, \sqrt{\sigma\tau}+\sigma )$ we have that $\norm{x_i} \leq \sqrt{\sigma\tau}+\sigma $, and so
\[
J_1(r; {p_r}, T_{{p_r}}\MM) \leq \left(\tau-\sqrt{\tau^2-(\sqrt{\sigma\tau}+\sigma)^2} + \sigma\right)^2 
.\]
By simplifying and bounding this expression using Remark \ref{rem:taylor2_sqrt1-x2_clean} we get
	\begin{align*}
		\left((\sigma+\tau)-\sqrt{\tau^2-(\sqrt{\sigma\tau}+\sigma)^2}\right)^2 = & (\sigma+\tau)^2 + \tau^2-(\sqrt{\sigma\tau}+\sigma)^2 - 2(\sigma+\tau) \sqrt{\tau^2-(\sqrt{\sigma\tau}+\sigma)^2}\\
		= &
		\sigma^2 + 2\sigma\tau + 2 \tau^2 -  (\sqrt{\sigma\tau}+\sigma)^2 \\
		&- 2\tau(\sigma+\tau) \sqrt{1-(\sqrt{\sigma/\tau}+\sigma/\tau)^2}\\
		\leq &
		\sigma^2 + 2\sigma\tau + 2 \tau^2 - (\sqrt{\sigma\tau}+\sigma)^2\\
		& - 2\tau^2 (1-1/2 (\sqrt{\sigma/\tau}+\sigma/\tau)^2- (\sqrt{\sigma/\tau}+\sigma/\tau)^4 )\\& - 2\tau\sigma (1-1/2 (\sqrt{\sigma/\tau}+\sigma/\tau)^2- (\sqrt{\sigma/\tau}+\sigma/\tau)^4)\\
		= &
		\sigma^2 -  \tau^2(\sqrt{\sigma/\tau}+\sigma/\tau)^2 \\
		& + \tau^2  ((\sqrt{\sigma/\tau}+\sigma/\tau)^2  +2(\sqrt{\sigma/\tau}+\sigma/\tau)^4) \\& + \tau\sigma  ((\sqrt{\sigma/\tau}+\sigma/\tau)^2+2(\sqrt{\sigma/\tau}+\sigma/\tau)^4) \\
		= &
		\sigma^2 + 2\tau^2 (\sqrt{\sigma/\tau}+\sigma/\tau)^4\\& + \tau\sigma ( (\sqrt{\sigma/\tau}+\sigma/\tau)^2 +2(\sqrt{\sigma/\tau}+\sigma/\tau)^4)\\
		\leq &
		\sigma^2 + 2 \left(\sqrt{\sigma} + \frac{\sigma}{\sqrt{\tau}}\right)^4
		+ 3\sigma
		\left(\sqrt{\sigma} + \frac{\sigma}{\sqrt{\tau}}\right)^2 
	.\end{align*}
Using the fact that $ \sigma \leq \tau $ we get
\begin{align*}
J_1(r; {p_r}, T_{{p_r}}\MM) &\leq
\sigma^2 + 2(2\sqrt{\sigma})^4 + 3\sigma (2\sqrt{\sigma})^2
\end{align*}
Thus, we obtain
$$J_1(r; {p_r}, T_{{p_r}}\MM) \leq 49 \sigma^2.$$
\end{proof}

\subsection{Proof of Lemma \ref{thm:J1pH_clean}}\label{sec:proof_thmJ1pH_clean}

\begin{proof}
We first wish to denote by $\beta_j$ the principal angles between $H$ and $T_{p_r}\MM$ and their matching principal pairs $(u_j, w_j)\in T_{p_r}\MM\times H$ (see Definition \ref{def:principal_angles_clean}).  
Throughout the proof we work on the sectional planes defined by $\LL_j = Span\{u_j,w_j\}$.
Thus, we can define the orthogonal complement of $u_j$ and $w_j$ on $\LL_j$ by $y_j$ and $\tilde y_j$ correspondingly.
That is, both $\{u_j, y_j\}$ and $\{w_j, \tilde y_j\}$ are orthogonal bases of $\LL_j$.
Since for any $i \neq j$ we have that $u_i \perp w_j$, $u_i \perp u_j$ and $w_i \perp w_j$, we have that both $\{\tilde y_j\}_{j=1}^{d}$ and $\{y_j\}_{j=1}^{d}$ are orthonormal sets.
Then, complete the sets $\{\tilde y_j\}_{j=1}^{d}$ and $\{y_j\}_{j=1}^{d}$  to an orthonormal basis of $H^\perp$  and $T_{p_r}\MM^\perp$ through adding the orthonormal sets $\{\tilde y_j\}_{j=d+1}^{D-d}$ and $\{y_j\}_{j=d+1}^{D-d}$ correspondingly. Note that, since $H^\perp$ is a $D-d$ dimensional space, and $\{\wtilde y_j\}_{j=1}^{d} \in H^\perp$, we need only to add $D-2d$ vectors to have an orthonormal basis.
Explicitly, we know that for all $j' = 1\ldots, d$
$$\textrm{Span}\{\tilde y_j\}_{j=d+1}^{D-d}\perp \LL_{j'} ~~,~~ \textrm{Span}\{ y_j\}_{j=d+1}^{D-d}\perp \LL_{j'}.$$ 
Thus, for $j=d+1, \ldots, D-d$ we have both $\tilde y_j , y_j\in H^\perp \cap T_{p_r}\MM^\perp$ and without limiting the generality we can choose $\tilde y_j = y_j$ for such $j$.
Using this notation we get that for any point $x\in U_\textrm{ROI}$
\[
\dist^2(x - {p_r}, H) = 
\sum_{j=1}^{D-d} \lrangle{x - {p_r}, \tilde y_j}^2 =
\sum_{j=1}^{D-d} \lrangle{x - {p_r}, \tilde y_j}^2 + \sum_{j=1}^{D-d} \left[\lrangle{x - {p_r},  y_j}^2 - \lrangle{x - {p_r},  y_j}^2\right] =
\]
\[
=\sum_{j=1}^{D-d} \lrangle{x - {p_r}, y_j}^2 + \sum_{j=1}^{D-d}\left[ \lrangle{x - {p_r},  \tilde y_j}^2 - \lrangle{x - {p_r},  y_j}^2\right]
,\]
and since $\tilde y_j = y_j$ for $j=d+1,\ldots,D-d$ we have
\begin{align}\label{eq:EJ_1_clean}
	\dist^2(x - {p_r}, H) &=
	\dist^2(x - {p_r}, T_{{p_r}}\MM) + \sum_{j=1}^{d}\left[ \langle x - {p_r}, \tilde{y}_j \rangle^2 - \langle x - {p_r}, {y}_j \rangle^2\right]
.\end{align}

The remainder of the proof is achieved through the following set of claims:
\begin{enumerate}
    \item\label{thm:claim_J1H_clean} From \eqref{eq:EJ_1_clean} it follows that
    \[J_1(r; {p_r}, H) =  J_1(r;{p_r},T_{p_r}\MM) +  R_1(r;{p_r},H),\]
    where
    \begin{equation}\label{eq:sum_k_clean}
    R_1(r; {p_r}, H) = \sum\limits_{j=1}^{d} \frac{1}{\#|U_\textrm{ROI}|}\sum\limits_{r_i\in U_\textrm{ROI}}  \langle r_i - {p_r}, \tilde{y}_j \rangle^2 - \langle r_i - {p_r}, {y}_j \rangle^2
    .\end{equation}
    \item Thus, in order to bound $J_1(r; {p_r}, H)$ from below we can focus on bounding $R_1(r; {p_r}, H)$.
    We then consider separately two sets of indices $\KK', \KK''$ such that $\KK' \cup \KK'' = \{1,\ldots, d\}$, where for all $j\in \KK'$ we have $\beta_j > \alpha$ and for all $j\in \KK''$ we have $\beta_j\leq \alpha$.
    Notice that since $\angle_{max}(H, T_{p_r}\MM) > \alpha$ we, get that $\KK' \not = \emptyset$ and that $\#\KK'' \leq d-1$.
    Writing this explicitly we get
    \begin{equation}\label{eq:RpRpp_clean}
        R_1(r; {p_r}, H) = R_1'(r; {p_r}, H)  +  R_1''(r; {p_r}, H)
    \end{equation}
    where 
    \begin{equation}\label{eq:R1p_def_clean}
    R_1'(r; {p_r}, H) = \sum\limits_{j\in\KK'} \frac{1}{\#|U_\textrm{ROI}|}\sum\limits_{r_i\in U_\textrm{ROI}}  \langle r_i - {p_r}, \tilde{y}_j \rangle^2 - \langle r_i - {p_r}, {y}_j \rangle^2
    ,\end{equation}
    and
    \begin{equation}\label{eq:R1pp_def_clean}
    R_1''(r; {p_r}, H) = \sum\limits_{j\in\KK''} \frac{1}{\#|U_\textrm{ROI}|}\sum\limits_{r_i\in U_\textrm{ROI}}  \langle r_i - {p_r}, \tilde{y}_j \rangle^2 - \langle r_i - {p_r}, {y}_j \rangle^2
    .\end{equation}
    \item\label{thm:claim_R1pp_bound_clean} We show in Section \ref{sec:proof_claim_R1pp_bound_clean} that
    \begin{equation}\label{eq:R1pp_bound_clean}
    R_1''(r; {p_r}, H) \geq
    -9 \sigma^2
    .\end{equation}
    \item\label{thm:claim_ri_tyk_bound_clean} Then, when we focus on $j\in\KK'$ we show that for $a_1 = \frac{1}{4}$ and $a_2 = \frac{1}{8}$ and $x_0\in T_{p_r}\MM$ such that $\norm{{p_r} - x_0} = a_1\sqrt{\sigma\tau}$ and $B_{T_{p_r}\MM}(x_0, a_2 \sqrt{\sigma\tau} + \sigma)\subset B_D({p_r}, \sqrt{\sigma \tau} - \sigma) \subset U_\textrm{ROI}$, and 
    \begin{equation}\label{eq:ri_tyk_bound_clean}
        \lrangle{r_i - {p_r}, \tilde y_j}^2 \geq \frac{1}{64}\sigma \tau sin\alpha - \frac{1}{2}\sigma^{3/2}\tau^{1/2}
    .\end{equation}
    This is proven in \ref{subsec:proof_lem_ri_tyk_bound_clean}.
    \item\label{thm:claim_useStochasticLemma} From Assumption \ref{assume:noise} in Section \ref{sec:SamplingAssumptions} we know that $ \sqrt{\frac{\sigma}{\tau}} < \frac{1}{2} $. 
    In addition, setting $ \rho = a_2 \sqrt{\sigma \tau} $ we know that $ B_{T_{p_r}\MM}(x_0, \rho+\sigma)\subset B_{T_{p_r}\MM}({p_r}, \sqrt{\sigma \tau} - \sigma) $, so we can use Lemma \ref{lem:num_of_samples_in_a_ball_clean}.
    Explicitly, let $\mathrm{V}_d = \frac{\pi^{d/2}}{\Gamma(d/2+1)}$ be the volume of a $ d $-dimensional unit ball, and denote 
     \begin{align*}
    \mu_{\min} &= \mathrm{V}_{D-d}\sigma^{D-d} \min_{\substack{p\in \MM\\ x\in B_{T_p\MM}(0,\sqrt{\sigma\tau}- \sigma)}}  \sqrt{det(G_p(x))}\\
    \mu_{\max} &= \mathrm{V}_{D-d}\sigma^{D-d}\max_{\substack{p\in \MM \\ x \in B_{T_p\MM}(0,\sqrt{\sigma\tau}- \sigma)}}  \sqrt{det(G_p(x))}
    ,
    \end{align*}
    where $ \vol(B_{{T_p\MM}^\perp}(\sigma)) $ is the volume of a $ D-d $ dimensional ball with radius $ \sigma $ and $ G_p $ is the matrix representing the Riemannian metric at $ p $ in the chart $ \varphi_p $ of \eqref{eq:phi_def_clean}.
    Then, from Lemma \ref{lem:num_of_samples_in_a_ball_clean}, we get that for any $ \eps, \delta $ there is $ N $ such that for all $ n > N $
    \begin{align*}
    &\#\{r_i|Proj_{T_p\MM}(r_i) \in B_{T_p\MM}(x_0,\rho) \}
    \leq
    n(2\cdot \mu_{min}\cdot \mathrm{V}_d \cdot\rho^d+\eps) \\
    &\#\{r_i|Proj_{T_p\MM}(r_i) \in B_{T_p\MM}(x_0,\rho+\sigma) \}
    \geq
    n \left(\frac{\mu_{max}}{2}\cdot \mathrm{V}_{d}\cdot \rho^d - \eps\right)
    \end{align*}
    with probability of at least $1-\delta$.
    By setting $ \eps $ appropriately we get 
    \begin{align*}
    &\#\{r_i|Proj_{T_p\MM}(r_i) \in B_{T_p\MM}(x_0,\rho) \}
    \leq
    3 n \cdot\mu_{min}\cdot \mathrm{V}_d \cdot\rho^d \\
    &\#\{r_i|Proj_{T_p\MM}(r_i) \in B_{T_p\MM}(x_0,\rho+\sigma) \}
    \geq
    \frac{n}{3}  \cdot\mu_{max}\cdot \mathrm{V}_{d} \cdot \rho^d 
    .\end{align*}
    \item\label{thm:claim_R1p_bound_clean} Rephrasing \ref{thm:claim_ri_tyk_bound_clean}-\ref{thm:claim_useStochasticLemma} as one statement, there is a large subset of $U_\textrm{ROI}$ with large values of $\lrangle{r_i - p, \tilde y_j}^2$, with probability of at least $1-\delta$.
    Accordingly, we show that for any $ \delta $ there is $ N_\delta $ such that for all $ n>N_\delta $ 
    \begin{align}
    R_1'(r; p, H) \geq  &\frac{\mu_{max}}{\mu_{min} }\cdot\frac{1}{9} \left(\frac{1-\sqrt{\sigma/\tau}}{ 8}\right)^d\cdot \left(\frac{1}{64}\sigma \tau sin\alpha - \frac{1}{2}\sigma^{3/2}\tau^{1/2} \right)  -9 \sigma^2
    \label{eq:R1p_bound_clean},\end{align}
    with probability of at least $1-\delta$.
    This is proven in \ref{subsec:proof_lem_R1p_bound_clean}.
    
    \item Combining \eqref{eq:RpRpp_clean} with \eqref{eq:R1pp_bound_clean} and \eqref{eq:R1p_bound_clean}, we get that for $\alpha = \sqrt{C_\sM/M}$, where $M = \frac{\tau}{\sigma}$, and $C_\sM$ is a constant the following holds:
    For any $ \delta > 0$ there is $ N $ such that for all $ n > N $
    \begin{equation}
        R_1(r;p,H) \geq 109 \sigma^2
    ,\end{equation}
    with probability of at least $1-\delta$.
    \item From Claim \ref{thm:claim_J1H_clean} above, since $J_1(r;{p_r},T_{p_r}\MM) \geq 0$ we achieve that with probability of at least $1-\delta$ there is $N_\delta$ large enough such that for all $ n > N_\delta $ we have
$
       J_1(r; p, H) \geq 109 \sigma^2  
    ,$
    as required.
    
\end{enumerate}

\end{proof}

Next, we prove Claims \ref{thm:claim_R1pp_bound_clean}, \ref{thm:claim_ri_tyk_bound_clean}, and \ref{thm:claim_R1p_bound_clean} of the above proof.
For clarity of presentation after proving Claim \ref{thm:claim_R1pp_bound_clean} we prove Claim \ref{thm:claim_R1p_bound_clean} and only then show the correctness of the formula \eqref{eq:ri_tyk_bound_clean} presented in Claim 
\ref{thm:claim_ri_tyk_bound_clean}.

\subsubsection{Proof of Claim \ref{thm:claim_R1pp_bound_clean} of  Lemma \ref{thm:J1pH_clean}}\label{sec:proof_claim_R1pp_bound_clean}

\begin{proof} 
It is clear from \eqref{eq:R1pp_def_clean} that
\[
R_1''(r; p, H) \geq
- \sum\limits_{j\in\KK''} \frac{1}{\#|U_\textrm{ROI}|}\sum\limits_{r_i\in U_\textrm{ROI}}   \langle r_i - p, {y}_j \rangle^2 =  - \frac{1}{\#|U_\textrm{ROI}|}\sum\limits_{r_i\in U_\textrm{ROI}}\sum\limits_{j\in\KK''}\langle r_i - p, {y}_j \rangle^2
.\]
Furthermore, since $y_j\in T_p\MM^\perp$ and by Lemma \ref{lem:dist_ri_TpM_clean} we get 
\[
\sum\limits_{j\in\KK''}\abs{\lrangle{r_i - p, y_j}}^2 \leq dist^2(r_i - p, T_p\MM) \leq \left(\tau - \sqrt{\tau^2 - \norm{Proj_{T_p\MM}(r_i- p) }^2} + \sigma\right)^2 
,\]
and since $r_i\in U_\textrm{ROI} \subset B_D(p, \sqrt{\sigma\tau} + \sigma)$ we get
\begin{align*}
     \sum\limits_{j\in\KK''}\langle r_i - p,y_j \rangle^2 
     & \leq 
	\left(\tau+\sigma-\sqrt{\tau^2-(\sqrt{\sigma\tau}+\sigma)^2}\right)^2 \\
	& = \lrbrackets{\tau + \sigma - \tau\sqrt{1 - \lrbrackets{\sqrt{\frac{\sigma}{\tau}} + \frac{\sigma}{\tau} }^2 }}^2
\end{align*}
by Taylor expansion (Remark \ref{rem:taylor_sqrt1-x2_clean})
\begin{align*}
     \sum\limits_{j\in\KK''}\langle r_i - p,y_j \rangle^2 
	& \leq 
	\lrbrackets{\tau + \sigma - \tau\lrbrackets{1 - \frac{1}{2}\lrbrackets{\sqrt{\frac{\sigma}{\tau}} + \frac{\sigma}{\tau} }^2}}^2 \\
	&=
	\lrbrackets{\sigma + \tau \frac{1}{2}\lrbrackets{\sqrt{\frac{\sigma}{\tau}} + \frac{\sigma}{\tau} }^2}^2 \\
\end{align*} 
since $ \sigma < \tau $ we get $\sqrt{\frac{\sigma}{\tau}} > \frac{\sigma}{\tau} $ and
\begin{align}\label{eq:ri_yk_bound_s_square_clean}
     \sum\limits_{j\in\KK''}\langle r_i - p,y_j \rangle^2 
	& \leq 
	\lrbrackets{\sigma + 2\tau \frac{\sigma}{\tau}}^2 = 9 \sigma^2
\end{align}
and
\begin{equation*}
R_1''(r; p, H) \geq
-  \frac{1}{\#|U_\textrm{ROI}|}\sum\limits_{r_i\in U_\textrm{ROI}}   9 \sigma^2 =  -9 \sigma^2
\end{equation*}

\end{proof}

\subsubsection{Proof of Claim \ref{thm:claim_R1p_bound_clean} of Lemma \ref{thm:J1pH_clean}}\label{subsec:proof_lem_R1p_bound_clean}

\begin{SCfigure}
	\centering
	\includegraphics[width=0.4\textwidth]{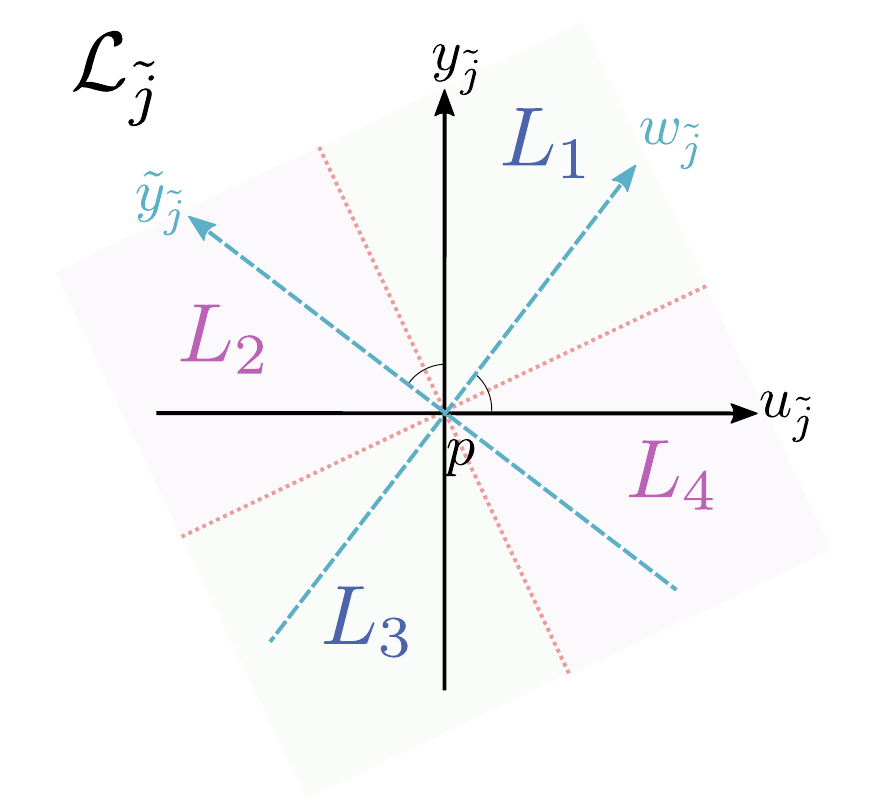}
	\caption{ Illustration of $\mathcal{L}_{\tilde j}$ and the sections $L_1,L_2,L_3,L_4$. $U_{bad}$ is marked in light green and $U_{good}$ is marked in light pink. The angles $\angle (\tilde y_{\tilde j}, y_{\tilde j}), \angle (w_{\tilde j},u_{\tilde j})$ which equal to $\beta_{\tilde j}$ are marked in black. The bisectors of these angles are marked in dotted red lines.}
	\label{fig:Lk_clean}
\end{SCfigure}
\begin{proof} 
For all $j\in \KK'$ we have $\beta_j>\alpha$.
Assume without loss of generality that there is only one index $\tilde j$ in $\KK'$ (otherwise we can treat each index separately and arrive at the same conclusion), then \eqref{eq:R1p_def_clean} can be rewritten as
\[
R_1'(r; p, H) = \frac{1}{\#|U_\textrm{ROI}|}\sum\limits_{r_i\in U_\textrm{ROI}}  \langle r_i - p, \tilde{y}_{\tilde j} \rangle^2 - \langle r_i - p, {y}_{\tilde j} \rangle^2
.\]
Thus, the only property which affect the score of $R_1'$ is the difference between the measurements $\langle r_i - p, \tilde{y}_{\tilde j} \rangle^2$ and $\langle r_i - p, {y}_{\tilde j} \rangle^2$, both on the 2D plane $\LL_{\tilde j}$.
Accordingly, using the bisector of $\angle(\tilde y_{\tilde  j}, y_{\tilde  j})$ and its orthogonal complement, we can split $\LL_{\tilde j}$ into four regions $L_1, L_2, L_3, L_4$ (see Figure \ref{fig:Lk_clean}), where in two regions ($L_2$ and $L_4$ in Figure \ref{fig:Lk_clean}) $\lrangle{r_i - p, \tilde y_{\tilde j}} \geq \lrangle{r_i - p, y_{\tilde j}}$ and in the other two regions ($L_1$ and $L_3$ in Figure \ref{fig:Lk_clean}) $\lrangle{r_i - p, \tilde y_{\tilde j}} \leq \lrangle{r_i - p, y_{\tilde j}}$.
By denoting 
$$U_{bad} = \{r_i \in U_\textrm{ROI} ~|~ Proj_{\LL_{\tilde j}}(r_i - p)\in L_1 \cup L_3\},$$
and
$$U_{good} = \{r_i \in U_\textrm{ROI} ~|~ Proj_{\LL_{\tilde j}}(r_i - p)\in L_2 \cup L_4\}$$
we get that
\[
R_1'(r; p, H) = \frac{1}{\#|U_\textrm{ROI}|}\left[\sum\limits_{r_i\in U_{good}}\left[  \langle r_i - p, \tilde{y}_{\tilde j} \rangle^2 - \langle r_i - p, {y}_{\tilde j} \rangle^2\right] + \sum\limits_{r_i\in U_{bad}} \left[ \langle r_i - p, \tilde{y}_{\tilde j} \rangle^2 - \langle r_i - p, {y}_{\tilde j} \rangle^2\right]\right]
\]
\[
R_1'(r; p, H) \geq \frac{1}{\#|U_\textrm{ROI}|}\left[\sum\limits_{r_i\in U_{good}}  \langle r_i - p, \tilde{y}_{\tilde j} \rangle^2 - \sum\limits_{r_i\in U_\textrm{ROI}} \langle r_i - p, {y}_{\tilde j} \rangle^2\right]
.\]
Similar to \eqref{eq:ri_yk_bound_s_square_clean}, 
\[
\frac{1}{\#|U_\textrm{ROI}|}\sum\limits_{r_i\in U_\textrm{ROI}} \langle r_i - p, {y}_{\tilde j} \rangle^2 \leq 9 \sigma^2
,\]
and thus
\[
R_1'(r; p, H) \geq \frac{1}{\#|U_\textrm{ROI}|}\sum\limits_{r_i\in U_{good}}  \langle r_i - p, \tilde{y}_{\tilde j} \rangle^2
-9\sigma^2
.\]
Therefore, all we need to show is that given $n$ large enough, there are enough samples in $U_{good}$ for which the value $\lrangle{r_i - p, \tilde y_{\tilde j}}^2$ is large enough.
Using Lemma \ref{lem:num_of_samples_in_a_ball_clean}, as described in Claim \ref{thm:claim_useStochasticLemma}, since $U_\textrm{ROI}\subset B_D(Proj_{\MM}(r), \sqrt{\sigma \tau}+\sigma)$, then for any $\delta$ there is $N$ large enough such that for all $n>N$  with probability of at least $1 - \delta$ 
\[
\#|U_\textrm{ROI}| < 3n\cdot\mu_{min}\cdot \mathrm{V}_{d}\cdot (\sqrt{\sigma\tau}+\sigma)^d 
.\]
Thus,
\[
R_1'(r; p, H) \geq \frac{1}{3n\cdot\mu_{min}\cdot V\cdot ( \sqrt{\sigma\tau} +\sigma)^d }\sum\limits_{r_i\in U_{good}}  \langle r_i - p, \tilde{y}_{\tilde j} \rangle^2
-9\sigma^2
.\]
Below in the proof of Claim \ref{thm:claim_ri_tyk_bound_clean} we show that \eqref{eq:ri_tyk_bound_clean} holds (the proof below is independent of the current one, but utilizes the notion of $U_{good}$ defined above).
Explicitly, for $ r_i \in U_{good} $
\begin{equation*}
\lrangle{r_i - p, \tilde y_j}^2 \geq \frac{1}{64}\sigma \tau \sin\alpha_0 - \frac{1}{2}\sigma^{3/2}\tau^{1/2}
.\end{equation*}
Combining this with Lemma \ref{lem:num_of_samples_in_a_ball_clean} we get that for any $\delta$ there is $N$ large enough such that for all $n>N$  with probability of at least $1-\delta$
\begin{align*}
R_1'(r; p, H) 
&\geq 
\frac{\frac{n}{3}\cdot \mu_{max}\cdot \mathrm{V}_d \cdot(a_2\sqrt{\sigma\tau})^d}{3n\cdot\mu_{min}\cdot V\cdot ( \sqrt{\sigma\tau}+\sigma)^d}\left(\frac{1}{64}\sigma\tau sin\alpha_0 -  \frac{1}{2}\sigma^{3/2}\tau^{1/2} \right)
-9\sigma^2 \\
&=
\frac{ a_2^d\cdot\mu_{max}}{9\cdot\mu_{min}}\left(\frac{ 1}{1+\sqrt{\sigma/\tau}}\right)^d\left(\frac{1}{64}\sigma\tau sin\alpha_0 - \frac{1}{2}\sigma^{3/2}\tau^{1/2} \right)
-9 \sigma^2
.\end{align*}
Since $1/(1+x) \geq 1-x$ for sufficiently small $x$ we have that, for large enough $M$ of Assumption \ref{assume:noise}, of section \ref{sec:SamplingAssumptions},
\[
\left(\frac{ 1}{1+\sqrt{\sigma/\tau}}\right)^d \leq (1-\sqrt{\sigma/\tau})^d
\]
holds.
Thus, we have 
\[
R_1'(r; p, H) \geq \frac{ a_2^d\cdot\mu_{max}}{9\cdot\mu_{min}}(1-\sqrt{\sigma/\tau})^d\left(\frac{1}{64}\sigma\tau sin\alpha_0 - \frac{1}{2}\sigma^{3/2}\tau^{1/2} \right)
-9 \sigma^2,
\]
Since $a_2 = 1/8$ we have 
\[
R_1'(r; p, H) \geq \frac{ \mu_{max}}{\mu_{min}}\cdot\frac{1}{9}\left(\frac{1-\sqrt{\sigma/\tau}}{ 8}\right)^d \left(\frac{1}{64}\sigma\tau sin\alpha_0 - \frac{1}{2}\sigma^{3/2}\tau^{1/2} \right)
-9 \sigma^2,
\]
as required.

\begin{figure}[htb]
    \centering
    \includegraphics[width=0.8\textwidth]{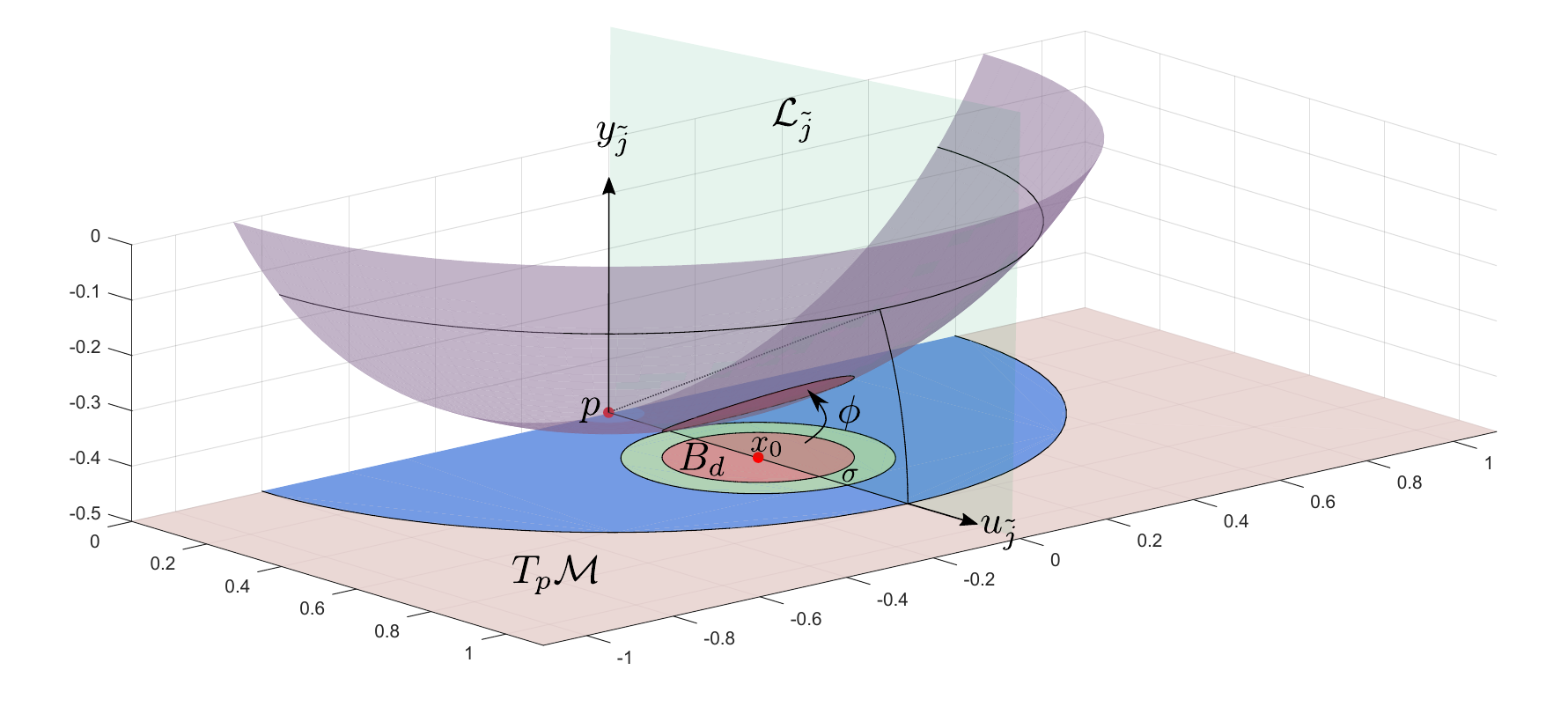}
    \caption{Assisting illustration for the proof of Claim \ref{thm:claim_ri_tyk_bound_clean}: the sphere section \emph{(}in purple\emph{)} represents the manifold; The $xy$-plane represents $T_p\MM$; the plane $\LL_{\tilde j}$ is spanned by $y_{\tilde j}\in T_p\MM^\perp$ and $u_{\tilde j}\in T_p\MM$; the blue disc on $T_p\MM$ is $B_{T_p\MM}(p, \sqrt{\sigma\tau} - \sigma)$; the manifold is considered to be the graph of the function $\phi_p:B_{T_p\MM}(p, \sqrt{\sigma\tau}- \sigma)\to T_p\MM^\perp$. }
    \label{fig:ManifoldGoodDisc_clean}
\end{figure}
\end{proof}
\subsubsection{Proof of Claim \ref{thm:claim_ri_tyk_bound_clean} of Lemma \ref{thm:J1pH_clean}}\label{subsec:proof_lem_ri_tyk_bound_clean}

\begin{proof} 
As a result of Corollary \ref{cor:GraphOfFunctionTau_ROI}, $r_i \in U_\textrm{ROI}$ can be written as 
$$
r_i = \underbrace{p+(x_i, \phi_p(x_i))}_{p_i}+\eps_i ,
$$
where $x_i=Proj_{T_p\MM}(r_i - p)\in T_{p}\MM$ and $\|\eps_i\| \leq \sigma$.
Since 
$$ 
B_D(p, \sqrt{\sigma\tau}-\sigma ) \subset B_D(r, \sqrt{\sigma\tau}) = U_\textrm{ROI}
$$
(see the blue disc on $T_p\MM$ in Figure \ref{fig:ManifoldGoodDisc_clean}), we look for a point $x_0\in T_p\MM\cap \LL_{\tilde j}$ and a radius $\rho >0$ such that $\Gamma_{\phi_p, B_{T_p\MM}(x_0, \rho)}^\sigma\subset B_D(p, \sqrt{\sigma\tau} - \sigma)$, where
\begin{equation}
    \Gamma_{\phi_p, B_{T_p\MM}(x_0, \rho)}^\sigma \defeq \{x ~|~ dist(x, \Gamma_{\phi_p, B_{T_p\MM}(x_0, \rho)}) < \sigma\}
,\end{equation}
where we remind that
\[
\Gamma_{\phi_p, B_{T_p\MM}(x_0, \rho)} = \{p + (x, \phi_p(x))_{T_p\MM} ~|~ x\in B_{T_p\MM}(x_0, \rho)\}
.\]
Furthermore, we choose $x_0$ and $\rho$, such that $\lrangle{r_i - p, \tilde y_{\tilde j}}^2$ is large for any point $r_i\in\Gamma_{\phi_p, B_{T_p\MM}(x_0, \rho)}^\sigma$ (see Figure \ref{fig:ManifoldGoodDisc_clean} for an illustration).

For convenience, we denote by $q^y, q^x, q^{\vec x}$ the projections of $q\in \Gamma_{\phi_p, B_{T_p\MM}(x_0, \rho)}^\sigma$ onto $y_{\tilde j}$ (i.e., $T_p\MM^\perp\cap \LL_{\tilde j}$), $u_{\tilde j}$ (i.e., $T_p\MM\cap \LL_{\tilde j}$) and $T_p\MM$ respectively.
Let $a_1, a_2<1$ and we define 
\[\rho = a_2 \cdot\sqrt{\sigma\tau}, ~~~\mbox{ and }~~~
x_0 = (a_1\sqrt{\sigma\tau})\cdot u_{\tilde j}\in T_p\MM \cap \LL_{\tilde j}.\]
That is, 
$
\norm{p - x_0} = a_1 \cdot \sqrt{\sigma\tau}
$.
In order to make sure that $\Gamma^\sigma_{\phi_p, B_{T_p\MM}(x_0, \rho)}\subset U_{\textrm{ROI}}$ we restrict the choice of $a_1, a_2$ such that
\begin{equation}\label{eq:a1a2_restriction_clean}
    \forall x\in B_{T_p\MM}(x_0, a_2\sqrt{\sigma\tau}): \norm{p+(x,\phi_p(x))_{T_p\MM} ~-~ p}
    + \sigma < \sqrt{\sigma\tau} - \sigma
,\end{equation}
We wish to reiterate that $U_\textrm{ROI} = \{r_i ~|~ \norm{r - r_i}\leq \sqrt{\sigma \tau}\}$ and so $U_\textrm{ROI}\subset B_D(r, \sqrt{\sigma\tau})$.
Furthermore, since $\norm{r - p}\leq \sigma$, we have $B_D(p, \sqrt{\sigma\tau} - \sigma)\subset  B_D(r, \sqrt{\sigma\tau})$.
Accordingly, all points $p+(x, \phi_p(x))_{T_p\MM}$ for  $x\in B_{T_p\MM}(x_0, a_2 \sqrt{\sigma\tau})$ are within our region of interest even when moved $\sigma$ away from the manifold in some direction into $\MM_\sigma$ (in Fig. \ref{fig:ManifoldGoodDisc_clean}  $B_{T_p\MM}(x_0, a_2\sqrt{\sigma\tau})$ is the orange disc and the projections of $p+(x, \phi_p(x))_{T_p\MM} + \eps(x)$ onto $T_p\MM$ are limited by the green disc containing the orange disc).
Using the calculations in Appendix \ref{subsec:simplification_app} we can use a simpler demand using $a_1$ and $a_2$, which ensures that the inequality \eqref{eq:a1a2_restriction_clean} is satisfied. The simplified requirement is
\begin{equation}\label{eq:a1a2_simplification_clean}
    (a_1+a_2)<\frac{1}{\sqrt{2}} -\frac{\sqrt{2\sigma}}{\sqrt{\tau}}
.\end{equation}

Let us now bound the value of $\lrangle{\tilde q - p, \tilde y_{\tilde j}}^2$ from below, for any $\tilde  q\in \Gamma^\sigma_{\phi_p,B_{T_p\MM}(x_0, a_2\sqrt{\sigma\tau})}$.
Every such $\tilde  q$ can be written as
$
\tilde q = q + \eps
$,
where $q\in \Gamma^0_{\phi_p,B_{T_p\MM}(x_0, a_2\sqrt{\sigma\tau})}$, and so 
 \begin{equation}\label{eq:qxBound_clean}
 (a_1 - a_2)\sqrt{\sigma\tau}  \leq \abs{q^x} \leq \|q^{\vec{x}}\| \leq \|p-x_0\|+a_2\sqrt{\sigma\tau} = (a_1 + a_2)\sqrt{\sigma\tau}   
.\end{equation}
From Lemma \ref{lem:f_bound_circle_clean} we get that
\[
\abs{\lrangle{q, y_{\tilde j}}} \leq \tau-\sqrt{\tau^2-\norm{q^{\vec{x}}}^2} 
.\]
Thus, 
\[
\abs{\langle \tilde{q}, y_{\tilde j} \rangle} = |\tilde{q}^y| \leq \tau+\sigma-\sqrt{\tau^2-\|q^{\vec{x}}\|^2}
,\]
and by plugging the right hand side of \eqref{eq:qxBound_clean} we get
\begin{align*}
\langle \tilde{q}, y_{\tilde j} \rangle^2 & \leq \left(\tau+ \sigma -\sqrt{\tau^2-(a_1+a_2)^2\sigma\tau}\right)^2
 =\tau^2 \left(1+ \sigma/\tau -\sqrt{1-(a_1+a_2)^2\sigma/\tau}\right)^2
\\&\leq \tau^2(1+\sigma/\tau -(1-(a_1+a_2)^2\sigma/\tau))^2
= \tau^2 (\sigma/\tau +  (a_1+a_2)^2\sigma/\tau)^2
\\&=  \sigma^2(1 +  (a_1+a_2)^2)^2
\leq 2\sigma^2
,\end{align*}
where the last inequality comes from \eqref{eq:a1a2_simplification_clean}.
Since $\angle(y_{\tilde j}, \tilde y_{\tilde j}) =\beta_{\tilde j} $ we can use the Euclidean geometry on $\LL_{\tilde j}$ (Figure \ref{fig:Lk_clean}) to get
\[
\lrangle{q, \tilde y_{\tilde j}} = \lrangle{q, \textrm{Rot}(\beta_{\tilde j}) y_{\tilde j}} = \lrangle{\textrm{Rot}(-\beta_{\tilde j}) q,  y_{\tilde j}}
,\]
where $\textrm{Rot}(\theta)$ denotes the rotation matrix in $\RR^D$ with respect to the angle $\theta$ in $\LL_{\tilde j}$ .
Therefore,
$$|\langle q, \tilde{y}_{\tilde j} \rangle| \geq \abs{-|q^x|\sin \beta_{\tilde j}  + |q^y|\cos \beta_{\tilde j}} = \abs{|q^x|\sin \beta_{\tilde j}  - |q^y|\cos \beta_{\tilde j}}.$$
Using Lemma \ref{lem:f_bound_circle_clean} and \eqref{eq:qxBound_clean} as before we get
\begin{align*}
 |\langle q, \tilde{y}_{\tilde j} \rangle| &=\abs{  |q^x|\sin \beta_{\tilde j}  -  \left(\tau-\sqrt{\tau^2-\|q^{\vec{x}}\|^2} \right)\cos \beta_{\tilde j}}
 \\
 &\geq\left((a_1-a_2)\sqrt{\sigma\tau}\sin \beta_{\tilde j}  -  \tau\left(1-\sqrt{1- (a_1+a_2)^2\sigma/\tau} \right)\cos \beta_{\tilde j} \right)
 \\
 &\geq\tau\left((a_1-a_2)\sqrt{\sigma/\tau}\sin \beta_{\tilde j}-\cos \beta_{\tilde j} + \cos \beta_{\tilde j}(1- (a_1+a_2)^2\sigma /\tau) \right)
 \\
 &= \tau \left((a_1-a_2)\sqrt{\sigma/\tau}\sin \beta_{\tilde j} -(a_1+a_2)^2\sigma/\tau \cos\beta_{\tilde j}\right)
 \\
 &\geq \tau \left((a_1-a_2)\sqrt{\sigma/\tau}\sin \beta_{\tilde j} -(a_1+a_2)^2\sigma/\tau \right) 
,\end{align*}
and
\begin{equation}\label{eq:q_yk_prod_clean}
    |\langle q, \tilde{y}_{\tilde j} \rangle| \geq \tau \left((a_1-a_2)\sqrt{\sigma/\tau}\sin \alpha_0 -(a_1+a_2)^2\sigma/\tau \right) 
.\end{equation}
Since 
$\|\tilde{q} - q\| \leq \sigma $
we get 
$$\abs{\lrangle{\tilde q - q, \tilde y_{\tilde j}}} = |\langle \tilde{q}, \tilde{y}_{\tilde j}\rangle - \langle q, \tilde{y}_{\tilde j}\rangle| \leq\sigma,$$
and so,
\begin{align*}
|\langle q, \tilde{y}_{\tilde j}\rangle| - |\langle \tilde{q}, \tilde{y}_{\tilde j}\rangle|&\leq |\langle q, \tilde{y}_{\tilde j}\rangle - \langle \tilde{q}, \tilde{y}_{\tilde j}\rangle| \leq  \sigma\\
|\langle \tilde{q}, \tilde{y}_{\tilde j}\rangle|&\geq |\langle q, \tilde{y}_{\tilde j}\rangle| - \sigma\\
\langle \tilde{q}, \tilde{y}_{\tilde j}\rangle^2&\geq \langle q, \tilde{y}_{\tilde j}\rangle^2 - 2\sigma|\langle q, \tilde{y}_{\tilde j}\rangle| + \sigma^2\\
\langle \tilde{q}, \tilde{y}_{\tilde j}\rangle^2&\geq \langle q, \tilde{y}_{\tilde j}\rangle^2 - 2\sigma|\langle q, \tilde{y}_{\tilde j}\rangle|
.\end{align*}
 Substituting $\abs{\lrangle{q, \tilde y_{\tilde j}}}$ with the bound from \eqref{eq:q_yk_prod_clean} we get
\begin{align*}
    \lrangle{\tilde q, \tilde y_{\tilde j}}^2 & \geq \tau^2 \left((a_1-a_2)\sqrt{\sigma/\tau}\sin \alpha_0 -  (a_1+a_2)^2{\sigma/\tau}\right)^2 \\&~~~ - 2\sigma\tau \left((a_1-a_2)\sqrt{\sigma/\tau}\sin \alpha_0 -  (a_1+a_2)^2{\sigma/\tau}\right) \\
    &\geq \tau^2\left((a_1-a_2)\sqrt{\sigma/\tau}\sin \alpha_0 -  {\sigma/\tau}\right)^2 - 2\sigma\tau \left((a_1-a_2)\sqrt{\sigma/\tau}\sin \alpha_0\right)\\
    &= \tau^2 \left((a_1-a_2)^2{\sigma/\tau}\sin^2 \alpha_0 - 2(a_1-a_2)(\sigma/\tau)^{3/2}\sin\alpha_0 + (\sigma/\tau)^2\right) \\&~~~ -2(a_1-a_2)\sigma^{3/2}\sqrt{\tau} \sin \alpha_0 \\
    &\geq (a_1-a_2)^2{\sigma\tau}\sin^2 \alpha_0 - 2(a_1-a_2)\sigma^{3/2}\tau^{1/2}\sin\alpha_0 + \sigma^2 -2(a_1-a_2)\sigma^{3/2}{\tau}^{1/2} \\
    &= (a_1-a_2)^2{\sigma\tau}\sin^2 \alpha_0 - 2(a_1-a_2)\sigma^{3/2}\tau^{1/2} + \sigma^2 -2(a_1-a_2)\sigma^{3/2}{\tau}^{1/2} \\
    &= (a_1-a_2)^2{\sigma\tau}\sin^2 \alpha_0 - 4(a_1-a_2)\sigma^{3/2}\tau^{1/2} + \sigma^2 
\end{align*}
Thus, by choosing for example $a_1 = \frac{1}{4}$ and $a_2 = \frac{1}{8}$ we get
\begin{equation}
    \lrangle{\tilde q, \tilde y_{\tilde j}}^2 \geq \frac{1}{64}\sigma \tau sin\alpha_0 - \frac{1}{2}\sigma^{3/2}\tau^{1/2}
.\end{equation}
\end{proof}
\subsubsection{Simplification of \eqref{eq:a1a2_restriction_clean} to achieve \eqref{eq:a1a2_simplification_clean}}\label{subsec:simplification_app}
We wish to rewrite the requirement of \eqref{eq:a1a2_restriction_clean} in terms of $a_1$ and $a_2$.
By Lemma \ref{lem:f_bound_circle_clean}, for all $x\in B_{T_p\MM}(x_0, a_2 \sqrt{\sigma\tau})$:
\begin{equation*}
    \|\phi_p(x)\| \leq \tau - \sqrt{\tau^2 - (a_1+a_2)^2 {\sigma\tau}} 
;\end{equation*}
see Figure \ref{fig:ManifoldGoodDisc_clean}.
Thus,
\begin{align*}
\norm{(x,\phi_p(x))_{T_p\MM}}
& \leq \sqrt{(a_1+a_2)^2 {\sigma\tau} + \left(\tau - \sqrt{\tau^2 - (a_1+a_2)^2 {\sigma\tau}}\right)^2} \\
&= \sqrt{(a_1+a_2)^2{\sigma\tau} + \tau^2 - 2\tau\sqrt{\tau^2 - (a_1+a_2)^2 {\sigma\tau}} + \tau^2 - (a_1+a_2)^2 {\sigma\tau}} \\
&= \sqrt{2\tau^2 - 2\tau^2\sqrt{1 - (a_1+a_2)^2 {\sigma/\tau}} } \\
&= \tau\sqrt{2} \sqrt{1 - \sqrt{1 - (a_1+a_2)^2 {\sigma/\tau}} } \\
&\leq \tau\sqrt{2} \sqrt{1 - (1 - (a_1+a_2)^2 {\sigma/\tau}) } \\
&= \sqrt{2}(a_1+a_2) \tau\sqrt{ {\sigma/\tau} } \\
&= \sqrt{2}(a_1+a_2) \sqrt{ {\sigma\tau} },\end{align*}
Where the second inequality results from applying \eqref{eq:Taylor2ndOrder}.
Therefore, the requirement of \eqref{eq:a1a2_restriction_clean} translates to
$$
\norm{(x,\phi_p(x))_{T_p\MM}}_{T_p\MM,y_{\tilde j}}\leq \sqrt{2}(a_1+a_2)\sqrt{\sigma\tau}<\sqrt{\sigma\tau}-2\sigma
,$$
which can be simplified to 
$$
(a_1+a_2)<\frac{1}{\sqrt{2}} - \frac{\sqrt{2\sigma}}{\sqrt{\tau}} 
.$$

\subsection{Supporting Lemmas on Sample size in a given volume}
In this section we concentrated all assisting lemmas that are used in the proofs of Step 1.

\begin{Lemma}[Number of samples in a ball]\label{lem:sampling_density_almost_uniform_clean}
Suppose $\nu$ is a distribution on $\Omega\subset B_d(0,R) \subset \RR^d$ which is close to the uniform distribution $\mu$. That is, there exists $\mu_{max}$, $\mu_{min}$  such that for any $x\in \Omega$ we have  $\mu_{\min}\mu(x) \leq \nu(x) \leq \mu_{\max}\mu(x)$. Suppose $X=\{x_j\}_{j=1}^n$ is a set of $n$ i.i.d. sample from $\nu$, and denote the volume of a $ d $-dimensional unit ball by  $ \mathrm{V}_d = \frac{\pi^{d/2}}{\Gamma(d/2+1)} $. For any $\eps$, $\delta$ and radius $\rho$, there is $N$, such that if $n>N$ the following holds: For any $x_0\in \Omega$ such that $B_d(x_0,\rho)\subset \Omega$, we have 
\[
n(\frac{\mu_{max}}{2}\cdot \mathrm{V}_d \cdot\rho^d-\eps) < \#|X \cap B_d(x_0, \rho)| < n(2\cdot \mu_{min}\cdot \mathrm{V}_d \cdot\rho^d+\eps)
\]
with probability of at least $1-\delta$, where $\#|A|$ denotes the number of elements in the set $A$.
\end{Lemma}
\begin{proof}
Since $\Omega \subset B_d(0,R)\subset\RR^d$, there exists an $\tilde\eps$-net (denoted by $\Xi$) such that 
\[
\#\abs{\Xi} = \ceil{ \frac{3R}{\tilde\eps}}^d
,\]
where $ \ceil{x} $ is the ceiling value of $ x $ \cite{vershynin2018high}.
Around each point $p$ in $\Xi$, we consider a ball $B_d(p, (1 - \tilde\eps)\rho)$. 
Note, that this $\tilde\eps$-net along with these balls are independent of the choice of a specific ball $B_d(x_0,\rho)$.

For each of the $B_d(p, (1-\tilde\eps)\rho)$ , we  consider our sample set as $n$ i.i.d random variables   $Z^p_j$ which return the value $1$ if the sample lies within $B_d(p, (1-\tilde\eps)\rho)$ and $0$ if not.
Naturally, we get for all $j$ that
\[
\Pr[z_j^p = 1] = \int\limits_{B_d(p, (1-\tilde\eps)\rho)} d\nu 
\] 
Applying Hoeffding's inequality for each of the $B_d(p, (1 - \tilde\eps)\rho)$ we arrive at
\begin{equation}\label{eq:HoeffingtonInequality1_clean}
\Pr[\overline{Z}^p - \EE[\overline{Z}^p] \leq -\eps] \leq e^{-2 n \eps^2}
,    
\end{equation}
and
\begin{equation}\label{eq:HoeffingtonInequality2_clean}
\Pr[\overline{Z}^p - \EE[\overline{Z}^p] \geq \eps] \leq e^{-2 n \eps^2}
,    
\end{equation}
where
\[
\overline{Z}^p = \frac{1}{n}\sum_{j=1}^n Z_j^p.
\]
As a result
\[
\EE[\overline{Z}^p] = \frac{1}{n}\sum_{j=1}^n\Pr[Z_j^p = 1] =
\frac{1}{n}\sum_{j=1}^n\int\limits_{B_d(p, (1 - \tilde\eps)\rho)}d\nu
,\]
and
\begin{equation}\label{eq:EzBounds_clean}
\mu_{min} \cdot \vol(B_d(p, (1 - \tilde\eps)\rho)) \leq \EE[\overline{Z}^p] \leq \mu_{max} \cdot \vol(B_d(p, (1 - \tilde\eps)\rho))
.\end{equation}
Plugging this into \eqref{eq:HoeffingtonInequality1_clean} we get
\begin{equation}\label{eq:PrBound1_clean}
\Pr\left[\overline{Z}^p - \mu_{max}\vol\left(B_d(p, (1- \tilde\eps)\rho)\right)\leq -\eps\right] \leq e^{-2n \eps^2}
,\end{equation}
or, alternatively, since $\#|X \cap B_d(p,(1-\tilde\eps)\rho)| = n \cdot \overline{Z}^p$ we get
\[
\Pr\left[\#|X \cap B_d(p,(1-\tilde\eps)\rho)| \leq n(\mu_{max}\vol(B_d(p,(1-\tilde\eps)\rho))-\eps)\right]\leq e^{-2n \eps^2}
.\]

Denoting by $A_p$ the event $\#|X \cap B_d(p,(1-\tilde\eps)\rho)| \leq n(\mu_{max}\vol(B_d(p,(1-\tilde\eps)\rho))-\eps)$ we use the union bound to achieve 
\begin{equation}\label{eq:UnionBoundAp_clean}
\Pr\left(\bigcup\limits_{p\in\Xi} A_p\right) \leq \sum\limits_{p\in\Xi} \Pr\left(A_p \right) \leq \ceil{ 3R/\tilde\eps}^d \cdot  e^{-2n \eps^2}
.\end{equation}
Explicitly, the chances that there exists $B_d(p,(1-\tilde\eps)\rho)$ containing \textbf{less} sampled points than $n(\mu_{max}\vol(B_d(p,(1-\tilde\eps)\rho))-\eps)$ are less than $c \cdot e^{-2n \eps^2}$, where $c = \ceil{ 3R/\tilde\eps}^d$.
Going back to $B_d(x_0, \rho)$, we know that there exists a point $\tilde p\in \Xi$ such that 
\[
B_d(\tilde p,(1-\tilde\eps)\rho)\subset B_d(x_0, \rho)
.\]
As a result, for any $\delta, \eps, \rho$ there exists $N$ such that for all $n>N$
\[
\#|X \cap B_d(x_0, \rho)| > n(\mu_{max}\vol(B_d((1-\tilde\eps)\rho))-\eps)
= n(\mu_{max}\cdot \mathrm{V}_d \cdot(1-\tilde\eps)^d\rho^d-\eps)
,\]
with probability larger than $1-\delta$, where
\[
\mathrm{V}_d = \frac{\pi^{d/2}}{\Gamma(d/2+1)}
.\]

Similarly, instead of considering $B_d(p, (1-\tilde\eps)\rho)$ we look at $B_d(p, (1+\tilde\eps)\rho)$ for $p\in\Xi$ and alter the definitions of $Z^p$ accordingly.
Then, by plugging the left inequality of \eqref{eq:EzBounds_clean} into \eqref{eq:HoeffingtonInequality2_clean} we get
\[
\Pr[\overline{Z}^p - \mu_{min}\vol(B_d((1+\tilde\eps)\rho))\geq \eps] \leq e^{-2n\eps^2} 
\]
Utilizing the union bound once more we get the same bound as in \eqref{eq:UnionBoundAp_clean}. 
In other words, the chances that there exists a $B_d(p, (1+\tilde\eps)\rho)$ containing \textbf{more} sampled points than $n(\mu_{min} \cdot \vol(B_d((1+\tilde\eps)\rho)) + \eps )$ are less than $\ceil{3R/\tilde\eps}^d\cdot e^{-2n\eps^2}$.

Going back to $B_d(x_0, \rho)$, we know that there exists a point $\tilde p\in \Xi$ such that 
\[
B_d(x_0, \rho) \subset B_d(\tilde p,(1+\tilde\eps)\rho)
,\]
and for $\tilde\eps$ small enough
\[
B_d(\tilde p,(1+\tilde\eps)\rho) \subset \Omega
.\]
As a result, for any $\delta, \eps, \rho$ there exists $N$ such that for all $n>N$
\[
\#|X \cap B_d(x_0, \rho)| < n(\mu_{min}\vol(B_d((1+\tilde\eps)\rho))+\eps)
= n(\mu_{min}\cdot \mathrm{V}_d \cdot(1+\tilde\eps)^d\rho^d+\eps)
,\]
with probability larger than $1-\delta$.

Finally, we get that for any $\delta, \eps, \rho$ there exists $N$ large enough such that if $n>N$ we get
\[
n(\mu_{max}\cdot \mathrm{V}_d \cdot(1-\tilde\eps)^d\rho^d-\eps) < \#|X \cap B_d(x_0, \rho)| < n(\mu_{min}\cdot \mathrm{V}_d \cdot(1+\tilde\eps)^d\rho^d+\eps)
\]
Since this is true for any $\tilde\eps$, $(1+1/a)^a \leq 3$ and $(1-1/a)^a \geq 0.25$ we can choose $\tilde\eps = \frac{1}{c_1 d}$ such that
\[(1 - \tilde\eps)^d \geq 0.25^{1/c_1} \geq 0.5,\]
and
\[(1 + \tilde\eps)^d \leq 3^{1/c_1} \leq 2,\]
and achieve
\[
n(\frac{\mu_{max}}{2}\cdot \mathrm{V}_d \cdot\rho^d-\eps) < \#|X \cap B_d(x_0, \rho)| < n(2\cdot \mu_{min}\cdot \mathrm{V}_d \cdot\rho^d+\eps)
\]
\end{proof}

\begin{Lemma}[The projection of the Lebesgue measure onto $T_p\MM$ is almost uniform] \label{lem:measure_projecting_to_tangent_clean}
    Let $\MM$ be a $d$-dimensional sub-manifold of $\RR^D$ with bounded reach $\tau$ and a Riemannian metric $G$ pronounced through the chart $\varphi_p$ around a point $ p\in \MM $ \eqref{eq:LocalChart_clean}. 
    Let $r\in\MM_\sigma$, and let $\mu, \mu_\MM, \mu_{T_p\MM}$ denote the uniform distribution on $\MM_\sigma\subset\RR^D$, $\MM$, $T_p\MM$ correspondingly.
    Denote $Proj_\MM, Proj_{T_p\MM}$ the projection operators onto $\MM$ and $T_p\MM$.
    Then $\left(Proj_{T_p\MM}Proj_{\MM}\right)_*\mu$ is a measure on $T_p\MM$, and upon restricting this measure to $B_d(0,\rho)$ for some $\rho < \tau/2$ we get
    \[
    (Proj_{T_p\MM}Proj_\MM)_*\mu = \mathrm{V}_{D-d}\sigma^{D-d} \sqrt{det(G)}\mu_{T_p\MM}
    ,\]
    where $\mathrm{V}_d$ denotes the volume of a $d$-dimensional ball.
\end{Lemma}

\begin{proof}
We first note that since $\mu$ is the Lebesgue measure on $\MM_\sigma$ we have
\[
\int_{\MM_\sigma} d\mu = \int_{\MM}\mathrm{V}_{D-d}\sigma^{D-d}\mu_\MM
,\]
where $\mu_\MM$ is the uniform distribution on $\MM$. Thus,
\[
Proj_\MM \mu = {\mathrm{V}_{D-d}\sigma^{D-d}}\mu_\MM
.\]
Now, $Proj_\MM\mu$ is a measure defined on $\MM$, which can be pulled back to the tangent domain $T_p\MM\simeq\RR^d$ in the neighborhood $B_d(p, \tau/2)$ according to Corollary \ref{cor:GraphOfFunctionTau2_clean}.
If we denote the chart from $T_p\MM\simeq\RR^d$ to $\MM$ by $\varphi_p$ we get
\[
(Proj_{T_p\MM}Proj_\MM)_*\mu = \mathrm{V}_{D-d}\sigma^{D-d} \sqrt{det(G)}\mu_{T_p\MM}
,\]
where $G$ is the Riemannian metric expressed in this chart, and $\mu_{T_p\MM}$ is the Lebesgue measure on $T_p\MM$.

\end{proof}

\begin{corollary}\label{cor:measure_projecting_bounds_clean}
	From the fact that $\MM\in\C^k$ is compact and the restriction to a ball of radius $\tau/2$ we get that $\sqrt{det(G)}$ is bounded and
	\[
	\mu_{min} \mu_{T_p\MM} \leq (Proj_{T_p\MM}Proj_\MM)_*\mu \leq \mu_{max} \mu_{T_p\MM}
	\]
	where $\mu_{min}, \mu_{max}$ are constants that depend on $\tau$, and $\mu_{T_p\MM}$ is the Lebesgue measure on $T_p\MM$.
	The constants $ \mu_{min}, \mu_{max} $ can be described explicitly to show their exact relationship to $ \tau $.
\end{corollary}

Combining Lemma \ref{lem:sampling_density_almost_uniform_clean} and Corollary \ref{cor:measure_projecting_bounds_clean}, we have the following result

\begin{Lemma}\label{lem:num_of_samples_in_a_ball_clean}
	Let $\MM$ be a compact $d$-dimensional sub-manifold of $\RR^D$ with reach $\tau$  bounded away from zero, and a Riemannian metric $G_p$ pronounced through the chart $\varphi_p$ around a point $ p\in \MM $ \emphbrackets{see \eqref{eq:LocalChart_clean}}. 
    Let $\MM_\sigma$ be a tubular neighborhood around $ \MM $ of radius $ \sigma $ \emphbrackets{see \eqref{eq:Msigma}}, and assume $\sqrt{\frac{\sigma}{\tau}} < \frac{1}{2}$. 
    Suppose that $\mu$ is the uniform distribution on $ \MM_\sigma $. 
    Let $X = \{r_1, \ldots,r_n\}$ be $n$ points sampled i.i.d from $\mu$, and denote the volume of a $d$-dimensional unit ball by  $ \mathrm{V}_d = \frac{\pi^{d/2}}{\Gamma(d/2+1)} $.
    Denote,
    \begin{align}
    \label{eq:mu_min_max_def}
    \begin{split}
    \mu_{\min} &= \mathrm{V}_{D-d}\sigma^{D-d} \min_{\substack{p\in \MM\\ x\in B_{T_p\MM}(0,\sqrt{\sigma\tau}- \sigma)}}  \sqrt{det(G_p(x))}\\
    \mu_{\max} &= \mathrm{V}_{D-d}\sigma^{D-d}\max_{\substack{p\in \MM \\ x \in B_{T_p\MM}(0,\sqrt{\sigma\tau}- \sigma)}}  \sqrt{det(G_p(x))}
    .\end{split}
    \end{align}
    
    Then for any $\eps$ and $\delta$, there is $N$, such that for all $n>N$ the following holds: For any $x_0\in T_p\MM$ and $\rho \in \RR^+$ such that $B_{T_p\MM}(x_0,\rho + \sigma)\subset  B_{T_p\MM}(0,\sqrt{\sigma\tau}- \sigma)$ \emphbrackets{see the red, green, and blue discs in Figure \ref{fig:ManifoldGoodDisc_clean}}, we have 
\begin{align*}
\begin{split}
&\#\{r_i|Proj_{T_p\MM}(r_i) \in B_{T_p\MM}(x_0,\rho) \}
\leq
n(2\cdot \mu_{min}\cdot \mathrm{V}_d \cdot\rho^d+\eps) \\
&\#\{r_i|Proj_{T_p\MM}(r_i) \in B_{T_p\MM}(x_0,\rho+\sigma) \}
\geq
n \left(\frac{\mu_{max}}{2}\cdot V \cdot \rho^d - \eps\right)
\end{split}
\end{align*}
with probability of at least $1-\delta$.
\end{Lemma}

\begin{proof}
    Note that both the minimum and maximum in \eqref{eq:mu_min_max_def} exist and finite since $ \MM $ is compact and the determinant is continuous.
	We mention that since $ \sqrt{\frac{\sigma}{\tau}} < \frac{1}{2} $ we get that $ \sqrt{\sigma \tau} < \frac{\tau}{2} $ and the conditions of Corollary \ref{cor:measure_projecting_bounds_clean} are met.
	Note that 
	\[
	\#\{r_i|Proj_{T_p\MM}(r_i) \in B_d(x_0,\rho) \} \leq \#\{r_i|Proj_{T_p\MM}\circ Proj_\MM(r_i) \in B_d(x_0,\rho) \}
	,\]
	and from Lemma \ref{lem:sampling_density_almost_uniform_clean} combined with Corollary \ref{cor:measure_projecting_bounds_clean} we get
	\[
	\#\{r_i|Proj_{T_p\MM}\circ Proj_\MM(r_i) \in B_d(x_0,\rho) \} \leq n(2\cdot \mu_{min}\cdot \mathrm{V}_d \cdot\rho^d+\eps)
	.\]
	Thus, the first inequality is achieved.
	
	On the other hand,
	\[
	\#\{r_i|Proj_{T_p\MM}(r_i) \in B_d(x_0,\rho+\sigma) \}
	\geq \#\{r_i|Proj_{T_p\MM}\circ Proj_\MM(r_i) \in B_d(x_0,\rho)\}
	.\]
	Using again Lemma \ref{lem:sampling_density_almost_uniform_clean} combined with Corollary \ref{cor:measure_projecting_bounds_clean} we get
	\[
	\#\{r_i|Proj_{T_p\MM}\circ Proj_\MM(r_i) \in B_d(x_0,\rho)\} \geq n \left(\frac{\mu_{max}}{2}\cdot \mathrm{V}_{d} \cdot \rho^d - \eps\right)
	,\]
	and the second inequality holds.
\end{proof}

\section{Supporting lemmas for Step 2}
\begin{Lemma}\label{lem:H0_Tf_Init}
    Let the sampling assumptions of Section \ref{sec:SamplingAssumptions} hold. Let $(q, H_0)\in \MM_\sigma\times Gr(d,D)$ be the initialization of Algorithm \ref{alg:step2_clean}, and let $p = Proj_\MM(q)$. 
    Assume $\maxangle (H_0, T_p\MM) \leq \sqrt{C_\sM/M}$, where $C_\sM$ is a constant independent of $M$ \emph{(}see Theorem \ref{thm:Step1}\emph{)}. 
    Let $ f_{0}:\RR^d\to\RR^{D-d} $ be defined 
    as a function whose graph $ \Gamma_{f_{0}, r, H_0} $ coincides with $ \MM $ in the sense of Lemma \ref{lem:M_is_locally_a_fuinction_clean}; explicitly 
	\[
	\Gamma_{f_{0},r, H_0} = \{r + (x, f_{0}(x))_{H_0} ~|~ x\in B_{H_0}(r, c_{\pi/4}\tau)\}
	,\]
	where $c_{\pi/4}$ is a constant and $\Gamma_{f_0, r, H_0} \subset \MM$.
    Then, for any $\alpha  \geq 0$, there is  $M=\tau/\sigma $ large enough \emph{(}see assumption \ref{assume:noise} in Section \ref{sec:SamplingAssumptions}\emph{)}, such that  
	\[
	\maxangle(H_0, T_0 f_{0}) 
	\leq  \alpha
	.\]
\end{Lemma}
\begin{proof}
First, we define $f_{q}:\RR^d\to\RR^{D-d}$ as the function whose graph $\Gamma_{f_{q}, q, H_0}$ coincides with $\MM$ in the sense of Lemma \ref{lem:M_is_locally_a_fuinction_clean}.
Using the triangle inequality for angles (see Figure \ref{fig:bound_f_0_angle}) we get
    \begin{equation}\label{eq:Tri-angle}
    \maxangle(H_0, T_0 f_{0}) \leq 
    \maxangle(T_0 f_{q}, T_0 f_{0}) +
    \maxangle(H_0, T_0 f_{q}) 
    .\end{equation}
    From Lemma \ref{lem:Tf_Tftilde_Init} we have
    \begin{equation}\label{eq:Tri-angle_part1}
    \maxangle(H_0, T_0 f_{q})  \leq \frac{3\alpha_M}{2}
    ,\end{equation}
    where $\alpha_M = \sqrt{\frac{C_\sM}{M}}$.
    Thus, to finish the argument we are left with bounding $\maxangle(T_0 f_{q}, T_0 f_{0}) $.
    
    Denote by $r^*$ the projection of $r$ on the plane $H_0$ that passes through $q$ (see Figure \ref{fig:bound_f_0_angle}). 
    Let $ f_{r^*}:\RR^d\to\RR^{D-d} $ be defined as the function whose graph $ \Gamma_{f_{r^*}, r^*, H_0} $ coincides with $ \MM $ in the sense of Lemma \ref{lem:M_is_locally_a_fuinction_clean}.
    Note that $f_{r^*}$ is just a shifted version of $f_{0}$ (see Figure \ref{fig:bound_f_0_angle}).
    Since $dist(q,r^*) \leq dist(q,r) \leq 2 \sigma$, we can use Lemma \ref{lem:f_bound_circle_H_no_func_ver2_clean}
    and get that
    \begin{equation}\label{eq:f_r_star_to_f_q_star}
    \|(0, f_{r^*}(0))_{(r^*, H_0)}\| - \|(0, f_{q}(0))_{(q, H_0)}\| \leq 2 \sigma (\tan \alpha_M + c \frac{2 \sigma}{\tau}) =  2 \sigma (\tan \alpha_M +  \frac{c_1}{M}) 
    \end{equation}
    and thus, 
    \[
    dist(f_{r^*}(0), f_{q}(0)) \leq \sqrt{2 ^2\sigma^2 (\tan \alpha_M +  \frac{c_1}{M})^2 + (2 \sigma)^2} =  2\sigma\sqrt{ (\tan \alpha_M +  \frac{c_1}{M})^2 + 1}
    \]
    From  Corollary 3 in \cite{boissonnat2017reach}  we have that 
    \begin{equation}\label{eq:Tri-angle_part2}
    \maxangle(T_0 f_{q}, T_0 f_{0}) \leq 
    \frac{1}{2\tau} 2\sigma\sqrt{ (\tan \alpha_M +  \frac{c_1}{M})^2 + 1} \leq  \frac{1}{M} \sqrt{ (\tan \alpha_M +  \frac{c_1}{M})^2 + 1}
    \end{equation}
    Combining \eqref{eq:Tri-angle_part2} with \eqref{eq:Tri-angle_part1} and \eqref{eq:Tri-angle}, we have
    \[
    \maxangle(H_0, T_0 f_{0}) \leq \frac{3\alpha_M}{2} +  \frac{1}{M} \sqrt{ (\tan \alpha_M +  \frac{c_1}{M})^2 + 1} \leq \alpha
    \]
    for $M$ large enough.
    
\end{proof}
\begin{figure}
    \centering
    \includegraphics[width=0.5\linewidth]{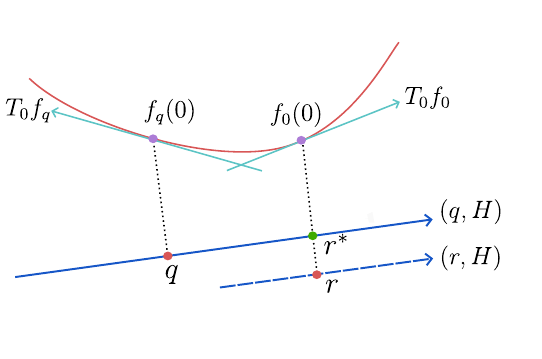}
    \caption{This figure illustrates concepts from the proof of Lemma \ref{lem:H0_Tf_Init} in which we aim to show that angular difference between $T_0 f_0$ (the tangent to the manifold at $f_0(0)$) and  $T_q f_0$ (the tangent to the manifold at $f_q(0)$) is bounded. We denote by $r^*$ the projection of $r$ onto $(q, H)$  }
    \label{fig:bound_f_0_angle}
\end{figure}

\begin{Lemma}\label{lem:Tf_Tftilde_Init}
	Let the sampling assumptions of Section \ref{sec:SamplingAssumptions} hold.
	Let $(q, H_0)\in \MM_\sigma\times Gr(d,D)$ be the initialization of Algorithm \ref{alg:step2_clean}. Assume $\maxangle (H_0, T_p\MM) \leq \alpha_M$, for $p = Proj_\MM(q)$.
	Let $ f_q:\RR^d\to\RR^{D-d} $ be defined as the function whose graph $ \Gamma_{f_q, q, H_0} $ coincides with $ \MM $ in the sense of Lemma \ref{lem:M_is_locally_a_fuinction_clean}; explicitly 
	\[
	\Gamma_{f_q, q_, H_0} = \{q + (x, f_q(x))_{H_0} ~|~ x\in B_{H_0}(q, c_{\pi/4}\tau)\}
	,\]
	where $c_{\pi/4}$ is a constant and $\Gamma_{f_q, q, H_0} \subset \MM$.
	Then, for any $\alpha_M \leq \alpha_c$ where $\alpha_c$ is a constant depending only on $c_{\pi/4}$ of Lemma \ref{lem:M_is_locally_a_fuinction_clean}, for $M=\tau/\sigma$ sufficiently large (see assumption \ref{assume:noise} in Section \ref{sec:SamplingAssumptions}) we have 
	\[
	\maxangle(H_0, T_0 f_q) 
	\leq  \frac{3 \alpha_M}{2}
	.\]
\end{Lemma}
\begin{proof}
    Let $\phi_p: T_p\MM \simeq \RR^d \rightarrow \RR^{D-d}$ be defined in \eqref{eq:phi_def_clean} and \eqref{eq:FunctionGraph_init}. Let $g_0:(q,T_p\MM) \simeq \RR^d \to \RR^{D-d}$ defined as $g_0(x) = \phi_p(x) + (p - q)$ (note that $ (p-q) \in T_p \MM^\perp$ and thus this can be understood with some abuse of notation). From Corollary \ref{cor:GraphOfFunctionTau2_clean}, $\Gamma_{g_0, q, T_p\MM}$ coincides with $\MM \cap \mathrm{Cyl}(p,\tau/2, \tau/2)$.
    We note that $\maxangle (T_0g_0, T_p\MM)=0$,  $\maxangle(H_0, T_p\MM)\leq \alpha_M$ where $\alpha_M \leq \alpha_c$, $\|p-q\| \leq \sigma$ and from assumption \ref{assume:noise} in Section \ref{sec:SamplingAssumptions}, we have that $\sigma \leq \frac{3\tau}{4 \cdot 16}$.  
    Then, we can apply Lemma \ref{lem:shift_ang_general}, with the above defined $g_0$ and $g_1 = f_q$, $G_0 = T_p\MM$, and $G_1 = H_0$ and get
    \[
    \maxangle (H_0, T_0f_q) \leq \maxangle(H_0, T_0g_0)  + \frac{4\|g_0(0)\|\left(2+\|g_0(0)\|/\tau\right)\alpha_M}{\tau}
    .\]
Since $T_0g_0 = T_p\MM$ and $\frac{4\|g_0(0)\|\left(2+\|g_0(0)\|/\tau\right)\alpha}{\tau} \leq \frac{\alpha_M}{2}$ for large enough $M=\tau/\sigma$ (assumption \ref{assume:noise} in Section \ref{sec:SamplingAssumptions}) we have 
\[
    \maxangle (H_0, T_0f_q) \leq \frac{3\alpha_M}{2}.
\]

\end{proof}
\begin{Lemma}\label{lem:dist_to_f_00}
    Let the sampling assumptions of Section \ref{sec:SamplingAssumptions} hold. Let $(q, H_0)\in \MM_\sigma\times Gr(d,D)$ be the initialization of Algorithm \ref{alg:step2_clean}, and let $p = Proj_\MM(q)$, the projection of $q$ onto $\MM$.  Assume $\maxangle (H_0, T_{p_q}\MM) \leq \sqrt{C_\sM/M}$, where $C_\sM$ is a constant independent of $M$ \emph{(}see Theorem \ref{thm:Step1}\emph{)}.
    Let $ f_{0}:\RR^d\to\RR^{D-d} $ be defined 
    as a function whose graph $ \Gamma_{f_{0}, r, H_0} $ coincides with $ \MM $ in the sense of Lemma~\ref{lem:M_is_locally_a_fuinction_clean}.
    Then, for $M$ large enough
	\[
	\| f_0(0) \| \leq 10 \sigma
	,\] 
 where $c$ is some general constant.
\end{Lemma}
\begin{proof}  
    First we show that $\|f_{q}(0)\| \leq 5\sigma $ (see Figure \ref{fig:f_q_star0_bound}).
    From Theorem \ref{thm:Step1} we have that $\|p-q\| \leq 4\sigma$. From Theorem \ref{thm:Step1} we also have $\maxangle(H_0, T_{p}\MM) \leq \alpha_M$, where we denote $\alpha_M =  \sqrt{C_\sM/M}$. 
    Then, denoting by $s$ the intersection between the affine plane parallel to $H_0$ that goes through $p$ and the line between $q$  and $(0, f_{q}(0))_{H_0, q}$ (see Fig. \ref{fig:f_q_star0_bound}), we also get that  $\angle (p, q, s) \leq \alpha_M$ as an angle between two lines in the perpendicular domains of two flats whose maximal principal angle is bounded by $\alpha_M$.
    Thus, 
    \begin{equation}\label{eq:q_star-s}
        \|q-s\| \leq 4\sigma 
    .\end{equation} 
    Similarly, 
    \begin{equation} \label{eq:bound_x_of_dist_to_f_00 }
    x = \|p - s\| \leq 4\sigma \sin \alpha_M.
    \end{equation}

    Denote by $y$ the distance between $s$ and $(0,f_{q}(0))_{H_0, q}$. Using Lemma \ref{lem:f_bound_circle_H_no_func_ver2_clean} we have that 
    \[
    y \leq x (\tan \alpha_M + c_1 x/\tau)
    .\]
    substituting \eqref{eq:bound_x_of_dist_to_f_00 }, we have
    \[
    y \leq 4\sigma \sin \alpha_M (\tan \alpha_M + 4c_1 \sigma \sin \alpha_M/\tau).
    \]
    Combining this with \eqref{eq:q_star-s} we have 
    \begin{equation}\label{eq:f_q_bound}
    \|f_{q}(0)\| \leq 4\sigma+y \leq 5 \sigma 
    \end{equation}
    for $M$ large enough, as $\alpha_M = \sqrt{C_\sM/M}$.

    Now we combine \eqref{eq:f_r_star_to_f_q_star} with  \eqref{eq:f_q_bound} and with the fact that  $dist(q,r^*) \leq 2 \sigma$ we conclude the proof.
\end{proof}
\begin{figure}
    \centering
    \includegraphics[width=0.5\linewidth]{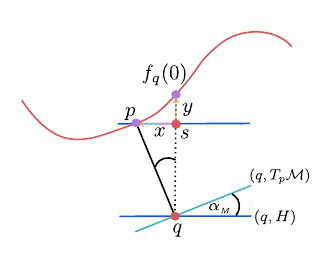}
    \caption{This Figure illustrates concepts from the proof of Lemma \ref{lem:dist_to_f_00}. It demonstrates that the distance between $q$ and $f_q(0)$ is small given that the distance between $q$ and $p=Proj_{\MM} (q)$ is small, and the angle between $H$ and $T_p\MM$ is also small.}
    \label{fig:f_q_star0_bound}
\end{figure}
\begin{Lemma}\label{lem:comute_kappa}
For any $\delta$ and for any $n,\alpha_1,r_1$, Let $C_{0}$ be the constant from Theorem 3.2 of \cite{aizenbud2021VectorEstimation}. We have that 
    \begin{equation}\label{eq:kappa_res_lem}
        \kappa = r_1\log_2(n) + \bar{C}_{\alpha_1,d} - \log\left(\ln\left(\frac{2r_1\log_2(n) + 2\bar{C}_{ \alpha_1,d}}{\delta}\right)\right)
    ,\end{equation}
    where
  \[
    \bar{C}_{ \alpha_1,d} = 1 + \log_2\left(\frac{ \alpha_1}{12\sqrt{d}}\right) -\log_2(C_0 )
    ,
    \]
    satisfies 
    \begin{equation}\label{eq:n_bound1_kappa_lem}
        2^{\kappa-1}C_0 \ln(1/\delta_1) \leq \frac{ \alpha_1}{12\sqrt{d}}n^{r_1},
    \end{equation}
    for $\delta_1 = \frac{\delta}{2\kappa}$, and $C_0$ from Theorem 3.2 of \cite{aizenbud2021VectorEstimation}.
    Furthermore, for $\kappa$ as in \eqref{eq:kappa_res_lem} we have 
    \[
     \alpha_1 2^{-\kappa} \leq C_{d} \ln\left(\frac{1}{\delta}\right) n^{-r_1}\left(\ln\left(\ln(n) \right)\right)^{2r_1}
    \]
\end{Lemma}

\begin{proof}
We find $\kappa$ that will satisfy \eqref{eq:n_bound1_kappa_lem}. Recalling that $\delta_1 = \frac{\delta}{2\kappa}$, we have that
	\[
    C_0 \left(\ln \frac{1}{\delta_1}\right) = C_0 \left(\ln \frac{2\kappa}{\delta}\right).
	\]
Rewriting \eqref{eq:n_bound1_kappa_lem}, we need $\kappa$ to satisfy
    \begin{equation}
        2^{\kappa-1}C_0\left(\ln \frac{2\kappa}{\delta}\right) \leq \frac{ \alpha_1}{12\sqrt{d}}n^{r_1},
    \end{equation}
    or, taking $\log_2$ of both sides, we have
    \begin{equation}
        \kappa - 1 +\log_2(C_0 ) + \log_2\left(\ln \frac{2\kappa}{\delta}\right) \leq \log_2\left(\frac{ \alpha_1}{12\sqrt{d}}\right) + r_1\log_2(n),
    \end{equation}
    or, 
    \begin{equation}\label{eq:kappa_bound1}
        \kappa  + \log_2\left(\ln \frac{2\kappa}{\delta}\right) \leq 1 + \log_2\left(\frac{ \alpha_1}{12\sqrt{d}}\right) -\log_2(C_0 ) +  r_1\log_2(n),
    \end{equation}
    To simplify the expression we denote the RHS of \eqref{eq:kappa_bound1} by $x$. Then we are looking for $\kappa$ such that
    \begin{equation}\label{eq:kappa_bound_simplified_x}
    \kappa+\log_2\left(\ln\left(\frac{2\kappa}{\delta}\right)\right) < x     
    \end{equation}
    We note that 
    \[
    \kappa = x-\log_2\left(\ln\left(\frac{2x}{\delta}\right)\right)
    \]
    satisfies \eqref{eq:kappa_bound_simplified_x} since
    \[
    x-\log_2\left(\ln\left(\frac{2x}{\delta}\right)\right)   +   \log_2\left(\ln\left( \frac{2x-2\log_2\left(\ln\left(\frac{x}{\delta}\right)\right)}{\delta} \right) \right)<x
    \]
    Thus, the following $\kappa$ satisfies Eq. \eqref{eq:kappa_bound1}
    \begin{equation}\label{eq:kappa_res}
        \kappa = r_1\log_2(n) + \bar{C}_{ \alpha_1,d} - \log\left(\ln\left(\frac{2r_1\log_2(n) + 2\bar{C}_{ \alpha_1,d}}{\delta}\right)\right)
    \end{equation}
    where 
    \[
    \bar{C}_{ \alpha_1,d} = 1 + \log_2\left(\frac{ \alpha_1}{12\sqrt{d}}\right) -\log_2(C_0 ).
    \]
    
    We now bound $ \alpha_1 2^{-\kappa}$ by 
    
       \begin{equation}
    \begin{array}{ll}
     \alpha_1 2^{-r_1\log_2(n) - \bar{C}_{\alpha_1,d} + \log\left(\ln\left(\frac{2r_1\log_2(n) + 2 \bar{C}_{\alpha_1,d}}{\delta}\right)\right)} &=   \alpha_1 n^{-r_1}2^{-\bar{C}_{ \alpha_1,d}}\left(\ln\left(\frac{2r_1\log_2(n) + 2\bar{C}_{ \alpha_1,d}}{\delta}\right)\right)
     \\
         & =  \alpha_1 2^{-\bar{C}_{\alpha_1,d}} n^{-r_1}\left(\ln\left(2r_1\log_2(n) + 2\bar{C}_{ \alpha_1,d}\right)+\ln\left(\frac{1}{\delta}\right)\right)\\
         & \leq C_{ \alpha_1,d} \ln\left(\frac{1}{\delta}\right) n^{-r_1}\left(\ln\left(\ln(n) \right)\right)^{2r_1}
    \end{array}
    \end{equation}
    Since $\alpha_1$ is bounded from above, there is  $ C_{d}$ independent of $\alpha_1$ for which
    \begin{equation}
      \alpha_1 2^{-\kappa} = \alpha_1 2^{-r_1\log_2(n) - \bar{C}_{ \alpha_1,d} + \log\left(\ln\left(\frac{2r_1\log_2(n) + 2 \bar{C}_{ \alpha_1,d}}{\delta}\right)\right)} \leq C_{ d} \ln\left(\frac{1}{\delta}\right) n^{-r_1}\left(\ln\left(\ln(n) \right)\right)^{2r_1}
    \end{equation}
\end{proof}

\subsection{Proof of Lemma \ref{lem:main_support_theorem_step2}}
 \label{subsubsec:proof_of_main_support_lemma}


\begin{proof}
From Lemma \ref{lem:Tf_TtildeF_D} we have that 
\begin{equation}\label{eq:bound_Tfl_Tftile_l}
\maxangle(T_0 f_\ell, T_0 {\wtilde{f}_\ell}) 
\leq  
\frac{\alpha}{6}.    
\end{equation}

From Lemma \ref{lem:AngleImprovement} we have that for any $\delta$ and $n>N_{\delta}$, 
\[
\Pr(\maxangle(T_0\wtilde f_\ell, T_0\pi^*_{q_0, H_\ell}) \geq 2\sqrt{d}\frac{C_0 \ln(1/\delta)}{n^{r_1}}) < \delta.
\]
Thus, as $ 12\sqrt{d} \frac{C_0 \ln(1/\delta)}{n^{r_1}}\leq \alpha $  (see assumption 3 in the lemma), we have 
\begin{equation}\label{eq:pr_alpha_6}
\Pr(\maxangle(T_0\wtilde f_\ell, T_0\pi^*_{q_0, H_\ell}) \geq \frac{\alpha}{6}) < \delta.
\end{equation}
And, from
\[
\maxangle (T_0f_{\ell}, H_{\ell+1}) \leq \maxangle (T_0f_\ell, T_0\wtilde{f}_\ell)  + \maxangle (T_0\wtilde{f}_\ell, H_{\ell+1}). 
\]
along with \eqref{eq:bound_Tfl_Tftile_l} and \eqref{eq:pr_alpha_6}, we have that 
\begin{equation}\label{eq:pr_alpha_3}
\Pr\left(\maxangle (T_0f_{\ell}, H_{\ell+1}) \geq \alpha/3\right) < \delta.    
\end{equation}

In order to apply Lemma \ref{lem:shift_ang_general} we denote $g_0 = f_\ell$, $g_1 = f_{\ell+1/2}$, $G_0 = H_\ell$, and $G_1 = H_{\ell+1}$.
Thus, we have that $\maxangle (T_0g_0, G_0) = \maxangle (T_0f_\ell, H_\ell)\leq \alpha$ and  $\maxangle(G_0,G_1) = \maxangle(H_\ell,H_{\ell+1})\leq \maxangle(H_\ell,T_0f_\ell) + \maxangle(T_0f_\ell,H_{\ell+1})$.
Under the assumption that the event described in \eqref{eq:pr_alpha_3} holds we also know that $\maxangle(G_0,G_1)  \leq \alpha + \alpha/3 = \beta$.
Lastly, as $\alpha \leq \pi/16$, $\beta = \frac{4\alpha}{3} \leq \alpha_c$, and  $\|f_\ell(0)\| \leq c_{\pi/4} \cdot \tau $ (follows from assumption 2 of the lemma), we can use the result of Lemma  \ref{lem:shift_ang_general} and obtain 

\begin{align*}
\|f_{\ell+1}(0)\| &\leq \|f_{\ell}(0)\|(1+(\alpha+\beta)^2 (3+ 2\|f_{\ell}(0)\|/\tau) ) 
\\&
\leq \|f_\ell(0)\|(1+(7/3\alpha)^2 (3+ 2\|f_\ell(0)\|/\tau) ) 
\\&
\leq \|f_\ell(0)\|\left(1+10\alpha^2 \left(3+ 2\sqrt{\frac{1}{32\tau}}\right) \right)
&& \left(\text{since } \|f_\ell(0)\| \leq \sqrt{\tau/32} \right)
\\&
\leq 
\|f_\ell(0)\|\left(1+10\alpha^2 \left(3+ \sqrt{\frac{1}{8\tau}}\right) \right)
\\&
\leq \|f_\ell(0)\|\left(1+40\alpha^2\right) && \left(\text{for } \tau \text{ large enough}\right)
\end{align*}

and 
\[
\begin{aligned}
    \alpha_{\ell+1} &= \maxangle (T_0f_{\ell+1}, H_{\ell+1}) = \maxangle (T_0f_{\ell+1/2}, H_{\ell+1}) \\&\leq \maxangle(T_0f_\ell,H_{\ell+1})  +
    \frac{4\|f_\ell(0)\|\left(2+\|f_\ell(0)\|/\tau\right) \beta} {\tau}.
    \\&= \maxangle(T_0f_\ell,H_{\ell+1})  +
    \frac{16\|f_\ell(0)\|\left(2+\|f_\ell(0)\|/\tau\right) \alpha} {3\tau}.
\end{aligned}
\]
Note that for $\|f_\ell(0)\| \leq\sqrt{\tau/32} -1$, we have $\frac{16\|f_\ell(0)\|\left(2+\|f_\ell(0)\|/\tau\right)} {3\tau}\leq 1/6$.
Furthermore, since $\maxangle(T_0f_\ell,H_{\ell+1}) \leq \alpha/3$, we have 
\[
\maxangle (T_0f_{\ell+1}, H_\ell)  = \maxangle (T_0f_{\ell+1/2}, H_\ell) \leq \alpha/2
.\]
\end{proof}

\subsection{Supporting lemmas for Lemma \ref{lem:main_support_theorem_step2}}
\label{sec:ProofofLemAlphaHalf}

\subsubsection{Bounding the error between $T_0 {f_\ell}$ and $T_0\wtilde{f}_\ell$}\label{sec:Tf_Tf_tilde_err}
\usetikzlibrary{calc}

\begin{figure}[H]
    \centering
    \adjustbox{center}{%
    \begin{tikzpicture}[auto]
        \node [Rbbblock] (main_main_lem) {L \textcolor{blue}{\ref{lem:Tf_TtildeF_D}}
        \[
        \maxangle(T_0 f_\ell, T_0 {\wtilde{f}_\ell}) 
        \leq  
         \frac{\alpha}{6}
        \]};
        
        \node [bbblock, below =1cm of main_main_lem] (main_lem) { L \textcolor{blue}{\ref{lem:bound_Df-Dfwtilde}}
        \vspace{-0.3cm}
         \begingroup\makeatletter\def\f@size{7}\check@mathfonts
        \[
        \|\DD\wtilde{f}(0) - \DD f(0)\| < \eps \alpha
        \]
        \endgroup
        };

        \draw [imp] (main_lem) to[out=90,in=-90,looseness=0.73] (main_main_lem) node [midway, below, sloped] (TextNode) {};
       
        \node [bbblock, below right =1cm and 1cm of main_main_lem] (Diff_Tf_bound) {L \textcolor{blue}{\ref{lem:Diff_Tf_bound}}
        \vspace{-0.3cm}
        \begingroup\makeatletter\def\f@size{7}\check@mathfonts
        \begin{gather*}
    	\norm{\DD_f[0] - \DD_{\tilde f}[0]}_{op}\leq \varepsilon 
    	\\
    	\Downarrow
    	\\
    	\sin(\maxangle(T_0 f, T_0 \wtilde{f})) \leq  \varepsilon 
    	\end{gather*}
    	\endgroup
	};
        
 	   \draw [imp] (Diff_Tf_bound) to[out=90,in=-90,looseness=0.73] (main_main_lem) node [midway, below, sloped] (TextNode) {};

        \node [bbblock, below left = 1cm and 1cm of main_lem] (bound_g) {L \textcolor{blue}{\ref{lem:g_bond_weak}}
        \vspace{-0.3cm}
         \begingroup\makeatletter\def\f@size{7}\check@mathfonts
         \[
        \sigma \leq g(0,\theta) \leq \sigma +  4\sigma \alpha^2
        \]
        \endgroup}; 
        
         \draw [imp] (bound_g) to[out=90,in=-90,looseness=0.73] (main_lem) node [midway, below, sloped] (TextNode) {};

        \node [bbblock, below =1cm of main_lem] (bound_partI) {L \textcolor{blue}{\ref{lem:bound_nablag}}
        \vspace{-0.3cm}
         \begingroup\makeatletter\def\f@size{7}\check@mathfonts
         \[
        \|\nabla_x{g}(0,\theta)\| =  \OO(\frac{\sigma}{\tau}\sin(\alpha))
        \]\endgroup}; 
        
         \draw [imp] (bound_partI) to[out=90,in=-90,looseness=0.73] (main_lem) node [midway, below, sloped] (TextNode) {};

        \node [bbblock, below right =1.5cm and -0.7cm of bound_partI] (bound_J) {L \textcolor{blue}{\ref{lem:bound_J}}
        \vspace{-0.3cm}
        \begingroup\makeatletter\def\f@size{7}\check@mathfonts
        \[
            \|I-J_{\wtilde{x}_\theta}\| = \frac{3\sigma}{\tau} + \OO(\frac{\sigma\sin \alpha}{\tau})
        \]
        \endgroup};

        \node [bbblock, below =0.5cm of bound_J] (Jacobi_bound_clean) {L \textcolor{blue}{\ref{lem:bound_A_inv_B}}
        }; 
        
        \node [bbblock, below right =1cm and -2cm of Jacobi_bound_clean] (Jacobi_expression_clean) {L \textcolor{blue}{\ref{lem:Dxtildew+Dxw}}
        \vspace{-0.3cm}
        \begingroup\makeatletter\def\f@size{7}\check@mathfonts
        \begin{equation*}
        \begin{array}{l}
            \| D_{\wtilde{x}_\theta} w + D_x w\| \leq \mathcal{O}(\frac{\sigma}{\tau}\sin \alpha)\\
            \| D_{\wtilde{x}_\theta} w \| \leq  \mathcal{O}(\sin \alpha)\\ \|D_x w\| \leq \mathcal{O}(\sin \alpha)\\
            1/2 \sigma^2 \leq \Delta(0,\wtilde x_\theta (0)) \leq 2 \sigma^2
        \end{array}
        \end{equation*}
        \endgroup
        }; 
        
        \node [bbblock, below left =1cm and -2cm of Jacobi_bound_clean] (hessian_bounds_clean) {L \textcolor{blue}{\ref{lem:bound_DDf}}  
        \vspace{-0.3cm}
        \begingroup\makeatletter\def\f@size{7}\check@mathfonts
        \[
        \HH f(\wtilde x_\theta)_w = \frac{\sqrt{2}\sigma}{\tau} + \OO(\frac{\sigma\sin \alpha}{\tau})
        \]
        \endgroup};
        
        \node [xbblock,text width=13em, below right =5.5cm and -2.5cm of Jacobi_bound_clean] (Df_Dftilde_diff) {L \textcolor{blue}{\ref{lem:Dfx-Dfxtilde}}
        \vspace{-0.3cm}
        \begingroup\makeatletter\def\f@size{7}\check@mathfonts
        \[
        \|Df(\wtilde{x}_\theta(0)) - Df(0)\|_{op} \leq \OO(\frac{\sigma}{\tau}\sin \alpha)
        \]
        \endgroup
};
        \draw [imp] (bound_J) to[out=90,in=-90,looseness=0.73] (bound_partI) node [midway, below, sloped] (TextNode) {};
        
        \draw [imp] (Jacobi_bound_clean) to[out=90,in=-90,looseness=0.73] (bound_J) node [midway, below, sloped] (TextNode) {};

        \draw [imp] (Jacobi_expression_clean) to[out=90,in=-90,looseness=0.73] (Jacobi_bound_clean) node [midway, below, sloped] (TextNode) {};
        
        \draw [imp] (Jacobi_expression_clean) to[out=90,in=0,looseness=0.73] (bound_J) node [midway, below, sloped] (TextNode) {};
        \draw [imp] (hessian_bounds_clean) to[out=90,in=-90,looseness=0.73] (Jacobi_bound_clean) node [midway, below, sloped] (TextNode) {};
        
        \draw [imp] (Df_Dftilde_diff)  to[out=90,in=-90,looseness=0.73] (Jacobi_expression_clean) node [midway, below, sloped] (TextNode) {};
        
        \draw [imp] (Jacobi_expression_clean)   to[out=90,in=-90,looseness=0.73] (7.2,-8) to[out=90,in=0,looseness=0.7] (bound_partI) node [midway, below, sloped] (TextNode) {};

        \draw [imp] (Df_Dftilde_diff)  to[out=90,in=-90,looseness=0.73] (9.5,-14.3) to[out=90,in=-90,looseness=0.73] (9.5,-8) to[out=90,in=0,looseness=0.7] (bound_partI) node [midway, below, sloped] (TextNode) {};

          \node [xbblock,text width=13em, below left =2.42cm and -1cm of bound_partI] (Grad_fl_bounds) {L  \textcolor{blue}{\ref{lem:bounding_stuff}}
         \vspace{-0.3cm} 
         \begingroup\makeatletter\def\f@size{7}\check@mathfonts
          \begin{equation*}
          \begin{array}{l}
            \|\wtilde{x}_\theta(0)\| \leq \sigma \sin (\alpha + c\frac{\sigma}{\tau}\alpha)
            \\
            \|f(\wtilde{x}_\theta(0))\| \leq 2\sigma \sin (\alpha + c\frac{\sigma}{\tau}\alpha) \tan \alpha 
            \\
            \|Df(0)\|_2 \leq \sin \alpha    
            \\
            \|Df(\wtilde{x}_\theta(0))\|_2 \leq \sin (\alpha + c\frac{\sigma}{\tau}\alpha)
          \end{array}
          \end{equation*}
          \endgroup
            }; 

         \draw [imp] (Grad_fl_bounds) to[out=90,in=-90,looseness=0.73] (bound_partI) node [midway, below, sloped] (TextNode) {};
         
         \draw [imp] (Grad_fl_bounds) to[out=90,in=-90,looseness=0.73] (bound_g) node [midway, below, sloped] (TextNode) {};
         
         \draw [imp] (Grad_fl_bounds) to[out=0,in=180,looseness=0.73] (Jacobi_bound_clean) node [midway, below, sloped] (TextNode) {};

         \node [bbblock, below right =1.6cm and -2.3cm of Grad_fl_bounds] (sectional_derivative_diff_bound) {L \textcolor{blue}{\ref{lem:wtilde_x_bound_clean_D}}
         \vspace{-0.3cm}
         \begingroup\makeatletter\def\f@size{7}\check@mathfonts
         \[
         \|\wtilde{x}_\theta(0)\| \leq \sigma \sin (\alpha + c\frac{\sigma}{\tau}\alpha)
         \]   \endgroup}; 
         
         \node [bbblock, below =4cm of Grad_fl_bounds] (Grad_f_bound) {L \textcolor{blue}{\ref{lem:x_tilde_bound_beta_D}}
         \vspace{-0.3cm}
          \begingroup\makeatletter\def\f@size{7}\check@mathfonts
          \[
        \norm{\wtilde x_\theta(0) } \leq \sigma \sin \beta(\wtilde x_\theta(0))
        \] \endgroup}; 

         \draw [imp] (sectional_derivative_diff_bound) to[out=90,in=-90,looseness=0.73] (Grad_fl_bounds) node [midway, below, sloped] (TextNode) {};

         \draw [imp] (Grad_f_bound) to[out=90,in=-90,looseness=0.73] (sectional_derivative_diff_bound) node [midway, below, sloped] (TextNode) {};

         \node [bbblock, below left =1.6cm and -2.3cm of Grad_fl_bounds] (wtilde_x_bound_clean) {L \textcolor{blue}{\ref{lem:alpha_beta_D}}
         \vspace{-0.3cm}
         	 \begingroup\makeatletter\def\f@size{7}\check@mathfonts
         	 \[
        	\alpha_\ell - c_2\alpha_\ell\frac{\sigma}{\tau}  \leq \beta(x_0) \leq \alpha_\ell + c_1\alpha_\ell\frac{\sigma}{\tau}
        	\]\endgroup
	        };

         \draw [imp] (wtilde_x_bound_clean) to[out=90,in=-90,looseness=0.73] (Grad_fl_bounds) node [midway, below, sloped] (TextNode) {};
          \draw [imp] (Grad_f_bound) to[out=90,in=-90,looseness=0.73] (wtilde_x_bound_clean) node [midway, below, sloped] (TextNode) {};

        \draw [imp] (wtilde_x_bound_clean) to[out=0,in=180,looseness=0.73] (sectional_derivative_diff_bound) node [midway, below, sloped] (TextNode) {};
    \end{tikzpicture}%
    }
    \caption{Road-map for proof of Lemma  \ref{lem:Tf_TtildeF_D}}
    \label{tikz:lemTf-Tftilde_lemmas}
\end{figure}
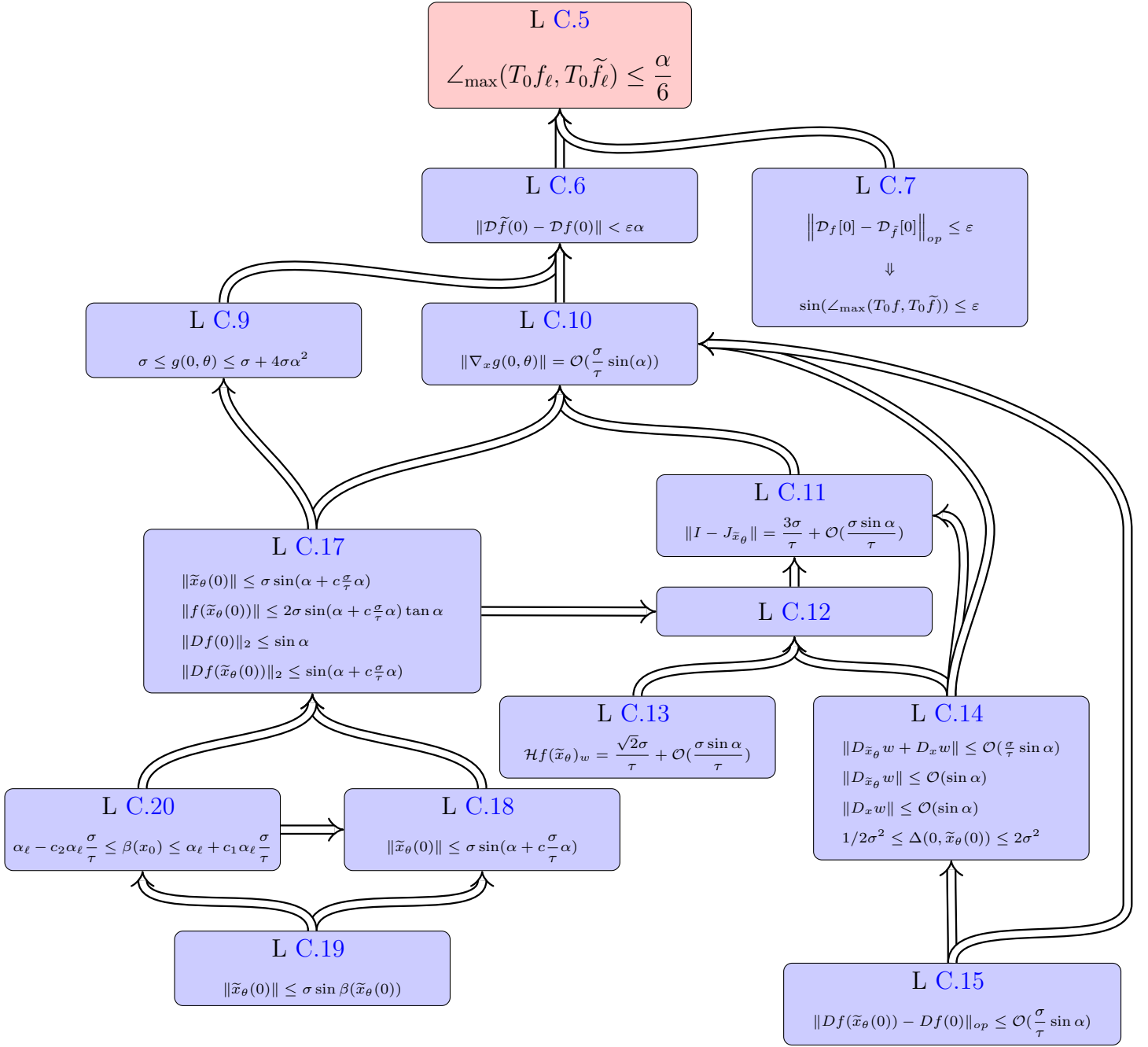

Back to Theorem \ref{thm:Step2} proof road-map see Figure \ref{tikz:thm33_lemmas}.

The main Lemma in the Section is Lemma \ref{lem:Tf_TtildeF_D}. A road-map for the proof appears in Figure~\ref{tikz:lemTf-Tftilde_lemmas}.

\begin{Lemma}\label{lem:Tf_TtildeF_D}
    Let the assumptions of Theorem \ref{thm:Step2} hold, and let $ f:H\simeq\RR^d\to\RR^{D-d} $ be a function such that its graph coincides with a neighborhood on the manifold $\MM$ (see Lemma \ref{lem:M_is_locally_a_fuinction_clean}). Let $\wtilde f:\RR^d\to\RR^{D-d} $ be the regression function defined in \eqref{eq:ftilde_def}. Denoting $\alpha = \maxangle(H, T_0 f)$, there is a constant $C_\stau$ large enough such that for $M = \frac{\tau}{\sigma}>C_\stau\sqrt{D \log D}$, and $\alpha \leq \sqrt{1/D}$, where $C_\sM$ is from Theorem \ref{thm:Step1}, we have 
    \[
	\maxangle(T_0 f, T_0 {\wtilde{f}}) 
	\leq  
	 \frac{\alpha}{6} 
	\]
\end{Lemma}
\begin{proof}
From Lemma \ref{lem:bound_Df-Dfwtilde} we have 
\[
\|\DD_{\wtilde{f}}[0] - \DD_f[0]\| < \frac{\alpha}{6}
\]
From Lemma \ref{lem:Diff_Tf_bound} we have that 
    \[
	\maxangle(T_0 f, T_0 {\wtilde{f}}) 
	\leq   \frac{\alpha}{6}
	\]
	and the proof is concluded. 
\end{proof}

\begin{Lemma}\label{lem:bound_Df-Dfwtilde}
 Let the assumptions of Theorem \ref{thm:Step2} hold, and let $ f:H\simeq\RR^d\to\RR^{D-d} $ be a function such that its graph coincides with a neighborhood on the manifold $\MM$ (see Lemma \ref{lem:M_is_locally_a_fuinction_clean}). Let $\wtilde f:\RR^d\to\RR^{D-d} $ be the regression function   defined in \eqref{eq:ftilde_def}. 
 For any $\eps>0$, denoting $\alpha = \maxangle(H, T_0 f)$, there is a constant $C_\stau$ large enough such that for $M = \frac{\tau}{\sigma} >C_\stau\sqrt{D \log D}$, and $\alpha \leq \sqrt{1/D}$, where $C_\sM$ is from Theorem \ref{thm:Step1}, we have 
    \[
    \|\DD\wtilde{f}(0) - \DD f(0)\| < \eps \alpha
    \]
\end{Lemma}
\begin{proof}
We reiterate the definition of $\Omega(x):H \simeq \RR^d \to 2^{\RR^{D-d}}$ from \eqref{eq:Omega_def}
\begin{equation*}
\Omega(x) = (x + H^\perp) \cap  \MM_\sigma
,    
\end{equation*}
where $x + H^\perp = \{x + y ~|~ y\in H^\perp\}$.
Next, denoting $S_{D-d}$ to be the $D-d$ dimensional unit sphere, we define $g(x, \theta):H\times S_{D-d}\to \RR$ the maximal length from $f(x)$ in the direction $\theta$ that is inside $\Omega(x)$.
Explicitly,
\begin{equation}\label{eq:g_def}
g(x, \theta) = \max\{y\in \RR ~|~ (x,f(x) + y\cdot \theta)_H  \in \Omega(x)\}
.    
\end{equation}
Note, that the farthest point from $f(x)$ in $\Omega(x)$ at each direction $\theta$, by which we define $g(x, \theta)$, belongs to $\partial\MM_\sigma$ (the boundary of the domain $\MM_\sigma$), and is therefore exactly $\sigma$ away from some point on the manifold itself.
Since we are viewing the manifold locally as the graph of the function $f:H\to H^\perp$ we denote this point by $(\wtilde x_\theta, f(\wtilde x_\theta(x)))_H$.
Explicitly, $\wtilde{x}_\theta:H \to H$ is such that
\begin{equation}\label{eq:x_tilde_def}
    (x ,f(x) + g(x, \theta) \cdot \theta)_H^T  = (\wtilde x_\theta(x), f(\wtilde x_\theta(x)))_H^T + \sigma \vec N_\theta (x)
,\end{equation}
where $\vec N_\theta \in \RR^D$ is perpendicular to $T_{\wtilde x_\theta}f$.
We introduced the definition of $\wtilde{x}_\theta$ here as it will be pivotal in the proofs of Lemmas \ref{lem:g_bond_weak} and \ref{lem:bound_nablag} upon which the the current proof relies.
Furthermore, we wish to stress here that by Lemma \ref{lem:bounding_stuff} $\wtilde x_\theta(0) \leq \sigma \sin (\alpha + c\alpha\sigma/\tau)$ for some general constant $c$.
Therefore, for sufficiently large $M = \frac{\tau}{\sigma}$, $\wtilde x_\theta(0)$ is within the domain of definition of the function $f$ which by Lemma \ref{lem:M_is_locally_a_fuinction_clean} is of radius of at least $c_{\pi/4}\tau$.

Next, by the definitions and Eq. \eqref{eq:g_def} and \eqref{eq:Omega_def} we have that 
\begin{equation*}
\wtilde{f}(x) - f(x) = \frac{\int\limits_{y\in\Omega(x)} y dy}{\int\limits_{y\in\Omega(x)} dy} 
.\end{equation*}
Since $\Omega \subset \RR^D$ is perpendicular to $H\simeq \RR^d$ we get that $\Omega\simeq \RR^{D-d}$.
Thus, by change of variables we can breakdown the integrals over $\Omega(x)$ to a radial component $r$ and directions on the $(D-d-1)$-dimensional sphere.
Explicitly,
\begin{equation*}
\wtilde{f}(x) - f(x) = 
\frac{\int\limits_{S_{D-d-1}}\int\limits_0^{g(x,\theta)} \theta r r^{D-d-1} drd\theta}{\int\limits_{S_{D-d-1}}\int\limits_0^{g(x,\theta)} r^{D-d-1} drd\theta}  = 
\frac{(D-d)\int\limits_{S_{D-d-1}}\theta g(x,\theta)^{D-d+1} d\theta}{(D-d+1)\int\limits_{S_{D-d-1}}g(x,\theta)^{D-d} d\theta}     
,\end{equation*}
where $dr$ is the measure over the radial component, $r^{D-d-1}$ is the Jacobian introduced by the change of variables and $d\theta$ is the measure over the $(D-d-1)$-dimensional sphere.
For brevity we introduce the notation $\wtilde D = D - d$ and get
\begin{equation}\label{eq:f_tilde-f_integrals}
\wtilde{f}(x) - f(x) = 
\frac{\int\limits_{\mathbb{S}_{\wtilde D-1}}\int\limits_0^{g(x,\theta)} \theta r r^{\wtilde D-1} drd\theta}{\int\limits_{\mathbb{S}_{\wtilde D-1}}\int\limits_0^{g(x,\theta)} r^{\wtilde D-1} drd\theta}  = 
\frac{\wtilde D \int\limits_{\mathbb{S}_{\wtilde D-1}}\theta g(x,\theta)^{\wtilde D+1} d\theta}{(\wtilde D+1)\int\limits_{\mathbb{S}_{\wtilde D-1}}g(x,\theta)^{\wtilde D} d\theta}     
.\end{equation}
 
Next, by taking the differential of that expression with respect to $x$ we have that 
\begin{multline*}
    \DD_{\wtilde{f}}[x] - \DD_f[x] =
\frac{\wtilde D}{(\wtilde D+1)}\left(
\frac{(\wtilde D+1)\int\limits_{\mathbb{S}_{\wtilde D-1}}\theta g(x,\theta)^{\wtilde D} \nabla_x g^T d\theta}
{\int\limits_{\mathbb{S}_{\wtilde D-1}}g(x,\theta)^{\wtilde D} d\theta} \right.\\
-\left.
\frac{\wtilde D\int\limits_{\mathbb{S}_{\wtilde D-1}}\theta g(x,\theta)^{\wtilde D+1} d\theta \int\limits_{\mathbb{S}_{\wtilde D-1}} g(x,\theta)^{\wtilde D-1}\nabla_x g^T d\theta}
{\left(\int\limits_{\mathbb{S}_{\wtilde D-1}}g(x,\theta)^{\wtilde D} d\theta\right)^2}     
\right)
,\end{multline*}
where $\nabla_x g$ stands for the gradient of $g(\theta, x)$ with respect to the $x$ variables only.
As can be seen in the above equation there is a multiplicative factor of size $\approx (D - d)$ for both summends.
In order to deal with this obstacle we wish to utilize the fact that in high dimensions most of the volume of a sphere is concentrated near an equator.
Thus, we split the domain into two different regions that we will deal with separately (we will use this trick in few of the other proofs). 
In case $D-d$ is small, then the following computations can be done without splitting the domain into two regions (one near the equator and the the second being the remaining cap), and include the $D-d$ factor in the constant that will be cancelled by $M$.
Therefore, we assume without losing the generality of our claim that $D-d = \wtilde D > 3$.
For any direction/unit vector $\vec{z}\in S_{\wtilde D}$ denote $\vec{z}^T\theta = z$ and 
\begin{equation}\label{eq:Omegas_def}
    \begin{aligned}
    \Omega_1 &= \{\theta~|~0\leq |\vec{z}^T\theta| \leq \xi\}\\
    \Omega_2 &= \{\theta~|~|\vec{z}^T\theta| > \xi\}
    \end{aligned}
\end{equation}
For some $\xi$ to be chosen later. Using the above notation, we have 
\begin{multline}\label{eq:zDftilde}
    \vec{z}^T\cdot  (\DD_{\wtilde{f}}[x] - \DD_f[x] )=
\frac{\wtilde D}{(\wtilde D+1)}\left(
\underbrace{\frac{(\wtilde D+1)\int\limits_{\mathbb{S}_{\wtilde D-1}}z g(x,\theta)^{\wtilde D} \nabla_x g^T d\theta}
{\int\limits_{\mathbb{S}_{\wtilde D-1}}g(x,\theta)^{\wtilde D} d\theta}  }_{I} \right.  \\
-\left.
\underbrace{\frac{\wtilde D\int\limits_{\mathbb{S}_{\wtilde D-1}}z g(x,\theta)^{\wtilde D+1} d\theta \int\limits_{\mathbb{S}_{\wtilde D-1}} g(x,\theta)^{\wtilde D-1}\nabla_x g^T d\theta}
{\left(\int\limits_{\mathbb{S}_{\wtilde D-1}}g(x,\theta)^{\wtilde D} d\theta\right)^2}}_{II}     
\right)
\end{multline}
First we treat part $(I)$ of Eq. \eqref{eq:zDftilde} by splitting the domain into $\Omega_1$ and $\Omega_2$ of Eq. \ref{eq:Omegas_def}.

\begin{align*}
    (I) &= \frac{(\wtilde D+1)\left(\int\limits_{\Omega_1}z g(x,\theta)^{\wtilde D} \nabla_x g^T d\theta + \int\limits_{\Omega_2}z g(x,\theta)^{\wtilde D} \nabla_x g^T d\theta \right) }
{\int\limits_{\mathbb{S}_{\wtilde D-1}}g(x,\theta)^{\wtilde D} d\theta} 
\end{align*}

From \ref{lem:g_bond_weak} we have that $0 < \sigma \leq g(0,\theta)\leq \sigma + 4\sigma\alpha^2$, and from Lemma \ref{lem:bound_nablag} we have $\|\nabla_x g(0,\theta)\|\leq c_1 \sigma \alpha/\tau$. Then,
\begin{align*}
\norm{(I)} &\leq \frac{(\wtilde D+1)\left(\xi\int\limits_{\Omega_1} g(x,\theta)^{\wtilde D} \|\nabla_x g^T\| d\theta + \int\limits_{\Omega_2} g(x,\theta)^{\wtilde D} \|\nabla_x g^T\| d\theta \right)}
{\int\limits_{\mathbb{S}_{\wtilde D-1}}g(x,\theta)^{\wtilde D} d\theta} \\
&\leq  \frac{(\wtilde D+1)c_1 \sigma \alpha \left(\xi\int\limits_{\Omega_1} g(x,\theta)^{\wtilde D}  d\theta + \int\limits_{\Omega_2} g(x,\theta)^{\wtilde D}  d\theta \right)}
{\tau\int\limits_{\mathbb{S}_{\wtilde D-1}}g(x,\theta)^{\wtilde D} d\theta}
\\
&\leq  \frac{\xi(\wtilde D+1)c_1 \sigma \alpha}{\tau} + \frac{(\wtilde D+1)c_1 \sigma \alpha \int\limits_{\Omega_2} g(x,\theta)^{\wtilde D}  d\theta }
{\tau\int\limits_{\mathbb{S}_{\wtilde D-1}}g(x,\theta)^{\wtilde D} d\theta}
\\
&\leq  \frac{\xi(\wtilde D+1)c_1 \sigma \alpha}{\tau} + \frac{(\wtilde D+1)c_1 \sigma \alpha ( \sigma + 4\sigma\alpha^2)^{\wtilde D}\int\limits_{\Omega_2}  d\theta }
{\tau \sigma^{\wtilde D}\int\limits_{\mathbb{S}_{\wtilde D-1}} d\theta}
\end{align*}

Furthermore, using the following concentration of measure inequality (see e.g. \cite{milman2009asymptotic,Venkatesan2012high})
\begin{equation}\label{eq:Omega3_to_ball_ratio}
\frac{\int\limits_{\Omega_2}  d\theta}{\int\limits_{\mathbb{S}_{\wtilde D-1}}  d\theta} \leq \frac{2}{\xi\sqrt{\wtilde D-2}}e^{-(\wtilde D - 2 )\xi^2/2}
,\end{equation}
we have that,
\begin{align*}
    \norm{(I)} 
&\leq  \frac{\xi(\wtilde D+1)c_1 \sigma\alpha}{\tau} + \frac{(\wtilde D+1)c_1 \sigma\alpha (1 + 4\alpha^2)^{\wtilde D} }
{\tau}\frac{2}{\xi\sqrt{\wtilde D-2}}e^{-(\wtilde D - 2 )\xi^2/2}
\\
&\leq  \frac{\xi(\wtilde D+1)c_1 \sigma\alpha}{\tau} + \frac{c_2 \sigma\alpha \sqrt{\wtilde D+1}( 1 + 4\alpha^2)^{\wtilde D} }
{\tau}\frac{2}{\xi}e^{-(\wtilde D - 2 )\xi^2/2}
,\end{align*}
where $c_2 = 4c_1 \geq \frac{\sqrt{\wtilde D +1}}{\sqrt{\wtilde D - 2}} c_1$.

Since $M = \frac{\tau}{\sigma}>C_\stau\sqrt{D \log D}$, we choose $\xi = 2\sqrt{\log D/D}$ and we have

\begin{align*}
    \norm{(I)} 
&\leq \frac{2c_1 \alpha(\wtilde D+1)}{C_\stau D} + \frac{c_2 \alpha \sqrt{\wtilde D+1}( 1 + 4\alpha^2)^{\wtilde D} }
{C_\stau\sqrt{D \log D}}\frac{2}{\xi}e^{-(\wtilde D - 2 )\xi^2/2}
\\
&\leq \alpha\left(\frac{2c_1 }{C_\stau} + \frac{c_1 \sqrt{\wtilde D+1}( 1 + 4\alpha^2)^{\wtilde D} } 
{C_\stau\sqrt{D\log D }}\frac{4\sqrt{D}}{\sqrt{\log D}}e^{-\log D}\right)
\\
&\leq \alpha\left(\frac{2c_1 }{C_\stau} + \frac{2c_1 \sqrt{D}( 1 + 4\alpha^2)^{\wtilde D} }
{C_\stau D\log D }\right)
,\end{align*}
where the second inequality is true since  $(\wtilde D-2)/D \geq 1/2$.
Since $\alpha \leq  \sqrt{1/D} $, and we also have that $M>C_\stau\sqrt{D \log D}$ we have that 
$( 1 + 4\alpha^2)^{\wtilde D}$ is bounded for any $D$. Thus, we have that for $C_\stau$ large enough, 
\begin{align}
\norm{(I)} 
&\leq \frac{\eps}{2}\alpha
\end{align}
for any $D$.

Next we bound $\|(II)\|$ from \eqref{eq:zDftilde}:
\begin{align*}
\frac{\wtilde D\int\limits_{\mathbb{S}_{\wtilde D-1}}z g(x,\theta)^{\wtilde D+1} d\theta \int\limits_{\mathbb{S}_{\wtilde D-1}} g(x,\theta)^{\wtilde D-1}\nabla_x g^T d\theta}
{\left(\int\limits_{\mathbb{S}_{\wtilde D-1}}g(x,\theta)^{\wtilde D} d\theta\right)^2}
\end{align*}
First, we note that 
\begin{align*}
\norm{\frac{\int\limits_{\mathbb{S}_{\wtilde D-1}} g(x,\theta)^{\wtilde D-1}\nabla_x g^T d\theta}
{\int\limits_{\mathbb{S}_{\wtilde D-1}}g(x,\theta)^{\wtilde D} d\theta}} & \leq \frac{c_1 \sigma\alpha}{\tau} \frac{\int\limits_{\mathbb{S}_{\wtilde D-1}} g(x,\theta)^{\wtilde D-1} d\theta}
{\int\limits_{\mathbb{S}_{\wtilde D-1}}g(x,\theta)^{\wtilde D-1} d\theta} = \frac{c_1 \sigma \alpha}{\tau}.
\end{align*}
Thus, similarly to the way we bounded $(I)$, we split the domain to $\Omega_1$ and $\Omega_2$ of Eq. \eqref{eq:Omega_def} and achieve
\begin{align*}
\norm{(II)} &\leq \frac{\wtilde Dc_1 \sigma\alpha}{\tau}\frac{\int\limits_{\mathbb{S}_{\wtilde D-1}}z g(x,\theta)^{\wtilde D+1} d\theta}
{\int\limits_{\mathbb{S}_{\wtilde D-1}}g(x,\theta)^{\wtilde D} d\theta}
\\
&\leq \frac{\wtilde Dc_1 \sigma\alpha}{\tau} \frac{ \left(\xi\int\limits_{\Omega_1} g(x,\theta)^{\wtilde D+1}  d\theta + \int\limits_{\Omega_2} g(x,\theta)^{\wtilde D+1}  d\theta \right)}
{\int\limits_{\mathbb{S}_{\wtilde D-1}}g(x,\theta)^{\wtilde D} d\theta}
\\
&\leq  \frac{\xi\wtilde Dc_1 \sigma\alpha}{\tau} + \frac{\wtilde Dc_1 \sigma\alpha ( \sigma + 4\sigma \alpha^2)^{\wtilde D+1}\int\limits_{\Omega_2}  d\theta }
{\tau \sigma^{\wtilde D}\int\limits_{\mathbb{S}_{\wtilde D-1}} d\theta}
\\
&\leq  \frac{\xi\wtilde Dc_1 \sigma\alpha}{\tau} + \frac{\wtilde Dc_1 \sigma^2\alpha ( 1 + 4\alpha^2)^{\wtilde D+1}\int\limits_{\Omega_2}  d\theta }
{\tau \int\limits_{\mathbb{S}_{\wtilde D-1}} d\theta}
\end{align*}
From \eqref{eq:Omega3_to_ball_ratio} we have that 
\begin{align*}
\norm{(II)} &\leq  \frac{\xi\wtilde Dc_1 \sigma\alpha}{\tau} + \frac{\wtilde Dc_1 \sigma^2\alpha ( 1 + 4\alpha^2)^{\wtilde D+1}}
{\tau} \frac{2}{\xi\sqrt{\wtilde D-2}}e^{-(\wtilde D - 2 )\xi^2/2}
\end{align*}
Since $M = \frac{\tau}{\sigma}>C_\stau\sqrt{D\log D}$, and choosing again $\xi = 2\sqrt{\log D/D}$ we have

\begin{align*}
\norm{(II)} &\leq  \frac{2\wtilde Dc_1 \alpha}{C_\stau D} + \frac{\wtilde Dc_1 \sigma\alpha ( 1 + 4\alpha^2)^{\wtilde D+1}}
{C_\stau \sqrt{D\log D}} \frac{\sqrt{D}}{\sqrt{\log D} \sqrt{\wtilde D-2}}e^{-\log D}
\\
&\leq \frac{2c_1 \alpha}{C_\stau } + \frac{\alpha c_1 \sigma ( 1 + 4\alpha^2)^{\wtilde D+1}}
{C_\stau D\log D} \frac{1}{\sqrt{\wtilde D-2}}
\end{align*}
Since From Theorem \ref{thm:Step1} we have that $\alpha \leq \sqrt{ 1/D}$, and thus 
$( 1 + \alpha)^{\wtilde D}$ is bounded for any $D$. Thus, we have that for $C_\stau$ large enough, 
\begin{align}
\norm{(II)}&\leq \frac{\eps}{2}\alpha
\end{align}
for any $D$.

Thus, we have from \eqref{eq:zDftilde} that for any unit vector $z$, $\|\vec{z}( \DD_{\wtilde{f}}[x] - \DD_f[x])\| \leq \eps\alpha$
or 
\[
\| \DD_{\wtilde{f}}[x] - \DD_{f}[x]\|_{op} \leq \eps\alpha
\]
\end{proof}

\begin{Lemma}\label{lem:Diff_Tf_bound}
	Let $ f, \tilde f $ be functions from $ \RR^d $ to $ \RR^{D-d} $, and denote their differentials by $ \DD_f, \DD_{\wtilde f} $ respectively.
	Denote by $ T_0 f, T_0 \wtilde f $ the tangent planes of the graphs of $ f $ and $ \wtilde f $ respectively.
	Assume that
	\[
	\norm{\DD_f[0] - \DD_{\tilde f}[0]}_{op}\leq \varepsilon 
	.\]
	Then, for sufficiently small $ \varepsilon $
	\[
	\sin(\maxangle(T_0 f, T_0 \wtilde{f})) \leq  \varepsilon 
	.\]
\end{Lemma}

\begin{proof}
	By definition we know that
	\[
	T_0 f = \{ \DD_{f}[0]v ~|~ v\in \RR^d \} \subset \RR^D
	,\]
	and
	\[
	T_0 \wtilde f_\ell = \{ (v, \DD_{\wtilde f}[0]v) ~|~ v\in \RR^d \} \subset \RR^D
	.\] 
	Let $ L_1, L_2 $ be two linear spaces, denote by $ Q_{L_1, L_2} :L_1\to L_2^\perp $  the following operator
	\[
	Q_{L_1, L_2} = {v - Proj_{L_2}(v)}
	.\]
	By Lemma \ref{lem:Q_angle_between_flats} we know that 
	\[
	\sin(\maxangle\lrbrackets{L_1, L_2}) = \norm{Q_{L_1, L_2}}_{op}
	.\]
	We now turn to look at the operator $ Q_{T_0 f, T_0 \wtilde f } $ operating on some vector $ v\in T_0 f $.
	\begin{align*}
	\norm{Q_{T_0 f, T_0 \wtilde f}(v,\DD_{f}[0]v )} 
	&\leq 
	\norm{(v,\DD_{f}[0]v ) - (v,\DD_{\wtilde f}[0]v )}\\
	&\leq 
	\norm{\DD_{f}[0] - \DD_{\wtilde f}[0]}_{op}\norm{(v,\DD_{f}[0]v)}
	,\end{align*}
	and thus,
	\[
	\sin (\angle_{max}(T_0 f, T_0 \wtilde f ))  = \|Q_{T_0 f, T_0 \wtilde f}\|_{op} \leq \|\DD_{f}[0] - \DD_{\wtilde f}[0]\|_{op} \leq  \varepsilon  
	\]
\end{proof}

\begin{Lemma}\label{lem:Q_angle_between_flats}
    Let $L_1$ and $L_2$ be two linear subspaces of $\RR^D$. Denote $Q_{L_1,L_2}:L_1 \rightarrow L_2^\perp$ defined as
    \[
    Q_{L_1,L_2}(v) = v - Proj_{L_2}(v)
    ,\]
    where $ Proj_{L_2} $ is the projection onto $ L_2 $.
    Then, 
    \[
    \sin \angle_{max}(L_1,L_2) = \|Q_{L_1,L_2}\|_{op}
    .\]
\end{Lemma}
\begin{proof}
	Recalling Definition \ref{def:principal_angles_clean}, the Principal Angles $ \beta_i $ between  $ L_1, L_2 $ and their corresponding pairs of vectors $ u_i\in L_1, w_i\in L_2 $ are defined as
	\begin{equation*}
	\begin{array}{ll}
	u_1, w_1 \defeq \argmin\limits_{\substack{u\in L_1, w\in L_2 \\ \norm{u}=\norm{w} =1}}\arccos\left(\abs{\langle u, w \rangle}\right), &
	\beta_1 \defeq \angle(u_1, w_1)
	\end{array}    
	,\end{equation*}
	and for $i > 1$
	\begin{equation*}
	\begin{array}{ll}
	u_i, w_i \defeq \argmin\limits_{\substack{u\perp\UU_{i-1}, w\perp\WW_{i-1} \\ \norm{u}=\norm{w} =1}}\arccos\left(\abs{\langle u, w \rangle}\right), &
	\beta_i \defeq \angle(u_i, w_i)
	\end{array}
	,\end{equation*}
	where $$\UU_i \defeq Span\{u_j\}_{j=1}^i ~,~ \WW_i \defeq Span\{w_j\}_{j=1}^i.$$
	We now wish to show that for all $ i $ we can choose
	\[
	w_i = \frac{Proj_{L_2}(u_i)}{\norm{Proj_{L_2}(u_i)}}
	.\]
	Since the definition is inductive so will be our proof.
	\paragraph*{Basis of the induction $ i=1 $:}
	We first denote
	\[
	v_1 = Proj_{L_2(u_1)}
	.\]
	Note that,
	\[
	\beta_1 = \angle(u_1, w_1) = \angle(u_1, \norm{v_1} w_1 )
	,\]
	and by the minimization problem defining $ \beta_1 $ we know that
	\[
	\beta_1 \leq \angle(u_1, v_1)
	.\]
	Then, since the projection onto a linear space minimizes the Least-Squares norm we get
	\begin{align*}
	\lrangle{u_1 - v_1, u_1 - v_1} 
	&\leq 
	\lrangle{u_1 - \norm{v_1}w_1, u_1 - \norm{v_1}w_1}\\
	\lrangle{u_1,u_1} - 2\lrangle{u_1, v_1} + \lrangle{v_1,v_1} 
	&\leq
	 \lrangle{u_1,u_1} - 2\lrangle{u_1, \norm{v_1}w_1} + \lrangle{\norm{v_1}w_1,\norm{v_1}w_1} \\
	 - 2\lrangle{u_1, v_1} + \norm{v_1}^2 
	 &\leq
	 - 2\lrangle{u_1, \norm{v_1}w_1} + \norm{v_1}^2 \\
	 \lrangle{u_1, \norm{v_1}w_1} 
	 &\leq
	 \lrangle{u_1, v_1}\\
	 \lrangle{u_1, w_1} 
	 &\leq
	 \lrangle{u_1, \frac{v_1}{\norm{v_1}}}
	 \\
	 \beta_1 =
	 \angle(u_1, w_1) =
	 \arccos(\lrangle{u_1, w_1}) 
	 &\geq
	 \arccos(\lrangle{u_1, \frac{v_1}{\norm{v_1}}}) = \angle(u_1, v_1)
	.\end{align*}
	Thus, 
	\[
	\angle(u_1, v_1) = \beta_1
	,\]
	and we can choose $ w_1 = v_1 $.
	\paragraph*{The induction step:}
	Now we assume that for all $ i\leq j $
	\[
	w_i = \frac{v_i}{\norm{v_i}}
	,\]
	where
	\[
	v_i = Proj_{L_2}(u_i)
	.\]
	And, we wish to show that 
	\[
	\angle(u_{j+1}, v_{j+1}) = \beta_{j+1}
	,\]
	where the fact that $ \beta_{j+1} \leq  \angle(u_{j+1}, v_{j+1})$ results directly from the definition of $ \beta_{j+1} $.
	
	We first note that since
	\[
	u_{j+1}\perp\UU_j
	,\]
	we have
	\[	Proj_{L_2}(u_{j+1})\notin Span\{Proj_{L_2}(u_i)\}_{i=1}^j = \WW_j
	,\]
	thus, 
	\[
	v_{j+1} \in \WW_j^\perp  
	.\]
	From here on we can repeat the same argument as in the basis of the induction, just replacing $ L_1, L_2 $ with $ \UU_j^\perp, \WW_j^\perp $ respectively.
\end{proof}

\begin{Lemma}\label{lem:g_bond_weak}
Let the conditions of Lemma \ref{lem:bound_Df-Dfwtilde} hold.
Let $g(x, \theta)$ be as defined in Equation~\eqref{eq:g_def}.
Then, for $\alpha$ smaller than some constant and $M$ larger then some constant, 
\[
\sigma \leq g(0,\theta) \leq \sigma +  4\sigma \alpha^2
\]
\end{Lemma}

\begin{proof}
    
    Since $(0,f(0))_H \in \MM$, any point $p\in \RR^D$ such that $\norm{p - (0,f(0))_H}\leq\sigma$ belongs to $\MM_\sigma$, the $\sigma$-tubular neighborhood of $\MM$. 
    In particular, this is also true for $p\in H^\perp\subset \RR^D$, and thus, we obtain the lower bound 
    \begin{equation*}
     \sigma \leq g(0,\theta)   
    ,\end{equation*}
    as by the definition of \eqref{eq:g_def} we have $g(0,\theta) = \max_{p\in (\textrm{Span}\{\theta\}\cap\MM_\sigma)} \norm{p - (0, f(0))_H}$.
    
    From Lemma \ref{lem:bounding_stuff} we have that 
    \[
    \|f(\wtilde{x}_\theta(0)) - f(0)\| \leq 2\sigma \sin (\alpha + c\frac{\sigma}{\tau}\alpha) \tan \alpha,
    \]
    where $c$ is some constant, and $\wtilde{x}_\theta$ defined in \eqref{eq:x_tilde_def}.
    
    Since $(0, f(0))_H + \theta g(0,\theta)$ is at distance  $\sigma$ from $(\wtilde{x}_\theta(0), f(\wtilde{x}_\theta(0)))_H$, and denoting $\theta\in H^\perp\subset \RR^D$ by $(0, \bar \theta)_H$ we have that
    \[
    \sigma  =
    \norm{(0, f(0) + \bar \theta g(0, \theta)  )_H -(\wtilde x_{\theta}(0), f(\wtilde x_\theta(0))_H}
    = \sqrt{\norm{\wtilde x_\theta}^2 + \norm{f(\wtilde x_\theta(0) - f(0)}^2 + g(0,\theta)^2 },\]
    and thus,
    \[
    g(0,\theta) \leq \sigma + 2\sigma \sin (\alpha + c\frac{\sigma}{\tau}\alpha) \tan \alpha
    .\]
    Then, for $\alpha$ smaller than some constant and $M$ larger than some constant we have 
    \[
    g(0,\theta) \leq \sigma + 4\sigma \sin^2 \alpha,
    \] 
    or,
    \[
    g(0,\theta) \leq \sigma + 4\sigma \alpha^2,
    \] 
\end{proof}
\begin{Lemma}\label{lem:bound_nablag}
Let the conditions of Lemma \ref{lem:bound_Df-Dfwtilde} hold.
Let $g(x, \theta)$ be as defined in Equation \eqref{eq:g_def}. Then, 
\[
\|\nabla_x{g}(0,\theta)\| \leq C \cdot \frac{\sigma}{\tau}\alpha 
,\]
where $\nabla_x{g}$ denotes the gradient of $g(x, \theta)$ with respect to the $x$ variables only, and $C$ is some constant.
\end{Lemma}
\begin{proof}
Following the definition of $\wtilde{x}_\theta$ in \eqref{eq:x_tilde_def} and $g(x,\theta)$ in \eqref{eq:g_def}, we have the following equations that describe the connection between $x, \wtilde{x}_\theta$ and $g(x,\theta)$
\begin{equation*}
\left(\begin{array}{c}
     \wtilde{x}_\theta \\
     f(\wtilde{x}_\theta) 
\end{array}\right)
+
\sigma \vec N(x,\wtilde{x}_\theta,\theta) = 
\left(\begin{array}{c}
     x \\
     f(x) 
\end{array}\right)_H
+ 
\left(\begin{array}{c}
     0 \\
     \bar \theta
\end{array}\right)_H g(x,\theta)
,\end{equation*}
where $\theta\in \RR^D$ is written as $(0, \bar \theta)_H$, and $ \vec N(x,\wtilde{x}_\theta,\theta)\in \RR^D$ is some unit vector perpendicular to $T_{\wtilde{x}_\theta}f$.
Explicitly, 
\begin{equation}\label{eq:c_2}
 \vec N(x,\wtilde{x}_\theta,\theta) \perp T_{\wtilde{x}_\theta}f
,\end{equation}
and 
\begin{equation}\label{eq:c_3}
\|\vec N(x,\wtilde{x}_\theta,\theta) \| = 1
.\end{equation}
Alternatively, we can write,
\begin{equation}\label{eq:c_1}
\sigma \vec N(x,\wtilde{x}_\theta,\theta) = 
\left(\begin{array}{c}
     x \\
     f(x) 
\end{array}\right)_H
-
\left(\begin{array}{c}
     \wtilde{x}_\theta \\
     f(\wtilde{x}_\theta) 
\end{array}\right)_H
+ 
\left(\begin{array}{c}
     0 \\
     \bar \theta
\end{array}\right)_H g(x,\theta)
\end{equation}

Taking the norm of \eqref{eq:c_1} and using \eqref{eq:c_3} we have  

\begin{equation*}
\left\|\left(\begin{array}{c}
     x \\
     f(x) 
\end{array}\right)
-
\left(\begin{array}{c}
     \wtilde{x}_\theta \\
     f(\wtilde{x}_\theta) 
\end{array}\right)
+ 
\left(\begin{array}{c}
     0 \\
     \bar  \theta
\end{array}\right) g(x,\theta)\right\|^2 = \sigma^2
\end{equation*}
or,
\[
g^2 + 2g \left(\begin{array}{c}
     x - \wtilde{x}_\theta \\
     f(x) - f(\wtilde{x}_\theta) 
\end{array}\right) ^T
\left(\begin{array}{c}
     0 \\
     \bar \theta
\end{array}\right)
+
\left\|\left(\begin{array}{c}
     x - \wtilde{x}_\theta \\
     f(x) - f(\wtilde{x}_\theta) 
\end{array}\right)\right\|^2 
-
\sigma^2 =0
\]
or,
\[
g^2 + 2g (f(x) - f(\wtilde{x}_\theta))^T \bar \theta 
+
\|x - \wtilde{x}_\theta\|^2 
+
\| f(x) - f(\wtilde{x}_\theta) \|^2
-
\sigma^2 =0
\]
the two solutions are 
\begin{multline}\label{eq:sol_g_pm}
    g_{\pm}(x,\theta) =\\ -f(x)^T \bar \theta + f(\wtilde{x}_\theta)^T \bar \theta \pm
    \sqrt{\sigma^2 + (f(x)^T \bar \theta - f(\wtilde{x}_\theta)^T \bar \theta)^2 - \|x- \wtilde{x}_\theta\|^2 - \|f(x)- f(\wtilde{x}_\theta) \|^2}
\end{multline}

From Lemma \ref{lem:bounding_stuff}, for $\alpha$ smaller than some constant, we have that
\[ 
\|0- \wtilde{x}_\theta(0)\|^2 + \|f(0)- f(\wtilde{x}_\theta(0)) \|^2 \leq \sigma^2
\]
and thus, the solutions of Eq. \eqref{eq:sol_g_pm} at $x=0$ are $g_{-}(0, \theta) < 0$ and $g_{+}(0, \theta) > 0$.
Therefore, from continuity we get that there is a neighborhood of $x=0$ such that the only non-negative solution is
\begin{multline}\label{eq:sol_g}
    g(x,\theta) =\\ -f(x)^T \bar \theta + f(\wtilde{x}_\theta)^T \bar \theta + 
    \sqrt{\sigma^2 + (f(x)^T \bar \theta - f(\wtilde{x}_\theta)^T \bar \theta)^2 - \|x- \wtilde{x}_\theta\|^2 - \|f(x)- f(\wtilde{x}_\theta) \|^2}
.\end{multline}
In addition, from the definition of $g$ the only valid solution is the non-negative one which appears on Eq. \eqref{eq:sol_g}.
Thus, denoting 
\begin{equation}\label{eq:Delta}
\Delta = \sigma^2 + (f(x)^T \bar \theta - f(\wtilde{x}_\theta)^T \bar \theta)^2 - \|x- \wtilde{x}_\theta\|^2 - \|f(x)- f(\wtilde{x}_\theta) \|^2,
\end{equation}
we have that near $x=0$
\begin{multline}\label{eq:nabla_g}
    \nabla_x{g}(x,\theta) =\\ -\DD_f[x]^T\bar\theta + J_{\wtilde{x}_\theta} \DD_f[\wtilde{x}_\theta]^T \bar \theta 
    +
    \frac{1}{\sqrt{\Delta}}
    \left( 
     (f(\wtilde{x}_\theta)^T \bar \theta - f(x)^T \bar \theta)\left(J_{\wtilde{x}_\theta}\DD_f[\wtilde{x}_\theta]^T \bar \theta - \DD_f[x]^T \bar \theta\right) 
    \right. \\
    \left.- 
    (I_d-J_{\wtilde{x}_\theta})(x- \wtilde{x}_\theta) 
    - 
    (\DD_f[x]^T - J_{\wtilde{x}_\theta}\DD_f[\wtilde{x}_\theta]^T)(f(x)- f(\wtilde{x}_\theta))  \right)
,\end{multline}
where $J_{\wtilde x_\theta(x)} = \DD_{\wtilde x_\theta (x)}[x]$ is the Jacobi matrix of the function $\wtilde x_\theta (x)$, and $I_d$ is the $d$-dimensional identity matrix.
Alternatively, we can write
\begin{multline}\label{eq:nabla_g2}
    \nabla_x{g}(x,\theta) =\\ (J_{\wtilde{x}_\theta} \DD_f[\wtilde{x}_\theta]^T - \DD_f[x]^T ) \bar\theta 
    +
    \frac{1}{\sqrt{\Delta}}
    \left( 
     (f(\wtilde{x}_\theta)^T\bar\theta - f(x)^T\bar\theta ) \left(J_{\wtilde{x}_\theta}\DD_f[\wtilde{x}_\theta]^T - \DD_f[x]^T \right)\bar\theta 
    \right. \\
    \left.- 
    (I_d-J_{\wtilde{x}_\theta})(x- \wtilde{x}_\theta) 
    - 
    (\DD_f[x]^T - J_{\wtilde{x}_\theta}\DD_f[\wtilde{x}_\theta]^T)(f(x)- f(\wtilde{x}_\theta))  \right)
\end{multline}

Next, using Lemma \ref{lem:bound_J}, Lemma \ref{lem:Dfx-Dfxtilde} and Lemma \ref{lem:bounding_stuff}, we bound 
\begin{equation}\label{eq:Df-JDfwtilde}
\begin{aligned}
     \|\DD_f[0]^T - J_{\wtilde{x}_\theta} \DD_f[\wtilde{x}_\theta(0)]^T \| &= \|\DD_f[0]^T -\DD_f[\wtilde{x}_\theta(0)]^T+\DD_f[\wtilde{x}_\theta(0)]^T - J_{\wtilde{x}_\theta(0)} \DD_f[\wtilde{x}_\theta(0)]^T \| 
     \\
     &\leq \|\DD_f[0]^T -\DD_f[\wtilde{x}_\theta(0)]^T\|+\|\DD_f[\wtilde{x}_\theta(0)]^T - J_{\wtilde{x}_\theta(0)} \DD_f[\wtilde{x}_\theta(0)]^T \|
     \\
     &\leq \|\DD_f[0]^T -\DD_f[\wtilde{x}_\theta(0)]^T\|+\|I_d - J_{\wtilde{x}_\theta(0)} \|\|\DD_f[\wtilde{x}_\theta(0)]^T \| 
     \\
     &= \OO(\frac{\sigma}{\tau}\sin \alpha) + \OO(\frac{\sigma}{\tau}) \|\DD_f[\wtilde{x}_\theta(0)]^T \|
     \\
     &= \OO(\frac{\sigma}{\tau}\cdot\alpha) 
\end{aligned}
\end{equation}

Now, using \eqref{eq:nabla_g2} and the fact that $\norm{\bar \theta} = 1$ we get
\begin{align*}
    \|\nabla_x{g}(0,\theta)\|  &\leq\|\DD_f[0]^T - J_{\wtilde{x}_\theta}[0] \DD_f[\wtilde{x}_\theta(0)]^T\|  \\\
    &~~~
    +
    \frac{1}{\sqrt{\Delta}}
    \left( 
     \|f(\wtilde{x}_\theta(0)) - f(0)\| \|J_{\wtilde{x}_\theta}[0]\DD_f[\wtilde{x}_\theta(0)]^T - \DD_f[0]^T \| 
    \right. \\
    &~~~ \left.+ 
    \|I_d-J_{\wtilde{x}_\theta}[0]\|\|\wtilde{x}_\theta(0)\| 
    + 
    \|\DD_f[0]^T - J_{\wtilde{x}_\theta}[0]\DD_f[\wtilde{x}_\theta(0)]^T\|\|f(0)- f(\wtilde{x}_\theta(0))\|\right)
\end{align*}
From Lemmas \ref{lem:bound_J} and \ref{lem:bounding_stuff} we know that $\|I_d - J_{\wtilde x_\theta}[0]\| =\OO(\frac{\sigma}{\tau})$, $\|\wtilde x_\theta(0)\|\leq \sigma \sin(\alpha + c\alpha{\sigma/\tau} )$, and that $\|f(\wtilde x_\theta(0)) - f(0)\| \leq 2\sigma \sin(\alpha + c\alpha\sigma/\tau)\tan(\alpha)$.
In other words, for $\alpha$ smaller than some constant we can say that $\|I_d - J_{\wtilde x_\theta}[0]\| = \OO(\sigma/\tau)$, $\|\wtilde x_\theta(0)\|\leq \OO(\sigma \alpha)$, and that $\|f(\wtilde x_\theta(0)) - f(0)\| \leq \OO(\sigma \alpha^2)$.
Combining this with \eqref{eq:Df-JDfwtilde} as well, we have
\begin{align*}
    \|\nabla_x{g}(0,\theta)\| &\leq\OO(\frac{\sigma}{\tau}\cdot\alpha)
    +
    \frac{1}{\sqrt{\Delta}}
    \left( 
    \OO(\frac{\sigma^2}{\tau}\cdot\alpha^3)
    + 
    \OO(\frac{\sigma^2}{\tau}\cdot\alpha) 
    +
    \OO(\frac{\sigma^2}{\tau}\cdot\alpha^3)\right)
\end{align*}
Since Lemma \ref{lem:Dxtildew+Dxw} gives us 
\[
 \Delta(0,\wtilde x_\theta (0)) \geq \frac{1}{2} \sigma^2
,\]
we have that $\frac{1}{\sqrt{\Delta}}\leq \sqrt{2}/\sigma$, and 
\begin{equation*}
    \|\nabla_x{g}(x,\theta)\| =  \OO(\frac{\sigma}{\tau}\cdot\alpha)
.\end{equation*}
\end{proof}

\begin{Lemma}\label{lem:bound_J}
Let the conditions of Lemma \ref{lem:bound_nablag} hold.
Let $g(x, \theta), \wtilde x(x)$ be as defined in Equation \eqref{eq:g_def} and \eqref{eq:x_tilde_def} respectively, let $I_d:\RR^d\to\RR^d$ denote the identity matrix and let $J_{\wtilde x_\theta}$ denote the differential of $\wtilde x(x)$ with respect to $x$. Then, 
\begin{equation*}
\|I_d-J_{\wtilde{x}_\theta}[0]\| = \OO(\frac{\sigma}{\tau})
.\end{equation*}    
\end{Lemma}
\begin{proof}
We begin by reiterating equations \eqref{eq:c_1},\eqref{eq:c_2}, and \eqref{eq:c_3}.
Namely we have
\begin{equation*}
\sigma \vec N(x,\wtilde{x}_\theta,\theta) = 
\left(\begin{array}{c}
     x \\
     f(x) 
\end{array}\right)
-
\left(\begin{array}{c}
     \wtilde{x}_\theta \\
     f(\wtilde{x}_\theta) 
\end{array}\right)
+ 
\left(\begin{array}{c}
     0 \\
     \bar \theta
\end{array}\right) g(x,\theta)
,\end{equation*}
where
\begin{equation*}
 \vec N(x,\wtilde{x}_\theta,\theta) \perp T_{\wtilde{x}_\theta}f
,\end{equation*}
and 
\begin{equation*}
\|\vec N(x,\wtilde{x}_\theta,\theta) \| = 1
.\end{equation*}
Thus, there is a vector $v(x,\wtilde x_\theta)\in \RR^{D-d}$ with $\|v(x,\wtilde x_\theta)\| = 1$,
\[
\vec N(x,\wtilde{x}_\theta,\theta) = \frac{1}{\sqrt{\|\DD_f[\wtilde{x}_\theta] v(x,\wtilde x_\theta)\|^2 + 1}} 
\left(\begin{array}{c}
     -\DD_f[\wtilde{x}_\theta]^T v(x,\wtilde x_\theta) \\
     v(x,\wtilde x_\theta)
\end{array}\right)
\]
or, denoting $ w(x,\wtilde x_\theta) = \frac{\sigma}{\sqrt{\|\DD_f[\wtilde{x}_\theta]^T v(x,\wtilde x_\theta))\|^2 + 1}}  v$, we have 
\[
\sigma \vec N(x,\wtilde{x}_\theta) = 
\left(\begin{array}{c}
     -\DD_f[\wtilde{x}_\theta]^T  w(x,\wtilde x_\theta) \\
     w(x,\wtilde x_\theta)
\end{array}\right)
\]

Using this pronunciation of $\vec N$ we can rewrite the above equation as 
\begin{equation}\label{eq:c_1_w}
\left(\begin{array}{c}
     -\DD_f[\wtilde{x}_\theta]^T w(x,\wtilde x_\theta) \\
     w(x,\wtilde x_\theta)
\end{array}\right) = 
\left(\begin{array}{c}
     x \\
     f(x) 
\end{array}\right)
-
\left(\begin{array}{c}
     \wtilde{x}_\theta \\
     f(\wtilde{x}_\theta) 
\end{array}\right)
+ 
\left(\begin{array}{c}
     0 \\
     \bar \theta
\end{array}\right) g(x,\theta)
.\end{equation}
From Eq. \eqref{eq:sol_g} in the proof of Lemma \ref{lem:bound_nablag} we know that 
$g(x, \theta) = -f(x)^T\bar\theta + f(\wtilde x_\theta)^T\bar \theta + \sqrt{\Delta}$, near $x=0$, where $\Delta$ is defined in Eq. \eqref{eq:Delta}.
Combining this with the last $D-d$ equations  we get,
\begin{equation}\label{eq:def_w}
    w(x,\wtilde x_\theta) =\left(f(x) - f(\wtilde{x}_\theta) + \bar \theta \left(- f(x)^T \bar \theta + f(\wtilde{x}_\theta) ^T \bar \theta + 
    \sqrt{\Delta}\right)\right)
.\end{equation}
Looking at the first $d$ equations of \eqref{eq:c_1_w}, we have 
\[
-\DD_f[\wtilde{x}_\theta]^T w(x,\wtilde x_\theta) - x + \wtilde{x}_\theta = 0
.\]
Denoting the function 
\begin{equation} \label{eq:G_def}
G(x,\wtilde{x}_\theta) = -\DD_f[\wtilde{x}_\theta]^T w(x,\wtilde x_\theta) - x + \wtilde{x}_\theta    
,\end{equation}
we aim at using the Inverse Function Theorem (IFT) to compute $J_{\wtilde{x}_\theta}$.
First, we compute  $\DD^{x}_G $ and $\DD^{\wtilde{x}_\theta}_G$, the partial differentials of $G$ with respect to the variables $x$ and $\wtilde x_\theta$:
\begin{equation*}
    \DD^x_G[x, \wtilde x_\theta] = - \DD_f[\wtilde{x}_\theta]^T \DD^x_w[x, \wtilde x_\theta] - I_d
\end{equation*}
\begin{equation*}
    \DD^{\wtilde{x}_\theta}_G[x, \wtilde x_\theta] = - \HH f(\wtilde x_\theta)^w - \DD_f[\wtilde{x}_\theta]^T \DD^{\wtilde{x}_\theta}_w[x, \wtilde x_\theta]  + I_d
,\end{equation*}
where $\HH f(\wtilde x_\theta)^w \in \RR^{d\times d}$ is the tensor Hessian of $f(\wtilde x_\theta):\RR^d\to\RR^{D-d}$ projected onto the target direction $w\in \RR^{D-d}$; that is
\begin{equation}\label{eq:H_w_def}
\HH f(\wtilde x_\theta)^w =
\left(\frac{\partial\left(\DD_f[\wtilde{x}_\theta]^T\right)}{\partial \wtilde{x}_1} w \right|\left. \frac{\partial\left(\DD_f[\wtilde{x}_\theta]^T\right)}{\partial \wtilde{x}_2} w \right| \cdots \left|\frac{\partial\left(\DD_f[\wtilde{x}_\theta]^T\right)}{\partial \wtilde{x}_d} w \right)
.\end{equation}
Notice that $\DD_f[\wtilde x_\theta]^T\in \RR^{d\times D-d}$; therefore, $\partial_{\wtilde x_j}\DD_f[\wtilde x_\theta]^T\in \RR^{d\times D-d}$ and $\partial_{\wtilde x_j}\DD_f[\wtilde x_\theta]^T w\in \RR^{d}$.

Next, using the IFT we have that 
\[
J_{\wtilde{x}_\theta} = -(\DD^{\wtilde{x}_\theta}_{G})^{-1} \DD^x_G = \left( I_d - \DD_f[\wtilde{x}_\theta]^T \DD^{\wtilde{x}_\theta}_ w + \HH f(\wtilde x_\theta)^w  \right)^{-1} \left(I_d + \DD_f[\wtilde{x}_\theta]^T \DD^x_w \right)
,\]
and thus
\[
J_{\wtilde{x}_\theta}[0] = (A(I_d-A^{-1}B))^{-1}A = (I_d-A^{-1}B)^{-1}A^{-1}A =  (I_d-A^{-1}B)^{-1}
,\]
where
\begin{equation}\label{eq:A_def}
A = I_d +  \DD_f[\wtilde{x}_\theta(0)]^T \DD^x_w[0,\wtilde{x}_\theta(0)]    
,\end{equation}
\begin{equation}\label{eq:B_def}
B = \DD_f[\wtilde{x}_\theta(0)]^T (\DD^x_w[0,\wtilde{x}_\theta(0)] + \DD^{\wtilde{x}_\theta}_w[0,\wtilde{x}_\theta(0)] )+ \HH f(\wtilde x_\theta(0))^w
.\end{equation}
From Lemma \ref{lem:bound_A_inv_B} we have that $\|A^{-1}B\| \leq \OO(\frac{\sigma}{\tau}) \leq 1/2$ for $\sigma/\tau$ smaller than some constant, and thus, using the first order approximation of this term we get that there is a matrix
\[\mathcal{E} = \sum_{t=2}^{\infty}(A^{-1}B)^t\]
such that
\[
J_{\wtilde{x}_\theta} = (I_d - A^{-1}B)^{-1} = I_d +A^{-1}B + \mathcal{E}
,\]
and
with
\[
\|\mathcal{E}\| \leq \|A^{-1}B\| \sum_{t=1}^\infty \frac{1}{2^t} = \|A^{-1}B\|
.\]

Thus we have, 
\begin{equation}\label{eq:I-J}
\|I_d -J_{\wtilde{x}_\theta}\| \leq  2\|A^{-1}B\|  =  \OO(\frac{\sigma}{\tau})
\end{equation}

\end{proof}

\begin{Lemma}\label{lem:bound_A_inv_B}
Let the conditions of Lemma \ref{lem:bound_J} hold and let $A$ and $B$ be as defined in \eqref{eq:A_def} and \eqref{eq:B_def}.
Then,
    \[
    \|A^{-1}B\| \leq \OO(\frac{\sigma}{\tau})
    \]
\end{Lemma}
\begin{proof}
We begin by noting that 
\[
\|A^{-1}B\| \leq \|A^{-1}\|\|B\|
,\]
where
\begin{equation*}
A = I_d +  \DD_f[\wtilde{x}_\theta(0)]^T \DD^x_w[0,\wtilde{x}_\theta(0)]    
,\end{equation*}
\begin{equation*}
B = \DD_f[\wtilde{x}_\theta(0)]^T (\DD^x_w[0,\wtilde{x}_\theta(0)] + \DD^{\wtilde{x}_\theta}_w[0,\wtilde{x}_\theta(0)] )+ \HH f(\wtilde x_\theta(0))^w
.\end{equation*}
Moreover,
\[
A^{-1} = (I_d +  \DD_f[\wtilde{x}_\theta(0)]^T \DD^x_w[0, \wtilde x_\theta(0)])^{-1} = I_d + \sum_{t=1}^\infty(\DD_f[\wtilde{x}_\theta(0)]^T \DD^x_w[0, \wtilde x_\theta(0)])^t
.\]
From Lemma \ref{lem:Dxtildew+Dxw} we have that $\|\DD^x_w[0, \wtilde x_\theta(0)]\| = \OO(\sin \alpha)$, where we remind the reader that $\DD^x_w[x, \wtilde x_\theta]$ is the partial differential of $w(x, \wtilde x_\theta)$ with respect to the $x$ variables only.
And, from Lemma \ref{lem:bounding_stuff}we have that $\|\DD_f[\wtilde{x}_\theta(0)]\|_2 \leq \sin (\alpha + c\frac{\sigma}{\tau}\alpha)$ . Thus $\|\DD_f[\wtilde{x}_\theta]^T \DD^x_w[0, \wtilde x_\theta(0)]\| = \OO(\sin\alpha)$, and thus, for $\alpha$ smaller than some constant we have
\begin{equation}\label{eq:bound_A}
    \|A^{-1}\| = 1 + \OO(\sin^2\alpha)
\end{equation}

Furthermore, from Lemma \ref{lem:Dxtildew+Dxw} we also know that $(\DD^x_w[0, \wtilde x_\theta(0)] + \DD^{\wtilde{x}_\theta}_ w[0, \wtilde x_\theta(0)] ) = \OO(\frac{\sigma}{\tau} \sin\alpha)$ and so 
\[
\| \DD_f[\wtilde{x}_\theta(0)]^T (\DD^x_w[0, \wtilde x_\theta(0)] + \DD^{\wtilde{x}_\theta}_ w[0, \wtilde x_\theta(0)] )\| \leq \OO(\frac{\sigma}{\tau}\sin^2 \alpha)
.\]
Combining this bound with the fact that $\|\HH f(\wtilde x_\theta(0))^w\| \leq \OO(\frac{\sigma}{\tau}) $ shown in Lemma \ref{lem:bound_DDf} we have 
\begin{equation}\label{eq:bound_B}
\|B\| \leq  \|\DD_f[\wtilde{x}_\theta]^T (\DD^x_w + \DD^{\wtilde{x}_\theta}_ w )\|+\| \HH f(\wtilde x_\theta)^w\| \leq \OO(\frac{\sigma}{\tau}) +  \OO(\frac{\sigma}{\tau}\sin^2 \alpha)
.
\end{equation}

Finally, from  \eqref{eq:bound_A} and \eqref{eq:bound_B} we have that for $\alpha$ smaller than some constant
\[
\|A^{-1}B\| \leq \OO(\frac{\sigma}{\tau})
\]
\end{proof}

\begin{Lemma}\label{lem:bound_DDf}
Let the conditions of Lemma \ref{lem:bound_J} and let $\HH f(\wtilde x_\theta(0))^w$ be as defined in \eqref{eq:H_w_def}.
Then, we have
\[
\|\HH f(\wtilde x_\theta(0))^w\|_{op} = \frac{\sqrt{2}\sigma}{\tau} + \OO\left(\frac{\sigma\sin \alpha}{\tau}\right)
\]
\end{Lemma}
\begin{proof}
We denote the tensor Hessian of $f:\RR^d\to\RR^{D-d}$ at $\wtilde x_\theta(0)$ by $\HH f(\wtilde x_\theta(0)):\RR^d\times\RR^d \to \RR^{D-d}$. 
For brevity of notation, throughout this proof we will use $\HH$ instead of $\HH f(\wtilde x_\theta(0))$. For any chosen direction $u\in \RR^{D-d}$ (i.e., a unit vector), $\HH$ can be thought of as a function:
$\HH ^u:\RR^d\times\RR^d \to \RR$
defined as $\HH ^u(v_1,v_2) = \lrangle{\HH (v_1,v_2), u}$. 
We note that this definition is consistent with the definition of $\HH f(\wtilde x_\theta(0))^u$ in \eqref{eq:H_w_def}.
Given $v\in\RR^d$, $\norm{v} = 1$ we also define $\HH_v:\RR^d \to \RR^{D-d}$
as $\HH_v(\cdot) = \HH(v,\cdot)$.
Note, that for any $w\in \RR^{D-d}$ 
\[
{\|\HH^w\|_{op}} = {\sup\limits_{v_1, v_2\in \mathbb{S}^{d}}|\lrangle{\HH(v_1, v_2), w}|} 
\leq \sup\limits_{v_1, v_2\in \mathbb{S}^{d}}\|\HH(v_1, v_2)\|_2\norm{w}_2
= \norm{w}_2\sup_{v\in\mathbb{S}^d}\|\HH_v\|_{op}
,\]
where the right-most equality is true since $\HH$ is symmetric.

Thus, in essence, we need to bound $\|\HH_v\|_{op}$ for an arbitrary $v$. 
From the definition of $\HH_v$ we know that 
\[
\HH_v\ =  \lim_{t\to 0} \frac{\DD_f[\wtilde x_\theta(0)] - \DD_f[\wtilde x_\theta(0) + tv]}{t} 
.\]
Then, from Lemma \ref{lem:alpha-beta_x_no_func}  we have for small $t$
\[
\sin(\maxangle (T_{\wtilde x_\theta}f,T_{\wtilde x_\theta + tv}f))\leq \frac{t}{\tau}(1+\tan^2\beta)+\OO(t^2/\tau^2)
\]
where $\beta = \maxangle(T_{\wtilde x_\theta}f,H)$. 
Therefore, applying Lemma \ref{lem:norm_from_angle} we get 
\[
 \|\DD_f[\wtilde x_\theta(0)] - \DD_f[\wtilde x_\theta(0) + tv] \|_{op} \leq \frac{t}{\tau}(1+\tan^2\beta)(1+\sin \beta) + \OO(t^2/\tau^2)
,\]
and we get 
\[
\|\HH_v\| \leq  \frac{1}{\tau}(1+\tan^2\beta)(1+\sin \beta)
.\]
Furthermore, from Lemma \ref{lem:alpha_beta_D} we know 
\[
\beta \leq \alpha + c_1\alpha{\frac{\sigma}{\tau}}
,\]
and so,
\[
\|\HH_v\| \leq  \frac{1}{\tau}(1+\tan^2
\left(\alpha + c_1\alpha{\frac{\sigma}{\tau}}\right))(1+\sin \left(\alpha + c_1\alpha{\frac{\sigma}{\tau}}\right))
= \frac{1}{\tau} + \OO(\frac{\sin \alpha}{\tau})
.\]
Thus, we obtain
\begin{equation}\label{eq:norm_H^w_bound}
\|\HH^w\| \leq \|w\| \left(\frac{1}{\tau} + \OO(\frac{\sin \alpha}{\tau})\right)    
.\end{equation}

Hence, all we are left with is bounding  $\|w(0,\wtilde x_\theta (0))\|$. 
From Eq. \eqref{eq:def_w}, Lemma \ref{lem:bounding_stuff}, and Lemma \ref{lem:Dxtildew+Dxw}, we have 
\begin{align*}
    \|w(0,\wtilde x_\theta(0))\| &=\|f(0) - f(\wtilde{x}_\theta(0)) + \theta \left(- f(0)\cdot \theta + f(\wtilde{x}_\theta(0)) \cdot \theta + 
    \sqrt{\Delta(0,\wtilde x_\theta (0))}\right)\|
    \\
    &\leq \|f(0) - f(\wtilde{x}_\theta(0))\| + \left|(- f(0) + f(\wtilde{x}_\theta(0))) \cdot \theta\right| + 
    \sqrt{\Delta(0,\wtilde x_\theta (0))}
    \\
    &\leq 2\sigma \sin \alpha + 2\sigma \sin \alpha + \sqrt{2}\sigma 
    \\
    &\leq \sqrt{2}\sigma  + 4\sigma \sin \alpha 
\end{align*}

Since $\|w\| \leq \sqrt{2}\sigma + \OO(\sin \alpha)$ 
we have from Eq. \eqref{eq:norm_H^w_bound}
\[
\|\HH^w\| \leq \frac{\sqrt{2}\sigma}{\tau} + \OO(\frac{\sigma\sin \alpha}{\tau})
\]
\end{proof}

\begin{Lemma}\label{lem:Dxtildew+Dxw}
Let the conditions of Lemma \ref{lem:bound_J} hold, let $w(x, \wtilde x_\theta)$ be as defined in \eqref{eq:def_w}, $\wtilde x_\theta(x)$ be as defined in \eqref{eq:x_tilde_def}, and $\Delta$ as defined in \eqref{eq:Delta}. 
Denote by $\DD^x_w, \DD^{\wtilde{x}_\theta}_w$ the partial differentials of $w$ with respect to the variables $x$ and $\wtilde x_\theta$.
Then, for $\alpha$ smaller than some constant
\begin{equation}\label{eq:Delta_bound}
    \frac{1}{2} \sigma^2 \leq \Delta(0,\wtilde x_\theta (0)) \leq 2 \sigma^2
,\end{equation}
\begin{equation} \label{eq:Dws_bound}
\| \DD^{\wtilde{x}_\theta}_ w[0, \wtilde x_\theta(0)] \| \leq  \mathcal{O}(\sin \alpha) \qquad \|\DD^x_w[0, \wtilde x_\theta(0)]\| \leq \mathcal{O}(\sin \alpha)
,\end{equation}
and
\begin{equation}\label{eq:Dxtildew+Dxw_bound}
\| \DD^{\wtilde{x}_\theta}_w[0, \wtilde x_\theta(0)] + \DD^x_w[0, \wtilde x_\theta(0)]\| \leq \mathcal{O}(\frac{\sigma}{\tau}\sin \alpha)
.\end{equation}
\end{Lemma}
\begin{proof}
First we bound $\Delta$ from Eq. \eqref{eq:Delta} at $x=0,\wtilde x_\theta = x_\theta(0)$ using Lemma \ref{lem:bounding_stuff}, and assuming $\alpha$ is smaller than some constant.
\begin{align}
\Delta(0,\wtilde x_\theta (0)) &= \sigma^2 + (f(0)^T \theta - f(\wtilde{x}_\theta(0))^T \theta)^2 - \|0- \wtilde{x}_\theta(0)\|^2 - \|f(0)- f(\wtilde{x}_\theta(0)) \|^2 
\notag\\
&\geq \sigma^2 -\|\wtilde{x}_\theta(0)\|^2 - 2\|f(0)- f(\wtilde{x}_\theta(0)) \|^2
\notag\\
&\geq \sigma^2 -2\sigma^2 \sin^2 \alpha - 4\sigma^2 \sin^2 \alpha 
\notag\\
&\geq \sigma^2(1-6\sin^2 \alpha )
\notag\\
&\geq \frac{1}{2} \sigma^2
.\end{align}
Similarly, 
\begin{align}
\Delta(0,\wtilde x_\theta (0)) &= \sigma^2 + (f(0)^T \theta - f(\wtilde{x}_\theta(0))^T \theta)^2 - \|0- \wtilde{x}_\theta(0)\|^2 - \|f(0)- f(\wtilde{x}_\theta(0)) \|^2 
\notag\\
&\leq \sigma^2 +\|\wtilde{x}_\theta(0)\|^2 + 2\|f(0)- f(\wtilde{x}_\theta(0)) \|^2
\notag\\
&\leq \sigma^2 +2\sigma^2 \sin^2 \alpha + 4\sigma^2 \sin^2 \alpha 
\notag\\
&\leq \sigma^2(1+6\sin^2 \alpha )
\notag\\
&\geq 2 \sigma^2
.\end{align}
and thus we showed Eq. \eqref{eq:Delta_bound}.

Next we compute $\DD^x_w$ and $D^{\wtilde{x}}_ w$
\begin{align}\label{eq:D_xw}
    \begin{split}
     \DD^x_w =& \DD_f[x] 
    +
    \theta {(\theta^T \DD_f[x])} \\
    &~~+
    {\frac{1}{\sqrt{\Delta}}} \theta \left( 
    2 {(f(x)\cdot \theta - f(\wtilde{x}_\theta) \cdot \theta)}{\theta^T \DD_f[x]} 
    - 
    2(x- \wtilde{x}_\theta)^T 
    - 
    2(f(x)- f(\wtilde{x}_\theta))^T\DD_f[x]  \right)\\
    =&\left(I_{D-d} 
    +
    \theta \theta^T  
    +
    \frac{1}{\sqrt{\Delta}} \theta \left( 
    2 (f(x)^T- f(\wtilde{x}_\theta)^T ) (\theta \theta^T - I_{D-d} )
      \right)\right) \DD_f[x] -
    \frac{2}{\sqrt{\Delta}} \theta (x- \wtilde{x}_\theta)^T 
    \end{split}
,\end{align}

\begin{align}\label{eq:D_xwtilde}
    \begin{split}
    \DD^{\wtilde{x}_\theta}_ w = & -\DD_f[\wtilde{x}_\theta] 
    - 
    \theta {(\theta^T \DD_f[\wtilde{x}_\theta])} \\
    &~~+
    {\frac{1}{\sqrt{\Delta}}} \theta \left( 
    -2 {(f(x)\cdot \theta - f(\wtilde{x}_\theta) \cdot \theta)}{\theta^T \DD_f[\wtilde{x}_\theta]} 
    + 
    2(x- \wtilde{x}_\theta)^T 
    +
    2(f(x)- f(\wtilde{x}_\theta))^T\DD_f[\wtilde{x}_\theta]  \right)\\
    =&-\left(I_{D-d} 
    +
    \theta \theta^T  
    +
    \frac{1}{\sqrt{\Delta}} \theta \left( 
    2 (f(x)^T- f(\wtilde{x}_\theta)^T ) (\theta \theta^T - I_{D-d} )
      \right)\right) \DD_f[x] +
    \frac{2}{\sqrt{\Delta}} \theta (x- \wtilde{x}_\theta)^T 
    \end{split}
.\end{align}
From Eq. \eqref{eq:D_xw}, Lemma \ref{lem:bounding_stuff}, and Eq. \eqref{eq:Delta_bound},  we have that 
\begin{align*}
     \|\DD^x_w[0,\wtilde x_\theta(0)]\| \leq& 
    \|I_{D-d} 
    +
    \theta \theta^T  
    +
    \frac{1}{\sqrt{\Delta}} \theta \left( 
    2 (f(0)^T- f(\wtilde{x}_\theta(0))^T ) (\theta \theta^T - I_{D-d} )
      \right)\|_{op} \|\DD_f[0]\|_{op} \\
      &~~+
    \frac{2}{\sqrt{\Delta}} \|(0- \wtilde{x}_\theta(0))^T\|_2 
    \\
    \leq& \OO( \sin \alpha)
\end{align*}
Similarly, from Eq. \eqref{eq:D_xwtilde}, Lemma \ref{lem:bounding_stuff}, and Eq. \eqref{eq:Delta_bound},  we have that 
\begin{equation}
     \|\DD^x_w[0,\wtilde x_\theta(0)]\| 
    \leq \OO( \sin \alpha)
\end{equation}
Thus, we showed \eqref{eq:Dws_bound}.

Now we show \eqref{eq:Dxtildew+Dxw_bound}. From \eqref{eq:D_xw} and \eqref{eq:D_xwtilde} we have 
\begin{equation}\label{eq:D_xw-Dxwtilde}
\begin{aligned}
    \sigma \DD^x_w + \sigma \DD^{\wtilde{x}_\theta}_ w 
    &= \DD_f[x] -\DD_f[\wtilde{x}_\theta] +\theta \theta^T (\DD_f[x] - \DD_f[\wtilde{x}_\theta]) \\
    &~~+\frac{1}{\sqrt{\Delta}} \theta 
    \bigg( 2 (f(x)\cdot \theta - f(\wtilde{x}_\theta) \cdot \theta)\theta^T - 2(f(x)- f(\wtilde{x}_\theta))^T\bigg) (\DD_f[x] - \DD_f[\wtilde{x}_\theta])\\
    &= \left(I_{D-d} +\theta \theta^T +\frac{1}{\sqrt{\Delta}} \theta 
    \bigg( 2 (f(x)^T - f(\wtilde{x}_\theta)^T)\theta\theta^T - 2(f(x)- f(\wtilde{x}_\theta))^T\bigg)\right) (\DD_f[x] - \DD_f[\wtilde{x}_\theta])\\
    &=\left(I_{D-d} +\theta \theta^T +\frac{2}{\sqrt{\Delta}} \theta 
      (f(x)^T - f(\wtilde{x}_\theta)^T) (\theta\theta^T - I_{D-d})\right) (\DD_f[x] - \DD_f[\wtilde{x}_\theta])
\end{aligned}
\end{equation}
Since we are bounding for $x= 0 $, we have from Lemma \ref{lem:bounding_stuff}, that $\|f(0) - f(\wtilde{x}_\theta(0))\| \leq \OO(\sigma \sin^2 \alpha)$, for $\alpha$ smaller than some constant.
Taking the norm of \eqref{eq:D_xw-Dxwtilde} we have that 
\[
\|\sigma \DD^x_w + \sigma \DD^{\wtilde{x}_\theta}_ w \| \leq \left\|I_{D-d} +\theta \theta^T +\frac{2}{\sqrt{\Delta}} \theta 
      (f(x)^T - f(\wtilde{x}_\theta)^T) (\theta\theta^T - I_{D-d})\right\| \|\DD_f[x] - \DD_f[\wtilde{x}_\theta]\|
\]
From Eq. \eqref{eq:Delta_bound} and Lemma \ref{lem:Dfx-Dfxtilde} we have,
\[
\|\sigma \DD^x_w + \sigma \DD^{\wtilde{x}_\theta}_ w \| =  \OO(\frac{\sigma}{\tau}\sin \alpha)
\]
\end{proof}

\begin{Lemma}\label{lem:Dfx-Dfxtilde}
Let $f$ be a differentiable function from $H$, a $d$-dimensional subspace of $\RR^D$, to $\RR^{D-d}$.
Assume, $\maxangle(T_0f, H)\leq \alpha$ and that $\mathrm{rch}(\Gamma_f)$ the reach of $\Gamma_f$ (the graph of the function $f$), is bounded by $\tau$.
\[
\|\DD_f[\wtilde{x}_\theta(0)] - \DD_f[0]\|_{op} \leq 
\frac{\sigma \sin (\alpha + c\frac{\sigma}{\tau}\alpha)}{\tau}(1 + \tan^2 \alpha)(1 + \sin\alpha)  +\OO(\sigma^2 \sin^2 (\alpha)/\tau^2)
\]
or
\[
\|\DD_f[\wtilde{x}_\theta(0)] - \DD_f[0]\|_{op} \leq \OO(\frac{\sigma}{\tau}\sin \alpha)
,\]
for $\alpha$ smaller than some constant and some constant $c
\in\RR$.
\end{Lemma}
\begin{proof}
From Lemma \ref{lem:alpha-beta_x_no_func} we have that 
\[
	\sin(\maxangle (T_{0}f, T_{\wtilde{x}_\theta}f)) \leq \frac{\norm{\wtilde{x}_\theta}}{\tau}(1 + \tan^2 \alpha)  +\OO(\norm{\wtilde{x}_\theta}^2/\tau^2)
\]
From Lemma \ref{lem:bounding_stuff}, we have $\|\wtilde{x}_\theta\| \leq \sigma \sin (\alpha + c\alpha{\frac{\sigma}{\tau}})$, for some general constant $c$, and thus
\[
	\sin(\maxangle (T_{0}f, T_{\wtilde{x}_\theta}f)) \leq \frac{\sigma \sin (\alpha + c\alpha{\frac{\sigma}{\tau}})}{\tau}(1 + \tan^2 \alpha)  +\OO(\sigma^2 \sin^2 (\alpha)/\tau^2)
\]
Moreover, we have that 
\[
\sin(\maxangle (T_{0}f, H)) \leq \sin \alpha
\]
Using Lemma \ref{lem:norm_from_angle} we have 
\[
\|\DD_f[\wtilde{x}_\theta] - \DD_f[x]\| 
\leq 
\frac{\sigma \sin (\alpha + c\alpha{\frac{\sigma}{\tau}})}{\tau}(1 + \tan^2 \alpha)(1 + \sin\alpha)  +\OO(\sigma^2 \sin^2 (\alpha)/\tau^2) 
\]

\end{proof}
\begin{Lemma}\label{lem:norm_from_angle}
Let $L_1, L_2$ be two linear operators from $H$ a $d$-dimensional subspace of $\RR^D$ to $\RR^{D-d}$.
Let, $\maxangle(H,(H,L_1(H))_H) \leq \alpha$, where $(H,L_1(H)_H)$ is the subspace spanned by $H$ and $L_1(H)$, the target space of $L_1$. Furthermore, let $\maxangle((H,L_1(H))_H,(H,L_2(H))_H) \leq \beta$.
Then, 
\[\|L_1 - L_2\|_{op} \leq \sin \beta(1 + \sin\alpha).\]
\end{Lemma}
\begin{proof}
For any $x\in H, \|x\| = 1$, from Lemma \ref{lem:Q_angle_between_flats}, there is $y\in H$ such that 
\[
\|(x,L_1(x)) - (y,L_2(y))\| \leq \sin \beta
.\]
Therefore, 
\[
\|x-y\|^2+\|L_1(x) - L_2(y)\|^2  = \|(x,L_1(x)) - (y,L_2(y))\|^2\leq \sin^2 \beta
,\]
and
\[
\|x-y\| \leq \sin \beta ~,~ \|L_1(x) - L_2(y)\| \leq \sin\beta
.\]
Note, that $\|L_1(y) - L_1(x)\| \leq \|L_1\|_{op}\|x-y\|$.
Since $\maxangle(H,(H,L_1(H))) \leq \alpha$ we have that $\|L_1\|_{op}\leq \sin \alpha$, and we get that 
\[
\|L_1(y) - L_1(x)\| \leq \sin \alpha \sin \beta
.\]
Furthermore,
\[
\|L_1(x) - L_2(x)\| = \|L_1(x) - L_1(y) + L_1(y) - L_2(x)\| \leq \|L_1(x) - L_1(y)\|+\|L_1(y) - L_2(x)\|
,\]
and so
\[
\|L_1(x) - L_2(x)\| \leq \sin \beta(1 + \sin\alpha)
.\]
\end{proof}

\begin{Lemma}\label{lem:bounding_stuff}
Let the conditions of Lemma \ref{lem:bound_Df-Dfwtilde} hold. Let $\wtilde x_\theta(x)$ be as defined in equation \eqref{eq:x_tilde_def} in the proof of Lemma \ref{lem:bound_Df-Dfwtilde}.
Then, for $\alpha$ smaller than some constant and $\frac{\tau}{\sigma}$ larger than some constant, we have
\begin{align} \label{eq:xtilde_bound_l}
& \|\wtilde{x}_\theta(0)\| &\leq \sigma \sin (\alpha + c_1{\frac{\sigma}{\tau}}\alpha)
\\
\label{eq:f_xtilde_bound_l}
&\|f(\wtilde{x}_\theta(0)) - f(0)\| &\leq 2\sigma \sin (\alpha + c_1{\frac{\sigma}{\tau}}\alpha) \tan \alpha 
\\ \label{eq:df_bound_l}
&\|\DD_f[0]\|_2 &\leq \sin \alpha    
\\ \label{eq:df_xtilde_bound_l}
&\|\DD_f[\wtilde{x}_\theta(0)]\|_2 &\leq \sin (\alpha + c_1{\frac{\sigma}{\tau}}\alpha)
\end{align}

\end{Lemma}
\begin{proof}
In essence, this lemma is a summary and rewriting of results from other lemmas which are meant to be used conveniently in the proof of Lemma \ref{lem:bound_Df-Dfwtilde}. 
Accordingly, \eqref{eq:xtilde_bound_l} is already achieved in Lemma \ref{lem:wtilde_x_bound_clean_D}.
Then, from Lemma \ref{lem:f_bound_circle_H_no_func_ver2_clean} we have
\[
    \|f(\wtilde{x}_\theta(0))\| \leq \norm{\wtilde{x}_\theta(0)}\tan \alpha  +\OO(\|\wtilde{x}_\theta(0)\|^2/\tau)
.\]
Thus, for $\alpha$ and $\frac{\sigma}{\tau}$ smaller than some constants, we achieve Eq. \eqref{eq:f_xtilde_bound_l}.
Next, since $\maxangle(H, T_0f) \leq \alpha$, by Lemma \ref{lem:Q_angle_between_flats} we have \eqref{eq:df_bound_l}.
Finally, denoting $ \beta(\wtilde x_\theta(0)) =  \maxangle(T_{\wtilde x_\theta(0)}f, H)$ by Lemma \ref{lem:Q_angle_between_flats} we have $\|\DD_f[\wtilde{x}_\theta(0)]\|_2 \leq \sin \beta(\wtilde{x}_\theta(0))$, and combining this with Lemma \ref{lem:alpha_beta_D} we obtain \eqref{eq:df_xtilde_bound_l}

\end{proof}

\begin{Lemma}\label{lem:wtilde_x_bound_clean_D}
3Let the conditions of Lemma \ref{lem:bound_Df-Dfwtilde} hold. Let $\wtilde x_\theta(x)$ be as defined in equation \eqref{eq:x_tilde_def} in the proof of Lemma \ref{lem:bound_Df-Dfwtilde}.
Then, for any unit vector $\theta\in \RR^{D-d}$, if $\alpha, \frac{\sigma}{\tau}$ are smaller than some constants, we have
	\begin{equation*}
	\|\wtilde{x}_\theta(0)\| \leq \sigma \sin (\alpha + c\frac{\sigma}{\tau}\alpha)
	\end{equation*}
 for some general constant $c\in\RR$.
\end{Lemma}
\begin{proof}
	From Lemma \ref{lem:x_tilde_bound_beta_D}, we get
	\begin{equation}\label{eq:xtilde_alpha_bound}
	\norm{\wtilde{x}_\theta(0)} \leq \sigma \sin \beta (\wtilde x_\theta(0)) 
	\end{equation}
	where $ \beta (\wtilde x_\theta(0)) = \maxangle(T_{\wtilde x_{\theta}(0)}f, H) $.
	Then, using Lemma \ref{lem:alpha_beta_D} we get
	\[
	\alpha - c_2\alpha \frac{\sigma}{\tau}  \leq \beta(\wtilde x_\theta(0)) \leq \alpha + c_1\alpha\frac{\sigma}{\tau}
	.\]
	Thus, we obtain
	\[
	\norm{\wtilde x_\theta(0)} \leq \sigma \sin (\alpha + c\frac{\sigma}{\tau}\alpha)
	,\]
 for $c=\max(c_1,c_2)$
	as required.
\end{proof}

\begin{Lemma}\label{lem:x_tilde_bound_beta_D}
Let the conditions of Lemma \ref{lem:bound_Df-Dfwtilde} hold. Let $\wtilde x_\theta(x)$ be as defined in equation \eqref{eq:x_tilde_def} in the proof of Lemma \ref{lem:bound_Df-Dfwtilde}.
Let $ T_{\wtilde x_\theta(0)}f $ be the tangent to the graph of $ f $ at the point $ (\wtilde x_\theta(0), f(\wtilde x_\theta(0))) $, $ \beta(\wtilde x_\theta(0)) =  \maxangle(T_{\wtilde x_\theta(0)}f, H)$.
	Then,
	\[
	\norm{\wtilde x_\theta(0) } \leq \sigma \sin \beta(\wtilde x_\theta(0))
	\]
\end{Lemma}
\begin{proof}
From \eqref{eq:x_tilde_def} of the proof of Lemma \ref{lem:bound_Df-Dfwtilde} (or more conveniently \eqref{eq:c_1} from the proof of Lemma \ref{lem:bound_nablag}) we have that 
\[
\|x-\wtilde{x}_\theta\| = \sigma \|Proj_H(\vec N_\theta(x,\wtilde{x}_\theta,\theta))\|,
\]
Using Lemma \ref{lem:angle_space_perp_space}, since $\vec N_\theta \in T_{\wtilde x_\theta(0)}f^\perp$ we have that 
\[
\|Proj_H(\vec N_\theta(x,\wtilde{x}_\theta,\theta))\| \leq \cos (\frac{\pi}{2}  - \beta(\wtilde x_\theta(0))) = \sin(\beta(\wtilde x_\theta(0))),
\]
and thus
\[
\|\wtilde{x}_\theta(0)\| = \sigma \sin(\beta(\wtilde x_\theta(0)))
\]
\end{proof}

\begin{Lemma}\label{lem:alpha_beta_D}
Let the conditions of Lemma \ref{lem:bound_Df-Dfwtilde} hold. Let $\wtilde x_\theta(x)$ be as defined in equation \eqref{eq:x_tilde_def} in the proof of Lemma \ref{lem:bound_Df-Dfwtilde}.
Let $ x_0 \in H $ be such that $ \norm{x_0} \leq \norm{\wtilde x_\theta(0)} $.
	Denote $ \beta(x) = \maxangle(T_{x}f, H) $ and let $ \alpha = \beta(0) $.
	Then,
	\[
	\alpha\left(1  - c_2\frac{\sigma}{\tau}\right)  \leq \beta(x_0) \leq \alpha\left(1  + c_1\frac{\sigma}{\tau}\right)
	,\]
 for some general constants $c_1, c_2$.
\end{Lemma}
\begin{proof}
	For convenience of notations we denote in this proof 
	\[\beta \defeq \beta(x_0).\]
	Using the result of Lemma \ref{lem:alpha_beta_x} we achieve
     \[
      \alpha- 2\frac{\norm{x_0}}{\tau}(1 + \alpha)  +c\|x_0\|^2/\tau^2 \leq \beta(x_0) \leq \alpha+ 2\frac{\norm{x_0}}{\tau}(1 + \alpha)  +c\|x_0\|^2/\tau^2
	 ,\]
 for some constant $c\in \RR$.
	and from the fact that $ \norm{x_0} \leq \norm{\wtilde x_\theta(0)} $ we get
 \begin{equation}\label{eq:init_quadratic_ineq}
      \alpha- 2\frac{\norm{x_\theta(0)}}{\tau}(1 + \alpha)  +c\|x_\theta(0)\|^2/\tau^2 \leq \beta(x_\theta(0)) \leq \alpha+ 2\frac{\norm{x_\theta(0)}}{\tau}(1 + \alpha)  +c\|x_\theta(0)\|^2/\tau^2
	 ,
 \end{equation}

	From Lemma \ref{lem:x_tilde_bound_beta_D} we know that
	\[
	\norm{\wtilde x_\theta(0)} \leq \sigma \sin\beta
	,\]
 where for brevity we write $\beta= \beta(x_{\theta}(0))$.
	So, we get
\begin{align*}
	\beta - \alpha
	&\leq
	2\frac{\sigma}{\tau}sin\beta\lrbrackets{1+\alpha   }  + c\frac{\sigma^2}{\tau^2}sin^2\beta\\
	\beta - \alpha
	&\leq 
	2\frac{\sigma}{\tau}\beta(1+\alpha)  + c\frac{\sigma^2}{\tau^2}\beta^2
        \\
	0
	&\leq c\frac{\sigma^2}{\tau^2}\beta^2+
	\beta(2\frac{\sigma}{\tau}+2\frac{\sigma}{\tau}\alpha-1)  + \alpha
	,\end{align*}
 
	The right hand side of this expression is a parabola with respect to $\beta$.
	Note, that for $\frac{\sigma}{\tau} = 0$ the roots are $\beta = \alpha$.
	Solving this parabola we get the roots
	\begin{align*}
	\beta_{+,-} 
	&=
	\frac{1-2\frac{\sigma}{\tau}(1+\alpha) \pm \sqrt{\lrbrackets{1-2\frac{\sigma}{\tau})1+\alpha)}^2 - 4c \frac{\sigma^2}{\tau^2}\alpha }}{2c\frac{\sigma^2}{\tau^2}} \\
	\end{align*} 
 Note that the root with the $+$ sign is of order $\frac{\tau^2}{\sigma^2}$ and therefore $\beta$ must be smaller than $\beta_-$.
 	\begin{align*}
	\beta_{-} 
	&=
	\frac{1-2\frac{\sigma}{\tau}(1+\alpha) - \sqrt{\lrbrackets{1-2\frac{\sigma}{\tau}(1+\alpha)}^2 - 4c \frac{\sigma^2}{\tau^2}\alpha }}{2c\frac{\sigma^2}{\tau^2}} 
 \\
 &=\frac{(1-2\frac{\sigma}{\tau}(1+\alpha))^2 - \lrbrackets{1-2\frac{\sigma}{\tau}(1+\alpha)}^2 + 4c \frac{\sigma^2}{\tau^2}\alpha}
 {2c \frac{\sigma^2}{\tau^2}  \left(1-2\frac{\sigma}{\tau}(1+\alpha) + \sqrt{\lrbrackets{1-2\frac{\sigma}{\tau}(1+\alpha)}^2 - 4c \frac{\sigma^2}{\tau^2}\alpha }\right)}
  \\
  &=
\frac{4c \frac{\sigma^2}{\tau^2}\alpha}
  {2c\frac{\sigma^2}{\tau^2} \left(1-2\frac{\sigma}{\tau}(1+\alpha) +\sqrt{1-4\frac{\sigma}{\tau}(1+\alpha) + 4\frac{\sigma^2}{\tau^2}(1+\alpha)^2- 4c\frac{\sigma^2}{\tau^2}\alpha }\right)}
\\
  &=
\frac{2 \alpha}
  {1-2\frac{\sigma}{\tau}(1+\alpha) +\sqrt{1-4\frac{\sigma}{\tau}(1+\alpha) + 4\frac{\sigma^2}{\tau^2}\left((1+\alpha)^2- c\alpha\right) }}
	\end{align*} 
From Remark \ref{rem:taylor_sqrt1-x2_clean}  we have

\begin{align*}
	\beta_{-} 
	&\leq
(2 \alpha)\left(
  1-2\frac{\sigma}{\tau}(1+\alpha) +1-4\frac{\sigma}{\tau}(1+\alpha) + 4\frac{\sigma^2}{\tau^2}\left((1+\alpha)^2- c\alpha\right) \right)^{-1}
  \\
  &=
 \alpha\left(
  1-3\frac{\sigma}{\tau}(1+\alpha) + 2\frac{\sigma^2}{\tau^2}\left((1+\alpha)^2- c\alpha\right) \right)^{-1}
	\end{align*} 
Since for small enough $x$ we have that $1/(1-x) \leq 1+2x$, we have for large enough $M$,
\begin{align}\label{eq:beta_boundc1_left}
	\beta_{-} 
	&\leq
 \alpha\left(
  1+6\frac{\sigma}{\tau}(1+\alpha) -4\frac{\sigma^2}{\tau^2}\left((1+\alpha)^2- c\alpha\right) \right) \nonumber
  \\
  &\leq
 \alpha\left(
  1+10\frac{\sigma}{\tau}(1+\alpha)  \right) \nonumber\\
  &\leq
 \alpha\left(
  1+c_1\frac{\sigma}{\tau} \right)
	\end{align} 
On the other hand, taking the left hand side of \eqref{eq:init_quadratic_ineq} we obtain similarly
\[
 \alpha- 2\frac{\norm{x_\theta(0)}}{\tau}(1 + \alpha)  +c\|x_\theta(0)\|^2/\tau^2 \leq \beta(x_\theta(0))
\]
\[
 \alpha- 2\frac{\norm{x_\theta(0)}}{\tau}(1 + \alpha)  \leq \beta
\]
\[
 \alpha- 2\frac{\sigma}{\tau}\beta(1 + \alpha)  \leq \beta
\]
\[
 \alpha \leq \beta \left(1+ 2\frac{\sigma}{\tau}(1 + \alpha) \right)
\]
\[
 \frac{\alpha}{1+ 2\frac{\sigma}{\tau}(1 + \alpha) } \leq \beta 
\]
and since $1/(1+x) = 1 - x + O(x^2)$ and $\sigma/\tau$ is sufficiently small we get
\begin{equation}\label{eq:beta_boundc2_right}
 \alpha \left(1- c_2\frac{\sigma}{\tau}\right) \leq \beta 
\end{equation}
Combining \eqref{eq:beta_boundc2_right} and \eqref{eq:beta_boundc1_left} we conclude the proof.
\end{proof}

\subsubsection{Bounding the finite sample error}\label{sec:FiniteSample_err}
Back to Theorem \ref{thm:Step2} proof road-map see Figure \ref{tikz:thm33_lemmas}.

\noindent In this section we show that  $ \maxangle(T_0 \wtilde f_\ell, H_{\ell + 1}) $ the angle between the tangent of $\wtilde{f}_\ell(0)$ and the tangent estimated using $n$ samples decays to zero as $n\to \infty$.
Namely, the main result of this subsection appears in the lemma below.

\begin{Lemma}\label{lem:AngleImprovement}
	Let $(q_\ell, H_\ell)$ be defined in Algorithm \ref{alg:step2_clean} and $ \pi^*_{q_\ell, H_\ell}(x) $ be defined in \eqref{eq:argmin2_clean} and let $ H_{\ell + 1} = T_0 \pi^*_{q_\ell, H_\ell} $ the tangent to the graph of $ \pi^*_{q_\ell, H_\ell} $ at $ \pi^*_{q_\ell, H_\ell}(0) $.
	Then, for all $ \delta>0 $ there is $ N_\delta $ such that for all $ n > N_\delta $ we have with probability $ 1-\delta $
	\[
	\maxangle(T_0\wtilde f_\ell, H_{\ell + 1}) \leq 2\sqrt{d}\frac{C_0 \ln(1/\delta)}{n ^{r_1}}
	,\]
	where ${r_1} = \frac{k-1}{2k + d}$ and $C_0$ is a constant.
\end{Lemma}
\begin{proof}
We first note that it is sufficient to bound the error of estimating the image of $\DD_{\wtilde{f}_\ell}[0]$, the differential of $ \wtilde f_\ell(x) $ at 0, by the image of $ \DD_{\pi^*_{q_\ell, H_\ell}}[0] $, the differential of the local polynomial least-squares regression $\pi^*_{q_\ell, H_\ell}$.
Explicitly, if 
\[
\|{\DD_{\pi^*_{q_\ell, H_\ell}}[0] - \DD_{ \wtilde f_\ell}[0]}\|_{op} \leq \sqrt{d}\frac{C_0 \ln(1/\delta)}{n^{r_1}}
,\]
then, by using Lemma \ref{lem:Diff_Tf_bound} we get that
\[
\sin(\maxangle(T_0\wtilde f_\ell, T_0\pi^*_{q_\ell, H_\ell})) \leq \sqrt{d}\frac{C_0 \ln(1/\delta)}{n^{r_1}}
,\]
which for sufficiently large $ n $ yields
\[
\maxangle(T_0\wtilde f_\ell, T_0\pi^*_{q_\ell, H_\ell}) \leq 2\sqrt{d}\frac{C_0 \ln(1/\delta)}{n^{r_1}}
,\]
as required.

Therefore, it is sufficient to show that for any $\delta$ there is $N_\delta$ such that for all $n>N_\delta$ we have 
\[
\|{\DD_{\pi^*_{q_\ell, H_\ell}}[0] - \DD_{ \wtilde f_\ell}[0]}\|_{op} \leq \sqrt{d}\frac{C_0 \ln(1/\delta)}{n^{r_1}}
,\]
with probability of at least $1 - \delta$.
Let us reiterate the minimization problem by which we derive the approximant.
Namely, given a sample $ \{r_i\}_{i=1}^{n} $ drawn i.i.d from $ \text{Unif}(\MM_\sigma) $, and a coordinate system $ (q, H)\in \RR^D\times Gr(d,D) $ we look for a polynomial $ \pi^*_{q,H} $ minimizing
\begin{equation}\label{eq:Step2_sample}
J_2(\pi ~|~ q, H) = \frac{1}{N_{q,H}}\sum_{r_i \in U_{\textrm{ROI}}^{n}} \norm{r_i - \pi(x_i)}^2
,\end{equation}
where $x_i$ are the projections of $r_i - q$ onto $H$, and $U_{\textrm{ROI}}^{n}(q,H)$ is defined through a bandwidth $\epsilon_{n}$ as
\begin{equation}\label{eq:ROI_ell_sample}
U_{\textrm{ROI}}^{n}(q,H) = {\{r_i\in U_\textrm{ROI}~|~\dist(x_i, q)<\epsilon_{n}\}}
,\end{equation}
and $N_{q,H}$ denotes the number of samples in $U_{\textrm{ROI}}^{n}(q, H)$.
Explicitly,
\begin{equation}\label{eq:argmin2_sample}
\pi^*_{q_\ell, H_\ell} = \argmin_{\pi\in \Pi_{k-1}^{d\mapsto D}} J_2(\pi ~|~ q_\ell, H_\ell)
.\end{equation}
We demand that the bandwidth $\epsilon_{n}\to 0$ as ${n}\to\infty$ such that 
\begin{equation}\label{eq:bandwidthDecay_sample}
0<\lim_{{n}\rightarrow\infty}N^{1/(2k+1)}\cdot \epsilon_{n} < \infty
.\end{equation}
And, the approximation is defined through $\DD_{\pi^*_{q_\ell, H_\ell}}[0]$

From Lemma \ref{lem:noise_cov_bound} we can apply Theorem 3.2 from \cite{aizenbud2021VectorEstimation} that gives convergence rates for local polynomial regression of vector valued functions in our case. Thus, we have that for every direction in the basis $\{x^j\}_{j=1}^d\subset H_\ell$ and every $\delta$ there exists  $N_{\delta}$ such that for all $n > N_{\delta}$ we have
\[
\Pr(\|{\partial_{x^j} \pi^*_{q_\ell, H_\ell}(0) - \partial_{x^j} \wtilde f_\ell(0)}\| > \frac{C_0 \ln(1/\delta)^{r_1}}{n^{r_1}}) < \delta
,\]
where $r_1 = \frac{k-1}{2k + d}$ and $C_0$ is a constant. Notice that $r_1 \leq 1/2$, and thus 
\[
\Pr(\|{\partial_{x^j} \pi^*_{q_\ell, H_\ell}(0) - \partial_{x^j} \wtilde f_\ell(0)}\| > \frac{C_0 \ln(1/\delta)}{n^{r_1}}) < \delta
.\]
Thus, taking into account all $d$ directions of the basis to $H_\ell$ we get that there are $C$ and $N_\delta$ such that for all $n > N_\delta$  
\[
\Pr(\|{\partial_{x^j} \pi^*_{q_\ell, H_\ell}(0) - \partial_{x^j} \wtilde f_\ell(0)}\| > \frac{C_0 \ln(1/\delta)}{n^{r_1}}\textrm{ for any } 1 \leq j \leq d) < d\delta
,\]
and thus
\[
\Pr(\|{\DD_{\pi^*_{q_\ell, H_\ell}}[0] - \DD_{\wtilde f_\ell}[0]}\|_{op} > \sqrt{d}\frac{C_0 \ln(1/\delta)}{n^r}) < d\delta
,\]
as required.
\end{proof}

In order to use convergence rate results of local polynomial regression for vector valued functions as described in Theorems 3.1 and 3.2 of \cite{aizenbud2021VectorEstimation} in our case, we need to show that the noise distribution $\eta_\ell$ defined in \eqref{eq:ftilde_def} is such that $\|\cov(\eta_\ell)\|\leq \sqrt{c/D}$.
\begin{Lemma}\label{lem:noise_cov_bound}
 Let $H_\ell\in Gr(d, D)$, and let $f_\ell:H_\ell\simeq\RR^d \rightarrow \RR^{D-d}$, defined as in \eqref{eq:fl_def}. Let  $\eta_\ell$ defined in \eqref{eq:ftilde_def}. Denote 
$\alpha_\ell = \maxangle (T_0f_{\ell}, H_{\ell})$ and assume $\alpha_\ell<1/D^{1/4}$. Then, 
\[
\|\cov(\eta_\ell)\|_{op} \leq \sqrt{\frac{c\sigma}{D-d}}
\]
    
\end{Lemma}
\begin{proof}
For ease of notation, denote $\wtilde{D} = D-d$.
Since we are interested in bounding 
\begin{equation}\label{eq:cov_norm}
\|\cov(\eta_\ell) \|_{op} = \max_{\vec{x}\in \mathbb{S}_{\wtilde{D} - 1} } \vec{x}^T \cov(\eta_\ell)\vec{x}
\end{equation}
we note that 
\[
\vec{x}^T\cov(\eta_\ell)\vec{x} = \mbox{Var}(\eta_\ell\cdot \vec x)
.\]
Thus, rewriting \eqref{eq:cov_norm} we have
\begin{equation}\label{eq:cov_norm2}
    \|\cov(\eta_\ell)\|_{op}= \max_{\vec{x}\in \mathbb{S}_{\wtilde{D} - 1}} \vec{x}^T\cov(\eta_\ell)\vec{x} = \max_{\vec{x}\in \mathbb{S}_{\wtilde{D} - 1}}\mbox{Var}(\eta_\ell\cdot \vec{x}) \leq \max_{\vec{x}\in \mathbb{S}_{\wtilde{D} - 1}}\EE((\eta_\ell\cdot \vec{x})^2)
\end{equation}
Thus, our goal is to bound, for any $\vec{z}\in \mathbb{S}_{\wtilde{D} - 1}$  the expression $\EE((\eta_\ell\cdot \vec{z})^2)$. From the definition of $g$ and $\Omega$ in \eqref{eq:g_def} and \eqref{eq:Omega_def}, we have that
\begin{align}\label{eq:E(Z_z)}
\EE((\eta_\ell\cdot \vec{z})^2) &= \frac{\int\limits_{y\in\Omega(x)} (y \cdot \vec{z})^2 dy}{\int\limits_{y\in\Omega(x)} dy} = 
\frac{\int\limits_{\mathbb{S}_{\wtilde D - 1}}\int\limits_0^{g(x,\theta)} (\theta \cdot \vec{z})^2 r^2 r^{\wtilde D-1} drd\theta}{\int\limits_{\mathbb{S}_{\wtilde D - 1}}\int\limits_0^{g(x,\theta)} r^{\wtilde D-1} drd\theta} \notag\\
& = \frac{\wtilde D\int\limits_{\mathbb{S}_{\wtilde D - 1}}(\theta \cdot \vec{z})^2 g(x,\theta)^{\wtilde D+2} d\theta}{(\wtilde D+2)\int\limits_{\mathbb{S}_{\wtilde D - 1}}g(x,\theta)^{\wtilde D} drd\theta}     
,\end{align}
where $dr$ is the measure over the radial component, $r^{\wtilde D-1}$ is the Jacobian introduced by the change of variables and $d\theta$ is the measure over the $(\wtilde D-1)$-dimensional sphere.

Following the rationale of the proof of Lemma \ref{lem:bound_Df-Dfwtilde}, we split $\mathbb{S}_{\wtilde D-1}$ into $\Omega_1$ and $\Omega_2$ of \eqref{eq:Omegas_def}.
That is,
\begin{equation*}
    \begin{aligned}
    \Omega_1 &= \{\theta~|~0\leq \vec{z}^T\theta \leq \xi\}\\
    \Omega_2 &= \{\theta~|~\vec{z}^T\theta > \xi\}
    \end{aligned}
,\end{equation*}
for some $\xi$ to be chosen later.
Thus, denoting $z = \theta^T \vec{z}$ we rewrite \eqref{eq:E(Z_z)} as` 
\begin{align*}
\EE((\eta_\ell\cdot \vec{z})^2)   &=\frac{\wtilde D\left(\int\limits_{\Omega_1}z^2 g(x,\theta)^{\wtilde D+2} d\theta + \int\limits_{\Omega_2}z^2 g(x,\theta)^{\wtilde D+2} d\theta\right)}{(\wtilde D+2)\int\limits_{\mathbb{S}_{\wtilde D-1}}g(x,\theta)^{\wtilde D} d\theta} 
\\
&\leq  \frac{\wtilde D\left(\xi^2 \int\limits_{\Omega_1} g(x,\theta)^{\wtilde D+2} d\theta + \int\limits_{\Omega_2} g(x,\theta)^{\wtilde D+2} d\theta\right)}{(\wtilde D+2)\int\limits_{\mathbb{S}_{\wtilde D-1}}g(x,\theta)^{\wtilde D} d\theta} 
\end{align*}

Since the conditions of Lemma \ref{lem:g_bond_weak} are met, we have $\sigma \leq g(0,\theta)\leq \sigma + 4\sigma \alpha_\ell ^2$, and thus
\begin{align*}
\EE((\eta_\ell\cdot \vec{z})^2)
&\leq  \frac{\wtilde D\left(\xi^2 \int\limits_{\Omega_1} g(x,\theta)^{\wtilde D+2} d\theta + \int\limits_{\Omega_2} g(x,\theta)^{\wtilde D+2} d\theta\right)}{(\wtilde D+2)\int\limits_{\mathbb{S}_{\wtilde D-1}}g(x,\theta)^{\wtilde D} d\theta}
\\
&\leq  \frac{\wtilde D\sigma^2(1+4\alpha_\ell ^2)^2\left(\xi^2  \int\limits_{\Omega_1} g(x,\theta)^{\wtilde D} d\theta + \int\limits_{\Omega_2} g(x,\theta)^{\wtilde D} d\theta\right)}{(\wtilde D+2)\int\limits_{\mathbb{S}_{\wtilde D-1}}g(x,\theta)^{\wtilde D} d\theta}
\\
&\leq  \frac{\wtilde D\sigma^2(1+4\alpha_\ell ^2)^2}{\wtilde D+2}\left(\xi^2  +\frac{ \int\limits_{\Omega_2} g(x,\theta)^{\wtilde D} d\theta}{\int\limits_{\mathbb{S}_{\wtilde D-1}}g(x,\theta)^{\wtilde D} d\theta} \right)
\\
&\leq  \frac{\wtilde D\sigma^2(1+4\alpha_\ell ^2)^2}{\wtilde D+2}\left(\xi^2  +\frac{\sigma^{\wtilde D}(1+4\alpha_\ell ^2)^{\wtilde D} \int\limits_{\Omega_2} d\theta}{\sigma^{\wtilde D}\int\limits_{\mathbb{S}_{\wtilde D-1}} d\theta} \right)
\\
&\leq  \frac{\wtilde D\sigma^2(1+4\alpha_\ell ^2)^2}{\wtilde D+2}\left(\xi^2 +\frac{(1+4\alpha_\ell ^2)^{\wtilde D}}{\xi\sqrt{\wtilde D-1}} e^{-(\wtilde D-1)\xi^2/2}\right)
,\end{align*}
where the last inequality comes from Eq. \eqref{eq:Omega3_to_ball_ratio}.
Since $\alpha_\ell <1/D^{1/4}$ we have that $(1+4\alpha_\ell ^2)^{\wtilde D}$ is bounded by some constant $c$. Choosing $\xi = 2\sqrt{\frac{\log(\wtilde D-1)}{\wtilde D-1}}$ we have 
\begin{align*}
\EE((\eta_\ell\cdot \vec{z})^2)
&\leq  \frac{\wtilde D\sigma^2(1+4\alpha_\ell ^2)^2 }{\wtilde D+2}\left(4\frac{\log(\wtilde D-1)}{\wtilde D-1} +\frac{c}{2\sqrt{\log(\wtilde D-1)}} (\wtilde D-1)^{-2}\right)
\\
&\leq  c_1\sigma^2\frac{\log(\wtilde D-1)}{\wtilde D-1}\leq  \sqrt{\frac{c\sigma}{\wtilde D}}
\end{align*}
for some constants $c,c_1$. Combining with Eq. 
\eqref{eq:cov_norm2}, we conclude the proof.

\end{proof}

\subsubsection{Bounding the distance of $q_\ell$ from $f_\ell(0)$}\label{sec:f0_est}

Back to Theorem \ref{thm:Step2} proof road-map see Figure \ref{tikz:thm33_lemmas}.

\begin{Lemma}\label{lem:dist_to_f_l0_srtong_small_alpha}
For $f_\ell$ defined in \eqref{eq:fl_def}. Denote 
$\alpha_\ell = \maxangle (T_0f_{\ell}, H_{\ell})$ and assume $\alpha_\ell<1/D$. Then,
for any $\delta$ there is $N$ such that for any number of samples $n > N$, we have 
\[
\| \pi^*_{r, H_\ell}(0) - f_\ell(0)\| \leq 8 \sigma D \alpha_\ell ^2 + \frac{c \ln\left(\frac{1}{\delta}\right)}{n^{r_0}}.
\]
with probability of at least $1-\delta$, where $r_0 = \frac{k}{2k+d}$ and $\pi^*_{r, H_\ell}(0)$ is the polynomial defined in Equation \eqref{eq:argmin2_clean}, and $c$ is some general constant.
\end{Lemma}
\begin{proof}
Using the triangle inequality, we have 
\begin{equation}\label{eq:f_ell0_triangle_l2}
    \|\pi^*_{r, H_\ell}(0) - f_\ell(0)\| \leq \| \pi^*_{r, H_\ell}(0)  - \wtilde{f}_\ell(0) \| + \| \wtilde{f}_\ell(0)  - f_\ell(0)\|.
\end{equation}
Applying Theorem 2.1 of \cite{aizenbud2021VectorEstimation} on $\| \pi^*_{r, H_\ell}(0)  - \wtilde{f}_\ell(0) \|$ we have that
\begin{equation}\label{eq:f_tilde0_bound_stone_l2}
        \| \pi^*_{r, H_\ell}(0) - \wtilde{f}_\ell(0)\| \leq \frac{c\ln\left(\frac{1}{\delta}\right)}{n^{r_0}}
\end{equation}
with probability of at least $1-\delta$.

Now we focus on bounding $\| \wtilde{f}_\ell(0)  - f_\ell(0)\|$. 
From \eqref{eq:f_tilde-f_integrals} we have that 
\begin{equation*}
\wtilde{f}_\ell(x) - f_\ell(x)=
\frac{\wtilde D\int\limits_{\mathbb{S}_{\wtilde D-1}}\theta g(x,\theta)^{\wtilde D+1} d\theta}{(\wtilde D+1)\int\limits_{\mathbb{S}_{\wtilde D-1}}g(x,\theta)^{\wtilde D} drd\theta}     
=
\frac{\wtilde D\int\limits_{\mathbb{S}_{\wtilde D-1}}\theta (g(x,\theta)^{\wtilde D+1} - \sigma^{\wtilde D+1}) d\theta}{(\wtilde D+1)\int\limits_{\mathbb{S}_{\wtilde D-1}}g(x,\theta)^{\wtilde D} drd\theta}     
\end{equation*}
or, looking at some direction $\vec{z}$ we have 
\begin{align*}
z^T\cdot (\wtilde{f}_\ell(x) - f_\ell(x))&=
\frac{\wtilde D\int\limits_{\mathbb{S}_{\wtilde D-1}}z (g(x,\theta)^{\wtilde D+1}- \sigma^{\wtilde D+1}) d\theta}{(\wtilde D+1)\int\limits_{\mathbb{S}_{\wtilde D-1}}g(x,\theta)^{\wtilde D} d\theta} 
\\
&\leq
\frac{\wtilde D}{(\wtilde D+1)}\sigma((1+4\alpha_\ell^2)^{\wtilde D+1} - 1)
\end{align*}

We note that for $x<1/n$ the following holds
\[
(1+x)^n-1<x((1+1/n)^n-1)/(1/n)<nx(e-1)<2nx.
\]
Using the above observation, and the fact that $4\alpha_\ell^2 <\wtilde D+1$ we have that 
\begin{align*}
z^T\cdot (\wtilde{f}_\ell(x) - f_\ell(x))&\leq
\frac{8\sigma\wtilde D}{(\wtilde D+1)}(\wtilde D+1)\alpha_\ell^2 \leq 8\sigma D \alpha_\ell^2
\end{align*}
Combining this with  \eqref{eq:f_ell0_triangle_l2} and \eqref{eq:f_tilde0_bound_stone_l2}, we have that for any $\delta>0$, for $\alpha_\ell >\frac{1}{D}$, and for number of samples $n>N$ large enough, 
\[
\|\pi^*_{r, H_\ell}(0) - f_\ell(0)\| \leq 8\sigma D \alpha_\ell^2 + \frac{C_2\ln\left(\frac{1}{\delta}\right)}{n^{r_0}},
\]
with probability of at least $1-\delta$.
\end{proof}

\subsubsection{Bounding the error induced by the shifted origin}\label{sec:ShiftedCenter_err}
\begin{figure}[ht]
	\centering
	\includegraphics[width=0.6\textwidth]{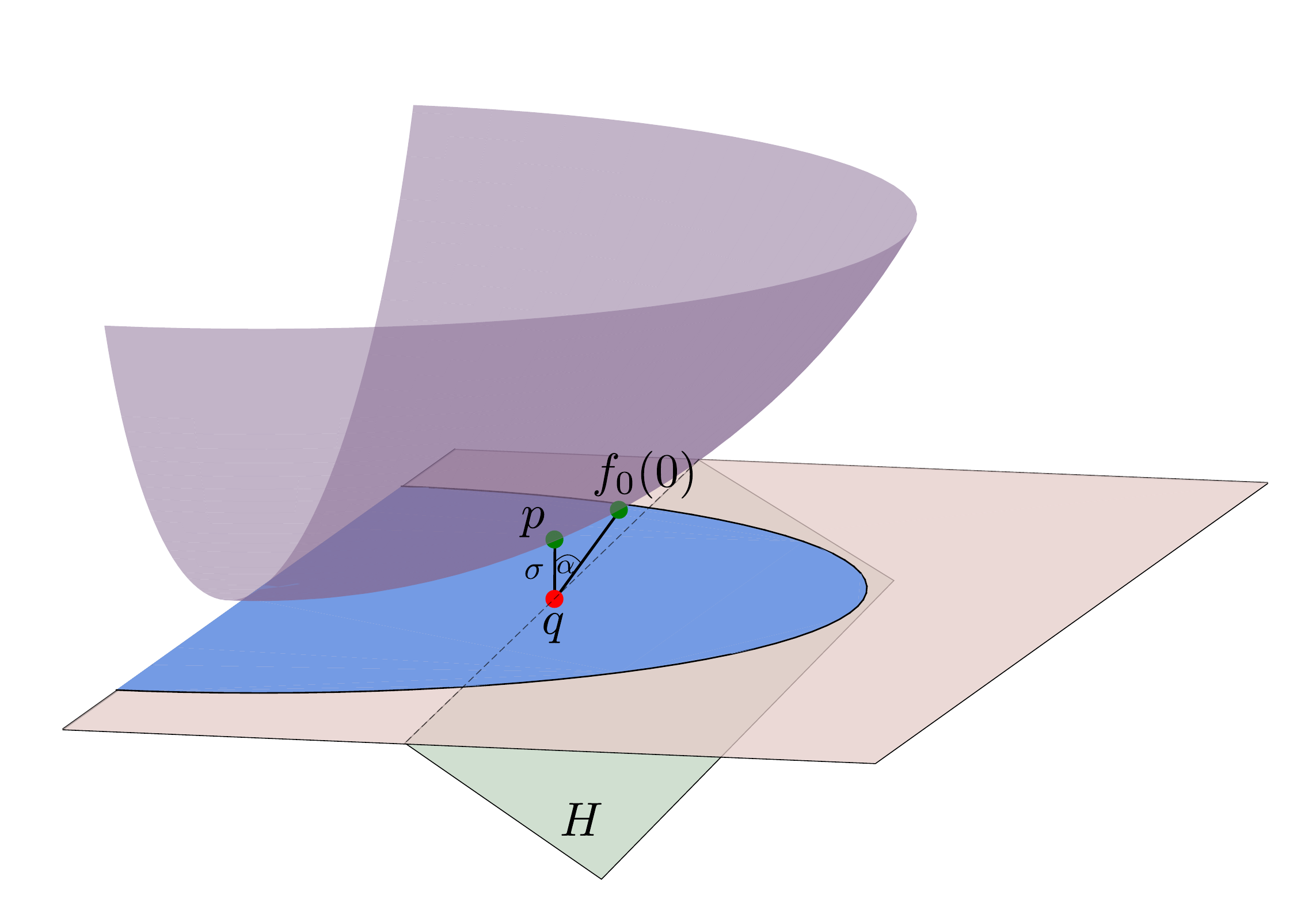}
	\caption{Illustration of the difference between $ p $ and $ f_0(0) $.}
	\label{fig:f0(0)vs_p}
\end{figure}

Back to Theorem \ref{thm:Step2} proof road-map see Figure \ref{tikz:thm33_lemmas}.

\begin{Lemma}\label{lem:shift_ang_general}
    Let $G_0$ be a $d$-dimensional linear space in $\RR^D$, and let $g_0:G_0 \rightarrow \RR^{D-d}$, such that the graph of $g_0$ is a manifold with reach bounded by $\tau$. Assume that $\maxangle (T_0g_0, G_0) \leq \alpha$. Let  $G_1$ be a $d$ dimensional linear space in $\RR^D$, such that $\maxangle(G_0,G_1)\leq \beta$. Define $g_1:G_1 \rightarrow \RR^{D-d}$ as the function who's graph coincides with the graph of $g_0$. 

    Then, for $\alpha \leq \pi/16$, $\beta\leq \beta_c$ where $\beta_c$ is some constant dependent only on $c_{\pi/4}$ of Lemma \ref{lem:M_is_locally_a_fuinction_clean}, and  $\|g_0(0)\| \leq c_{\pi/4} \cdot \tau$, for the constant $c_{\pi/4}$ defined in Lemma \ref{lem:M_is_locally_a_fuinction_clean}, we have 
    \[
    \maxangle (G_1, T_0g_1) \leq \maxangle(T_0g_0,G_1)  + \frac{4\|g_0(0)\|\left(2+\|g_0(0)\|/\tau\right)\beta}{\tau}
    \]
    Moreover,
    \[
    \|g_1(0)\| \leq \|g_0(0)\|(1+(\alpha+\beta)^2 (3+ 2\|g_0(0)\|/\tau) )
    \]
\end{Lemma}

We first need a supporting lemma that will show us that $g_{1}$ exists, and specifically, $g_{1}(0)$ exist.
\begin{Lemma} \label{lem:rough_w_bound}
Under the conditions of Lemma \ref{lem:shift_ang_general}, $g_{1}(0)$ exists and 
\[
\|Proj_{T_0 g_0} (o_1 + (0, g_{1}(0))_{G_1} - \tilde o_1) \| \leq \tau/2,
\]
where $o_1 = (0,0)_{G_0}$ is the origin and $\tilde o_1 = (0, g_0(0))_{G_0}$.
\end{Lemma}
\begin{proof}
We begin with defining the coordinate system $(\tilde o_1, G_1)$ with $\tilde o_1 = (0, g_0(0))_{G_0}$. Let $\wtilde g_1:(\wtilde o_1,G_1) \simeq \RR^d \to \RR^{D-d}$ be the function defined in Lemma \ref{lem:M_is_locally_a_fuinction_clean}, such that 
\[
\Gamma_{\wtilde g_1} =  \MM \cap \textrm{Cyl}(\wtilde o_1,c_{\pi/4}\tau, \tau/2)
\]
From Lemma \ref{lem:M_is_locally_a_fuinction_clean} we know that $\wtilde g_1$ is defined for any $x\in \RR^d$ such that $\norm{x} \leq c_{\pi/4}\tau$.
Now, we denote $x_o = Proj_{G_1}(o_1 - \tilde o_1)$, the projection of $o_1$ onto the affine space defined by $(\tilde o_1, G_1)$.
From the assumptions we know that $\norm {\tilde o_1 - o_1} \leq  c_{\pi/4} \tau$. 
Since $\maxangle (G_1,G_0) \leq \beta$ from Lemma~\ref{lem:angle_space_to_vec} we have that $\norm{x_o} \leq c_{\pi/4} \tau \sin \beta$. 
Thus, for any $\beta$ we have  $\|x_o\| \leq c_{\pi/4} \tau$, and $\wtilde{g}_1(x_o)$ is therefore defined (by Lemma \ref{lem:M_is_locally_a_fuinction_clean}). Since $\wtilde g_1$ identifies with $g_1$ up to some shift in the domain and target, it follows that $g_1(0)$ is well defined.

Next we bound  $ \|Proj_{T_0g_0} (o_1 + (0, g_{1}(0))_{G_1} - \tilde o_1) \| $.
Since $\wtilde{g}_1(0) = 0$, from Lemma \ref{lem:f_bound_circle_H_clean} and the triangle inequality for maximal angles between flats
we have that 
\[
\|\wtilde{g}_1(x_o)\| \leq \tau\cos(\alpha + \beta)-\sqrt{\tau^2-(\norm{x_o} + \tau\sin(\alpha+\beta))^2}
\]
Substituting $\|x_o\|$ in the right hand side we set 
\begin{align}
\|\wtilde{g}_1(x_o)\| &\leq \tau\cos(\alpha + \beta)-\sqrt{\tau^2-(\tau c_{\pi/4} \sin \beta + \tau\sin(\alpha+\beta))^2} \notag\\
&= \tau \left( \cos(\alpha + \beta)-\sqrt{1-(c_{\pi/4} \sin \beta + \sin(\alpha+\beta))^2} \right)    
\end{align}
Since  
\begin{multline}
 \|Proj_{T_0g_0} (o_1 + (0, g_{1}(0))_{G_1} - \tilde o_1) \| \leq \norm{o_1 + (0, g_{1}(0))_{G_1} - \tilde o_1)}   = \sqrt{\|\wtilde{g}_1(x_o)\|^2 + \|x_o\|^2} 
\\
= \tau \sqrt{c_{\pi/4}^2 \sin^2 \beta + \left( \cos(\alpha + \beta)-\sqrt{1-(c_{\pi/4} \sin \beta + \sin(\alpha+\beta))^2} \right)^2 }
\end{multline}
which, for small enough $\beta$ (i.e., smaller than some constant $\beta_c$ depending only on $c_{\pi/4}$) and fixed $\alpha$ is smaller than $0.5 \tau$.

\begin{figure}[ht]
    \centering
    \includegraphics[width=0.6\textwidth]{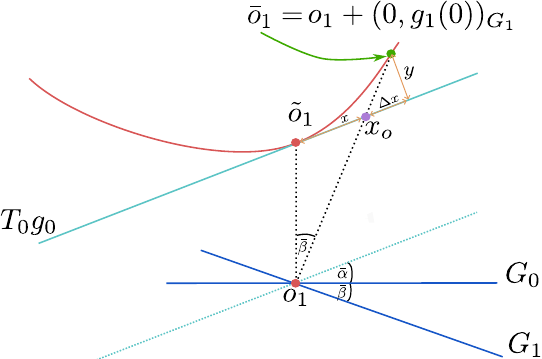}
    \caption{Illustration of an angle change of a coordinate system. We have $(o_1, G_0)$ as some coordinate system. Locally we look at $\MM$ (marked in solid red) as a graph of $g_0:(o_1, G_0)\to G_0\perp$. The point $\tilde o_1$ equals $g_0(0)$. Let $G_1$ be some rotated coordinate system and describe $\MM$ as a local graph of $g_1:(o_1, G_1)\to G_1^\perp$. }
    \label{fig:G0G1}
\end{figure}
\end{proof}

Next we prove Lemma \ref{lem:shift_ang_general}
\begin{proof}[proof of Lemma \ref{lem:shift_ang_general}] 
We first note that from Lemma \ref{lem:rough_w_bound} $(0,g_1(0))_{G_1}\in \MM$ exists.
Then, we denote by $o_1$ the origin, $\tilde o_1 = o_1 + (0, g_0(0))_{G_0}$ and $\bar o_1 = o_1 + (0, g_1(0))_{G_1}$ (see Figure \ref{fig:G0G1}).

\newcommand{\alphae}{\bar{\alpha}}
\newcommand{\betae}{\bar{\beta}}
Denote $\alphae = \angle(o_1,\tilde o_1, x_0) - \pi/2 $ and $\betae = \angle(\tilde o_1,o_1, \bar o_1)$, where $\angle(A,B,C)$ denotes the angle between the straight lines $\overline{AB}$ and $\overline{BC}$.
Note that $\alphae \leq \alpha$ and $\betae \leq \beta$.
Next, we bound the distance between $\tilde{o}_1$ and $x_o$, the intersection between $G_1^\perp$ and $T_0g_0$. Considering the triangle $o_1, \tilde{o}_1, x_o$, and using the sine theorem, we have that 
\begin{equation}\label{eq:x_eval}
x = \frac{\|g_0(0)\|\sin{\betae}}{\cos(\alphae+\betae)}.
\end{equation}
Next, we bound $\Delta x$, the distance between the point $x_o$ and the projection of $\bar{o}_1$ on $T_0g_0$. denote by $y$ the distance between $\bar o_1$ and $T_0g_0$ We have 
\begin{equation}\label{eq:Delta_x_as_y}
\Delta x = y \tan(\alphae+\betae).
\end{equation}
From Lemma B.6, we have that 
\begin{equation}\label{eq:y_bound}
y \leq \tau - \sqrt{\tau^2-(x+\Delta x)^2}
\end{equation}
substituting \eqref{eq:Delta_x_as_y} in \eqref{eq:y_bound}, we have
\[
y \leq \tau - \sqrt{\tau^2-(x+y \tan(\alphae+\betae))^2}.
\]
Simple algebraic manipulations give us
\[
y^2 (1+\tan^2(\alphae+\betae))-y(2\tau - 2x\tan(\alphae+\betae))+x^2 \geq 0.
\]
The above equation has two solutions.
\[
y_1 = \cos^2(\alphae+\betae) \left(\tau - x\tan(\alphae+\betae)-\sqrt{\tau^2-x^2-2\tau x\tan(\alphae+\betae)}\right),
\]
\[
y_2 = \cos^2(\alphae+\betae) \left(\tau - x\tan(\alphae+\betae)+\sqrt{\tau^2-x^2-2\tau x\tan(\alphae+\betae)}\right)
\]
and either $0\leq y \leq y_1$ or $y_2 \leq y$.

Note that since $x,\alphae,\betae$ are small, and using Remark \ref{rem:taylor_sqrt1-x2_clean}
\begin{align}
\sqrt{\tau^2-x^2-2\tau x\tan(\alphae+\betae)} &= \tau\sqrt{1-\frac{x^2}{\tau^2}-2\frac{1}{\tau} x\tan(\alphae+\betae)} 
\\&
\geq
\tau \left(1-\frac{x^2}{\tau^2}-2\frac{1}{\tau} x\tan(\alphae+\betae)\right) 
\\&
= \tau- x^2/\tau-2 x\tan(\alphae+\betae),
\end{align}
and thus,
\begin{align*}
y_1 &\leq \cos^2(\alphae+\betae) \left(\tau - x\tan(\alphae+\betae)-\tau+ x^2/\tau + 2 x \tan (\alphae+\betae)\right)\\
&= \cos^2(\alphae+\betae) \left( x^2/\tau +  x \tan (\alphae+\betae)\right),    
\end{align*}
\begin{align*}
y_2 &\geq \cos^2(\alphae+\betae) \left(\tau - x\tan(\alphae+\betae)+\tau- x^2/\tau-2 x\tan(\alphae+\betae)\right) \\
&= \cos^2(\alphae+\betae) \left(2\tau - 3x\tan(\alphae+\betae)- x^2/\tau\right)    
\end{align*}
for small $x,\alphae$, and $\betae$, and large $\tau$ $y_2>\tau$. Since $y$ cannot be larger than $\tau$ we get that  $0\leq y \leq y_1 = \cos^2(\alphae+\betae) \left( x^2/\tau +  x \tan (\alphae+\betae)\right)$.

Substituting the above inequality into \eqref{eq:Delta_x_as_y}, we get
\[
\Delta x \leq\cos(\alphae+\betae) \left( x^2/\tau +  x \tan (\alphae+\betae)\right) \sin(\alphae+\betae),
\]
and thus, using \eqref{eq:x_eval} and $\sin \betae \leq \sin(\alphae+\betae)$, we have
\begin{align}
x + \Delta x &\leq x \left(1+\cos(\alphae+\betae) \sin(\alphae+\betae)\left( x/\tau +   \tan (\alphae+\betae)\right) \right)
\\&=
x \left(1+\left( \|g_0(0)\|\sin{\betae}\sin(\alphae+\betae)/\tau + \sin^2 (\alphae+\betae)\right) \right)
\\&\leq
x \left(1+\left( \|g_0(0)\|\sin^2(\alphae+\betae)/\tau + \sin^2 (\alphae+\betae)\right) \right)
\\&=
x \left(1+\sin^2(\alphae+\betae)\left( \|g_0(0)\|/\tau + 1\right) \right)
,
\end{align}
substituting $x$ from \eqref{eq:x_eval}, we have
\begin{align}
\Delta x &\leq\cos(\alphae+\betae) \left( \left(\frac{\|g_0(0)\|\sin{\betae}}{\cos(\alphae+\betae)}\right)^2 \frac{1}{\tau}+  \frac{\|g_0(0)\|\sin{\betae}}{\cos(\alphae+\betae)} \tan (\alphae+\betae)\right) \sin(\alphae+\betae)
\\ &= \left( \frac{\left(\|g_0(0)\|\sin{\betae}\right)^2}{\tau\cos(\alphae+\betae)} +  \|g_0(0)\|\sin{\betae} \tan (\alphae+\betae)\right) \sin(\alphae+\betae)
\end{align}

Now, we are ready to bound the distance $\|\tilde{o}_1 - \bar{o}_1\|$
\begin{equation}\label{eq:baro-tilde0_2}
\begin{aligned}
\|\tilde{o}_1 - \bar{o}_1\|^2 &= \|\tilde{o}_1 - x_o\|^2 + \|x_o - \bar{o}_1\|^2 = (x+\Delta x)^2 + y^2
\\&=
x^2 \left(1+\sin^2(\alphae+\betae)\left( \|g_0(0)\|/\tau + 1\right) \right)^2
\\&~~~+
\left(\cos^2(\alphae+\betae) \left( x^2/\tau +  x \tan (\alphae+\betae)\right)\right)^2
\\&=
x^2 \left(1+\sin^2(\alphae+\betae)\left( \|g_0(0)\|/\tau + 1\right) \right)^2
\\&~~~+
x^2\left(\cos^2(\alphae+\betae) \left( x/\tau +  \tan (\alphae+\betae)\right)\right)^2
\\&=
x^2 \left(1+\sin^2(\alphae+\betae)\left( \|g_0(0)\|/\tau + 1\right) \right)^2
\\&~~~+
x^2\left(\cos(\alphae+\betae) \left( \|g_0(0)\|\sin\betae /\tau +  \sin (\alphae+\betae)\right)\right)^2
\\&\leq
x^2 \left(1+\sin^2(\alphae+\betae)\left( \|g_0(0)\|/\tau + 1\right) \right)^2
\\&~~~+
x^2\left(\cos(\alphae+\betae) \sin (\alphae+\betae)\left( \|g_0(0)\|/\tau+ 1 \right)\right)^2
\\&\leq
x^2 \left(1+\sin^2(\alphae+\betae)\left( \|g_0(0)\|/\tau + 1\right) \right)^2
\\&~~~+
x^2\left(\sin (\alphae+\betae)\left( \|g_0(0)\|/\tau+ 1 \right)\right)^2
\\&\leq
2x^2 \left(2+\sin(\alphae+\betae)\|g_0(0)\| /\tau\right)^2
\\&\leq
\frac{2(\|g_0(0)\|\sin{\betae})^2 \left(2+\sin(\alphae+\betae)\|g_0(0)\|/\tau \right)^2}{\cos(\alphae+\betae)}
\\&\leq
4(\|g_0(0)\|\sin{\betae})^2 \left(2+\sin(\alphae+\betae) \|g_0(0)\|/\tau  \right)^2
\end{aligned}
\end{equation}
where the last inequality is correct under the assumption that $\cos(\alphae+\betae)\geq 1/2$.

Finally, from Corollary 3 in \cite{boissonnat2017reach} we conclude that 
\[
\sin \frac{\maxangle(T_0g_0,T_0g_1)}{2} \leq \frac{\|\tilde o_1 - \bar o_1\|}{2\tau}.
\]
Since $x/2 \leq \sin x$, and using \eqref{eq:baro-tilde0_2}
\[
\maxangle(T_0g_0,T_0g_1) \leq \frac{2\|\tilde o_1 - \bar o_1\|}{\tau} \leq \frac{4\|g_0(0)\|\sin{\betae}\left(2+\sin(\alphae+\betae)\|g_0(0)\|/\tau  \right)}{\tau}. 
\]
we have  
\begin{align*}
\maxangle (T_0g_1, G_1) &\leq \maxangle (T_0g_0, G_1) + \maxangle(T_0g_0,T_0g_1)
\\&\leq 
\maxangle(T_0g_0,G_1)  + \frac{4\|g_0(0)\|\sin{\betae}\left(2+\sin(\alphae+\betae)\|g_0(0)\| /\tau \right)}{\tau}
\\&\leq 
\maxangle(T_0g_0,G_1)  + \frac{4\|g_0(0)\|\left(2+\|g_0(0)\|/\tau\right)\betae}{\tau}
\\&\leq 
\maxangle(T_0g_0,G_1)  + \frac{4\|g_0(0)\|\left(2+\|g_0(0)\|/\tau\right)\beta}{\tau}.
\end{align*}
This proves the first part of the lemma.

Next we bound $\|g_1(0)\|$. Note that 
$\|g_1(0)\| = \|x_0-o_1\| + \|\bar{o}_1 - x_0\|$.

\[
\|x_0-o_1\| = \frac{\|g_0(0)\| \cos\betae}{\cos(\alphae+\betae)} \leq \frac{\|g_0(0)\|}{\cos(\alphae+\betae)}.
\]
It can be shown using simple algebraic manipulations that,
\[
\|\bar{o}_1 - x_0\| \leq (\|g_0(0)\|^2/\tau+\|g_0(0)\|)  \frac{\sin^2(\alphae+\betae)}{\cos(\alphae+\betae)}
\]

\[
\|g_1(0)\| \leq  \frac{\|g_0(0)\|}{\cos(\alphae+\betae)} (~~~1 + (\|g_0(0)\|/\tau+1) \sin^2(\alphae+\betae)~~~)
\]
since $\cos(x) \geq 1- x^2/2$, $1/(1-x) \leq 1+2x$, $\sin(x) \leq x$, and $\alphae+\betae <1$,  we have 
\[
\|g_1(0)\| \leq  \frac{\|g_0(0)\|}{\cos(\alphae+\betae)} (~1 + (\|g_0(0)\|/\tau+1) \sin^2(\alphae+\betae)~)
\]
\begin{align}
\|g_1(0)\| &\leq  \frac{\|g_0(0)\|}{1-(\alphae+\betae)^2/2} (~1 + (\|g_0(0)\|/\tau+1) (\alphae+\betae)^2~) 
\\&
\leq \|g_0(0)\|(1+(\alphae+\betae)^2) (~1 + (\|g_0(0)\|/\tau+1) (\alphae+\betae)^2~) 
\\&
\leq \|g_0(0)\|(1+(\alphae+\betae)^2) (~1 + (\alphae+\betae)^2 + \|g_0(0)\|(\alphae+\betae)^2/\tau~) 
\\&
\leq \|g_0(0)\|(1+2(\alphae+\betae)^2 + (\alphae+\betae)^4 + \|g_0(0)\|(\alphae+\betae)^2/\tau+\|g_0(0)\|(\alphae+\betae)^4/\tau) 
\\&
\leq \|g_0(0)\|(1+3(\alphae+\betae)^2 + 2\|g_0(0)\|(\alphae+\betae)^2/\tau) 
\\&
\leq \|g_0(0)\|(1+(\alphae+\betae)^2 (3+ 2\|g_0(0)\|/\tau) ).
\end{align}
And thus we conclude the proof.
\end{proof}

\begin{Lemma}\label{lem:dist_given_angle}
    Let $G, H\in Gr(d, D)$ be two $d$-dimensional linear subspaces of $\RR^D$.
    Let $g: G \to \RR^{D-d}$ be such that the graph of $g$ is a manifold $\MM$ with a reach bounded by $\tau$ and let $h: H\to \RR^{D-d}$ denote the function over $H$ whose graph coincides with $\MM$ in some neighborhood \emph{(}see Figure \ref{fig:inverse_smot}\emph{)}. 
    Assume that $T_0 g$, the tanget of $g$ at $0$, is parallel to $G$, and that $\|g(0)\| \leq \sigma$. 
    If $\maxangle(H, T_0 h) \leq a $  is small enough, then 
    \[
    \|p_g-p_h\| \leq \frac{a}{\left(\frac{1}{2\sigma(2+\sigma/\tau)}-\frac{c}{\tau} \right)},
    \]
    where $p_g = (0, g(0))_{0, G}\in \MM$ and $p_h = (0, h(0))_{0, H}\in \MM$. 
\end{Lemma}
\begin{proof}
    \begin{figure}
        \centering
        \includegraphics[width=0.6\linewidth]{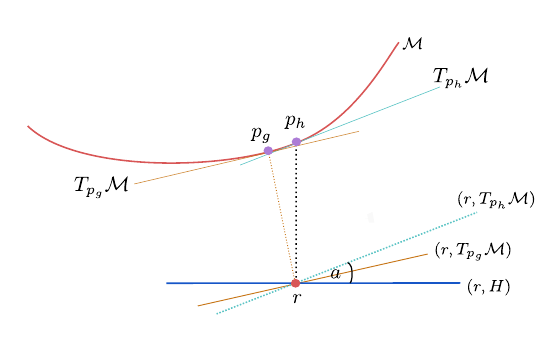}
        \caption{Supporting illustration for the proof of Lemma \ref{lem:dist_given_angle} where we bound the difference between $p_h$ and $p_g$, the evaluation of $r$ under the functions $h$ and $g$ correspondingly, given that the maximal angle between $T_{p_h}\MM$ and $T_{p_g}\MM$ is bounded by $a$.}
        \label{fig:inverse_smot}
    \end{figure}
     From the triangle inequality and \cite{boissonnat2017reach} we bound from above the angle between $T_0 g $ and $H$.
        \begin{equation}\label{eq:H_T_p angle bound}
        \maxangle(H, T_0 g) \leq \maxangle(H,T_0 h) + \maxangle(T_0 h,T_0 g) \leq a + c\|p_g-p_h\|/\tau
        .\end{equation}
    We now wish to give another bound tying between $\maxangle(H, T_0 g)$ and $\|p_g - p_h\|$ just from the other direction so that we could create an inequality limiting the values of $\|p_g- p_h\|$.
    Using the notation in Lemma \ref{lem:shift_ang_general}, we can follow the derivations with $\alpha = 0$, and with $\|g(0)\|\leq \sigma$, we get from \eqref{eq:baro-tilde0_2} that 
    \[
    \|p_g - p_h\|^2 \leq 4(\sigma \sin \maxangle(H, T_0 g))^2 (2+\sigma\sin \maxangle(H, T_0 g)/\tau)^2 \leq 4(\sigma \sin \maxangle(H, T_0 g))^2 (2+\sigma/\tau)^2.
    \]
    Thus,
    \[
    \|p_g -p_h\| \leq 2\sigma(2+\sigma/\tau) \sin \maxangle(H, T_0 g) \leq 2\sigma(2+\sigma) \maxangle(H, T_0 g),
    \]
    or,
    \[
    \maxangle(H, T_0 g)\geq \frac{\|p_g - p_h \|}{2\sigma(2+\sigma/\tau) }.
    \]

    Combining this with \eqref{eq:H_T_p angle bound}, we have
    \[
    \frac{\|p_g - p_h\|}{2\sigma(2+\sigma/\tau) } \leq a + c\|p_g - p_h\|/\tau
    \]
    or, 
    \[
    \|p_g - p_h\|\left(\frac{1}{2\sigma(2+\sigma/\tau)}-\frac{c}{\tau} \right) \leq a .
    \]
    Finally, we have,
    \[
    \|p_g - p_h\| \leq \frac{a}{\left(\frac{1}{2\sigma(2+\sigma/\tau)}-\frac{c}{\tau} \right)} .
    \]
\end{proof}

\end{appendices}
\bibliography{main}{}
\bibliographystyle{plain}

\end{document}